\documentclass{article}
\usepackage{setup2} 
\usepackage{booktabs}

\usepackage{float}
\usepackage{txfonts}
\usepackage{mathrsfs}
\restylefloat{table}
\usepackage{enumitem}

\usepackage{amsmath, amssymb, tikz, a4wide, pgfplots,hyperref} 
\pgfplotsset{compat=newest}

\usepackage{amsthm} 
\newtheorem{theorem}{Theorem}[section]  
\newtheorem{proposition}[theorem]{Proposition}
\newtheorem{lemma}[theorem]{Lemma}
\newtheorem{corollary}[theorem]{Corollary}
\theoremstyle{definition}
\newtheorem{definition}[theorem]{Definition}
\newtheorem{remark}[theorem]{Remark}

\newtheorem{example}[theorem]{Example}

\newenvironment{claim}
{\medskip\begin{em}\textbf{Claim:}}
{\end{em}}
\newenvironment{claimproof}
{\medskip\begin{em}Proof:\end{em}}
{$\square$}

\begin{document}

\title{Rigidity of symmetric frameworks with non-free group actions on the vertices}
\author{Alison La Porta\thanks{School of Mathematical Sciences, Lancaster University, UK, \texttt{a.laporta@lancaster.ac.uk} (corr. author)} \,
and Bernd Schulze\thanks{School of Mathematical Sciences, Lancaster University, UK, \texttt{b.schulze@lancaster.ac.uk}}}

\date{}
\maketitle

\begin{abstract}
For plane frameworks with reflection or rotational symmetries, where the  group action is not necessarily free on the vertex set, we introduce a phase-symmetric  orbit rigidity matrix for each irreducible representation of the group. 
We then use these generalised  orbit rigidity matrices to provide necessary  conditions for infinitesimal rigidity for  frameworks that are symmetric with a cyclic group that acts freely or non-freely on the vertices. Moreover, for  the reflection, the half-turn, and the three-fold rotational group in the plane, we establish complete combinatorial characterisations of symmetry-generic infinitesimally rigid frameworks. This extends  well-known characterisations for these groups to the case when the group action is not necessarily free on the vertices. 
The  presence of vertices that are fixed by non-trivial group elements requires the introduction of generalised versions of group-labelled quotient graphs and leads to  more refined types of combinatorial sparsity counts for characterising symmetry-generic infinitesimal rigidity.

\medskip


\end{abstract}

\section{Introduction}


Largely motivated by problems from the applied sciences such as engineering, robotics, biophysics, materials science, and computer-aided design, where structures are often symmetric, there has recently been significant interest in studying the impact of symmetry on the infinitesimal rigidity of (bar-joint) frameworks and related geometric constraint systems. We refer the reader to \cite{connellyguest} for an introduction to the theory and to \cite{bernd2017sym,KTHandbookSch} for a summary of recent results.

One line of research in this area, which  has seen a lot of progress lately, is to study when a symmetry-generic framework (i.e. a framework that is as generic as possible under the given symmetry constraints) is ``forced-symmetric rigid", in the sense that it has no non-trivial motion that breaks the original symmetry of the framework. See \cite{sw2011,sch10,jzt2016,bernstein,mtforced,nscyl,tanigawamatroids,dewar,cnswpair,dgl,kitnixsch,clki} for some key results on this topic for finite symmetric frameworks. 
(For an overview of the corresponding results on infinite periodic or crystallographic frameworks, see \cite{KTHandbookSch}.)

Another active line of research deals with the more difficult problem to determine when a symmetry-generic framework has no non-trivial motion at all. Since in this case the original symmetry of the framework is allowed to be destroyed by a motion, the framework is sometimes said to be ``incidentally symmetric" -- a term  coined by Robert Connelly. This problem was first tackled for the special class of isostatic (i.e. minimally infinitesimally rigid) frameworks in the plane \cite{SLaman,SLaman2}. While there is ongoing work on this case (see e.g. \cite{nsw2024}), the more general question of when a symmetry-generic framework is infinitesimally rigid -- rather than just isostatic -- is more challenging, as not every symmetric framework contains an isostatic sub-framework with the same symmetry on the same vertex set.
Combinatorial characterisations of incidentally symmetric infinitesimally rigid frameworks have so far only been obtained for special classes of cyclic groups in the Euclidean plane \cite{bt2015,Ikeshita} (see also \cite{RIandST16,KCandST20}), in some non-Euclidean normed planes   \cite{kitsonschulze} and for classes of body-bar and body-hinge frameworks with $\mathbb{Z}_2\times \cdots \times \mathbb{Z}_2$ symmetry in Euclidean $d$-space \cite{schtan_bodies}. 

A central result in the theory is that the rigidity matrix of a symmetric framework can be transformed into a block-decomposed form, where each block corresponds to an irreducible representation of the group \cite{kangwai2000,sch10a,owen10,kkm}. This breaks up the infinitesimal rigidity analysis  into independent sub-problems, one for each block matrix. For forced symmetric rigidity, one only focuses on the  block matrix corresponding to the trivial irreducible representation \cite{sw2011}.
The problem of characterising incidentally symmetric  infinitesimal rigidity may be solved by characterising the maximum rank of each block matrix for a symmetry-generic framework. By setting up phase-symmetric orbit rigidity matrices that are equivalent to, but are easier to work with than the block matrices in the block-decomposed rigidity matrix, this may be achieved via a Henneberg-type recursive construction for the corresponding group-labelled quotient graphs. This has been the most common approach for solving problems on the rigidity of  incidentally symmetric frameworks (see e.g. \cite{bt2015,Ikeshita}).

However, so far, all of the work on forced and incidentally  symmetric infinitesimal rigidity (except for the results for symmetric \emph{isostatic} frameworks mentioned above, and for the result on forced-symmetric rotationally-symmetric frameworks given in \cite[Section 4.3]{mtforced}) has made the assumption that the symmetry group acts freely on the vertex set of the framework. This simplifies the definition of the group-labelled quotient graph, the structure of the corresponding orbit rigidity matrices, and the types of sparsity counts on the group-labelled quotient graphs that appear in characterisations of symmetry-generic infinitesimal rigidity. In this paper we address this gap and start to extend the theory to symmetric frameworks where vertices may be fixed by non-trivial group elements.

Closing this gap is not just of pure mathematical interest, but has  important motivations arising from practical applications. For example, tools and results from symmetric rigidity theory have recently been applied to the design and analysis of material-efficient long-span engineering structures such as gridshell roofs and cable nets \cite{smmb,mmsb,schmil}. For this application, it is crucial to understand the infinitesimal (or equivalently static) rigidity properties of the vertical 2D projections of the 3D structures, known as form diagrams. For structural optimisation reasons,  these form diagrams (which are planar bar-joint frameworks that are often symmetric) frequently have vertices that are fixed by reflections and rotations. Other application areas for which it is important to understand the infinitesimal rigidity of symmetric frameworks where the group acts non-freely on the vertex set include the design of distributed control laws for multi-robot formations \cite{zts,stz} and the analysis of geometric constraint systems appearing in computer-aided design \cite{foguow}.

In this paper, we first extend the definition of a group-labelled quotient graph (also known as a quotient ``gain graph") for all cyclic groups to include the possibility of having vertices that are fixed by non-trivial group elements (Section~\ref{sec:sym}). This allows us to define generalised phase-symmetric orbit rigidity matrices for these groups (Section~\ref{sec:orbitmat}). In Section~\ref{sec:nec} we then use these matrices to establish necessary conditions for frameworks with reflection or rotational symmetry in the plane to be infinitesimally rigid. These conditions are given in terms of sparsity counts for the corresponding quotient gain graphs.
In Sections~\ref{sec:red} and~\ref{sec:ext} we then focus on the reflection group, half-turn group and rotational group of order 3 and show that the conditions given in Section~\ref{sec:nec} are also sufficient for a symmetry-generic framework in the plane  to be infinitesimally rigid. This extends the corresponding results  for the case when the group acts freely on the vertex set of the graph given in  \cite{bt2015}. The proofs of our main results in Section~\ref{sec:ext} are based on a Henneberg-type inductive construction on the relevant quotient gain graphs, using the graph operations described in Section~\ref{sec:red}. Finally, in Section~\ref{sec:fut} we describe some avenues of future work.

In a second paper \cite{lasch}, we establish the corresponding combinatorial characterisations for all cyclic (rotational) groups of odd order up to 1000 and for the cyclic groups of order 4 and 6, extending the results of Clinch, Ikeshita and Tanigawa  \cite{Ikeshita,KCandST20} to the case where there is a vertex that is fixed by the rotation.  This is done in a separate paper, since the sparsity counts and hence the proofs become significantly more complex for these groups. In that paper we will also show that for even order groups of order at least 8, the standard sparsity counts are not sufficient.

\section{Infinitesimal rigidity of frameworks}

We start by reviewing some basic terms and results from rigidity theory. See \cite{connellyguest,WW}, for example, for further details.

A \textit{(bar-joint) framework} in $\mathbb{R}^d$ is a pair $\left(G,p\right)$ where $G$ is a finite simple graph and $p:V\left(G\right)\rightarrow\mathbb{R}^d$ is a map such that $p(u)\neq p(v)$ for all $u\neq v\in V(G)$. We  also refer to $(G,p)$ as a \textit{realisation} of the graph $G$, the \textit{underlying graph} of $(G,p)$, and to $p$ as a \textit{configuration}. We will sometimes use $p_v$ to denote $p(v)$. Unless explicitly stated otherwise, we will assume throughout the paper that $p(V(G))$ affinely spans $\mathbb{R}^d$. 

An \textit{infinitesimal motion} of a $d$-dimensional framework $(G,p)$ is a function $m:V(G)\rightarrow\mathbb{R}^d$ such that for all $\{u,v\}\in E(G),$ 
\begin{equation}
\label{inf. motion}
     (p_u-p_v)^T(m(u)-m(v))=0.
\end{equation}

An infinitesimal motion $m:V(G)\rightarrow\mathbb{R}^d$ of a framework $(G,p)$ is called \textit{trivial} if there is a skew-symmetric matrix $M\in M_d(\mathbb{R})$ and a $d$-dimensional vector $t$ such that $m(u)=Mp_u+t$ for all $u\in V(G)$. Otherwise, it is called an \textit{infinitesimal flex}. 
We say a framework $(G,p)$ in $\mathbb{R}^d$ is \textit{infinitesimally rigid} if all of its infinitesimal motions are trivial. Otherwise, we say $(G,p)$ is \textit{infinitesimally flexible}. It is a well known fact that if the points of a framework affinely span all of $\mathbb{R}^d$, then the space of infinitesimal motions has dimension $\genfrac(){0pt}{2}{d+1}{2}$.

It is sometimes useful to see a motion $m:V(G)\rightarrow\mathbb{R}^d$ as a column vector with $d|V(G)|$ entries.

In order to study the infinitesimal rigidity of a framework, we analyse its rigidity matrix. Given a $d$-dimensional framework $(G,p)$, the \textit{rigidity matrix} of $(G,p)$, denoted by $R(G,p)$, is a $|E(G)|\times d|V(G)|$ matrix, whose rows correspond to the edges of $G$ and where each vertex is represented by $d$ columns. Given an edge $e=\{u,v\}\in E(G)$, the row corresponding to $e$ in $R(G,p)$ is 

\begin{center}
    $\begin{pmatrix}
    0 & \dots & 0 & [p_u-p_v]^T & 0 & \dots & 0 & [p_v-p_u]^T & 0 & \dots & 0
    \end{pmatrix}$,
\end{center}

where the $d$-dimensional row vector $[p_u-p_v]^T$ is in the $d$ columns corresponding to $u$, $[p_v-p_u]^T$ is in the columns corresponding to $v$, and there are zeros everywhere else.

By Equation (\ref{inf. motion}), the kernel of $R(G,p)$ is the 
space of infinitesimal motions of $(G,p)$. From this follows the well-know fact that $(G,p)$ is infinitesimally rigid if and only if $\textrm{rank}\left(R(G,p)\right)=d|V(G)|-\frac{d(d+1)}{2}$ or $G=K_{|V(G)|}$ and the points $p_i$ (for $i=1,\dots,|V(G)|$) are affinely independent. 

A \textit{self-stress} of a framework $(G,p)$ is a map $\omega:E(G)\rightarrow\mathbb{R}$ such that for all $u\in V(G)$, it satisfies $\sum_{v:\{u,v\}\in E(G)}\omega(\{u,v\})(p_u-p_v)=0$. Notice that $\omega$ is a self-stress if and only if $R(G,p)^T\omega=0$. So,  $(G,p)$ has a non-zero self-stress if and only if there is a non-trivial row dependency in $R(G,p)$.

Given a graph $G$, we say a configuration $p:V(G)\rightarrow\mathbb{R}^d$ is \textit{generic} if $R(G,p)$ has maximum rank among all configurations of $G$ in $\mathbb{R}^d$.  If $p$ is generic then we say $\left(G,p\right)$ is a \textit{generic framework}. The set of all generic configurations of $G$ is a dense, open subset of $\mathbb{R}^{d|V\left(G\right)|}$.

\section{Infinitesimal rigidity of symmetric frameworks}\label{sec:sym}

In this section, we will go through what it means for a graph and a framework to be symmetric and we will introduce some of the tools we will be using to study the rigidity of such frameworks. Throughout the paper, we let $\Gamma\simeq\mathbb{Z}_k$ for some $k\in\mathbb{N}.$ We will later restrict $k$ to be 2 or 3. 

\subsection{Symmetric graphs}
The set of all automorphisms of a graph $G$ forms a group under composition called the \textit{automorphism group} of $G$ and is denoted $\text{Aut}(G)$. We say $G$ is \textit{$\mathbb{Z}_k$-symmetric} if  $\text{Aut}(G)$ contains a subgroup isomorphic to $\mathbb{Z}_k=\{0,1,\dots,k-1\}$. For notational convenience, we will often identify $\mathbb{Z}_k$ with the multiplicative group $\Gamma=\left<\gamma\right>$ via the isomorphism defined by $1\mapsto\gamma$. (Later on, geometrically, $\gamma$ will correspond to the rotation by $2\pi/k$ about the origin in the plane.) The terms $\mathbb{Z}_k$-symmetric graph and $\Gamma$-symmetric graph are used interchangeably. Throughout the paper, the order  of $\Gamma$ is $|\Gamma|=k$.

Given $\gamma\in\Gamma$, we say $\gamma$ \textit{fixes $v\in V(G)$} if $\gamma(v)=v$, and $\gamma$ \textit{fixes a subset $U\subseteq V(G)$} if it fixes all $u\in U$. Similarly, $\gamma$ \textit{fixes $e\in E(G)$ }if $\gamma(e)=e$, and it \textit{fixes a subset $F\subseteq E(G)$} if it fixes all $f\in F$. We define the stabiliser  of a vertex $v\in V(G)$ to be $S_{\Gamma}(v)=\{\gamma\in\Gamma:\gamma\text{ fixes }v\}$, and the stabiliser of an edge $e\in E(G)$ to be $S_{\Gamma}(e)=\{\gamma\in \Gamma:\gamma\text{ fixes }e\}$.

For simplicity, we will assume throughout the paper that any vertex fixed by a non-trivial element of $\Gamma$ is fixed by all elements of $\Gamma$. This is justified, as our main focus is the infinitesimal rigidity of symmetric frameworks in the plane, and the definition of a symmetric framework (see Section~\ref{sec:symfwks}) will imply, for $d=2$, that any vertex that is fixed by a non-trivial element of $\Gamma$, $|\Gamma|\geq3$, will have to be placed at the origin, and conversely, since framework configurations are injective, any vertex at the origin must be fixed by every element of the group. We define  $V_0(G):=\{v\in V(G):S_{\Gamma}(v)=\Gamma\}$ and $\overline{V(G)}:=\{v\in V(G):S_{\Gamma}(v)=\{\textrm{id}\}\}$. We also use the notation $\overline{E(G)}:=\{e\in E(G):S_{\Gamma}(e)=\{\textrm{id}\}\}$. We will refer to the elements of $V_0(G),\overline{V(G)}$ and $\overline{E(G)}$ as the \emph{fixed vertices}, \emph{free vertices} and \emph{free edges} of  $G$, respectively.

 \subsection{Gain graphs}\label{sec:gg}

Let $\Tilde{G}$ be a $\Gamma$-symmetric graph. We denote the orbit of a vertex $v$ (respectively an edge $e$) of $\tilde G$ by $\Gamma v$ (respectively $\Gamma e$). Thus, $\Gamma v=\{\gamma v|\, \gamma\in \Gamma\}$ and $\Gamma e=\{\gamma e|\, \gamma\in \Gamma\}$. The collection of all vertex orbits and edge orbits of $\tilde G$ is denoted by $V$ and $E$, respectively. The \emph{quotient graph} $\tilde G /\Gamma$ of $\tilde G$ is a multigraph $G$ with vertex set $V(G)=V$, edge set $E(G)=E$ and  incidence relation satisfying $\Gamma e=\Gamma u\Gamma v$ if some (equivalently every) edge in $\Gamma e$ is incident with  a vertex in $\Gamma u$ and a vertex in $\Gamma v$.
Notice that the partitioning of $V(\tilde{G})$  given in the previous section induces a partition of $V(G)$ into the sets $V_0(G):=\{\Gamma v\in V(G):|\Gamma v|=1\}$ and $\overline{V(G)}=\{\Gamma v\in V(G):|\Gamma v|=|\Gamma|\}$ of \emph{fixed} and \emph{free vertices} of $G$, respectively.

Let $\tilde G$ be a $\Gamma$-symmetric graph with quotient graph $G$. For each vertex orbit $\Gamma v$ we fix a representative vertex $v^{\star}\in\Gamma v$.
We also fix an orientation on the edges of the quotient graph $G$.
For each directed edge $\Gamma e = (\Gamma u, \Gamma v)$ in the directed quotient graph, we assign the following labelling (or ``gain"):
\begin{itemize}
    \item If $\Gamma u,\Gamma v\in \overline{V(G)}$, then there exists a unique $\gamma\in \Gamma$, referred to as the \emph{gain} on $\Gamma e$, such that $\{u^{\star},\gamma v^{\star}\}\in \Gamma e$.
    \item If  at least one of $\Gamma{u},\Gamma{v}$ is fixed, say $\Gamma u\in V_0(G)$, then $\Gamma e=\{\{u^{\star},\gamma v^{\star}\}|\, \gamma\in \Gamma\}$. The size of $\Gamma e$ depends on the size of $\Gamma v$. If $|\Gamma e|=|\Gamma|$, then we define the \emph{gain} on $\Gamma e$ to be any $\gamma\in\Gamma$. Otherwise, $|\Gamma e|=1$, and we define the \emph{gain} on $\Gamma e$ to be $\textrm{id}$.
\end{itemize}

This gain assignment $\psi: E(G) \to \Gamma$ is well-defined and the pair $(G,\psi)$ is called the \emph{(quotient) $\Gamma$-gain graph} of $\tilde G$. Moreover, in a slight abuse of topological terminology, $\tilde G$ is called the \emph{covering graph} (or \emph{lifting}) of $(G,\psi)$.

In each of the cases above, we could re-direct $\Gamma e$ from $\Gamma v$ to $\Gamma u$ and re-label it with the group inverse of the original label chosen. Up to this operation, and up to the choice of representatives, and of gains for edges incident with a fixed and a free vertex, this process gives a unique quotient $\Gamma$-gain graph of $\tilde{G}$. We call two $\Gamma$-gain graphs \textit{equivalent} if they are obtained from the same $\Gamma$-symmetric graph by applying this process.

In general, this process gives rise to a  class of group-labelled, directed multigraphs, called $\Gamma$-gain graphs.

\begin{definition} \label{def:gaingraph}
  A \emph{$\Gamma$-gain graph} is a pair $(G,\psi),$ where $G$ is a directed multigraph and $\psi:E(G)\rightarrow\Gamma$ is a function that assigns a label to each edge such that, for some partition $V(G)=V_0(G)\,\dot\cup\,\overline{V(G)}$, where no vertex in $V_0(G)$ has a loop or is incident to parallel edges, the following conditions are satisfied:

\begin{enumerate}
\item for all $e\in E(G)$ with both endpoints in $V_0(G),$ $\psi(e)=\textrm{id}$; 
\item if $e,f\in E(G)$ 
are parallel and have the same direction, then $\psi(e)\neq\psi(f)$. If they are parallel and have opposite directions, then $\psi(e)\neq\psi(f)^{-1}$;
\item if $e\in E(G)$ is a loop, then $\psi(e)\neq\textrm{id}$.
\end{enumerate} 
\end{definition}

We call $\psi$ the \textit{gain function} of $(G,\psi)$. The elements of $V_0(G)$ and $\overline{V(G)}$ are called, respectively, the \emph{fixed} and \emph{free vertices} of $(G,\psi)$. 
   Throughout this paper, we assume that a loop at a vertex $v$ adds $2$ to the degree of $v$.

Let $(G,\psi)$ be a $\Gamma$-gain graph. Using a generalisation of the process described in Section 3.2 of \cite{bt2015}, we can obtain a unique $\Gamma$-symmetric graph $\tilde{G}$, which we call the \textit{covering} (or \textit{lifting}) of $(G,\psi)$, as follows. 

For each $v\in V_0(G)$, $V(\tilde{G})$ contains $v$; for each $v\in\overline{V(G)}$, $V(\tilde{G})$ contains $\{\gamma v|\, \gamma\in\Gamma\}$. For each edge $(u,v)\in E(G)$ with $u\in V_0(G)$, $E(\tilde{G})$ contains $\{u,v\}$ if $v\in V_0(G)$, and $E(\tilde{G})$ contains $\{\{u,\gamma v\}|\, \gamma\in\Gamma\}$ if $v\in \overline{V(G)}$;
for each $(u,v)\in E(G)$ with $u,v\in\overline{V(G)}$ and label $\gamma$, $E(\tilde{G})$ contains $\Gamma\{u,\gamma v\}$. This process gives a unique lifting, and is inverse to the one shown at the beginning of the section. Thus, each $\Gamma$-gain graph  uniquely determines a simple $\Gamma$-symmetric graph (up to equivalence). 

When drawing a $\Gamma$-gain graph $(G,\psi)$ it is important to distinguish between the fixed and free vertices of $(G,\psi)$. We will be doing so by representing the elements of $V_0(G)$ and $\overline{V(G)}$ by black and white circles, respectively. In Figure~\ref{first gain graph example}, we consider the cyclic group $\Gamma=\{\text{id},\gamma\}$ of order 2, and we give an example of a $\Gamma$-gain graph and its lifting.

\begin{figure}[htp]
    \centering
    \begin{tikzpicture}
  [scale=.7,auto=left]
  
  \filldraw (0,0) circle (2pt);
  \filldraw (2.27,1.37) circle (2pt);
  \filldraw (-2.27,-1.37) circle (2pt);
  \filldraw (0.95,-1.35) circle (2pt);
  \filldraw (-0.95,1.35) circle (2pt);
  \draw (0,0) -- (2.27,1.37);
  \draw (0,0) -- (-2.27,-1.37);
  \draw (-2.27,-1.37) -- (-0.95,1.35);
  \draw (2.27,1.37) -- (0.95,-1.35);
  \draw (-2.27,-1.37) -- (0.95,-1.35);
  \draw (2.27,1.37) -- (-0.98,1.35);
  \node[below] at (-2.27,-1.37) {$\gamma v_1$};
  \node[above] at (2.27,1.37) {$v_1$};
  \node[above] at (-0.95,1.35) {$v_2$};
  \node[below] at (0.95,-1.35) {$\gamma v_2$};
  \node[above] at (0,0.1) {$v_0$};  

  \draw[fill=black] (-6.5,0) circle (0.15cm);
  \draw (-5,1.37) circle (0.15cm);
  \draw (-7.5,1.37) circle (0.15cm);
  \draw[->] (-6.5,0) -- (-5.1,1.27); 
  \draw[->] (-5.15,1.37) -- (-7.36,1.37); 
  \draw[->] (-7.35,1.37) .. controls (-6.25,1) .. (-5.15,1.315);
  \node[above] at (-5,1.41) {$v_1$};
  \node[above] at (-7.5,1.41) {$v_2$};
  \node[left] at (-6.6,0) {$v_0$};
  \node[right] at (-5.9,0.5) {$\textrm{id}$};
  \node[above] at (-6.4,1.4) {$\textrm{id}$};
  \node[below] at (-6.4,1.1) {$\gamma$};

\end{tikzpicture}
    \caption{A $\Gamma$-gain graph and its lifting, where $\Gamma=\{\text{id},\gamma\}$ has order 2.}
    \label{first gain graph example}
\end{figure}
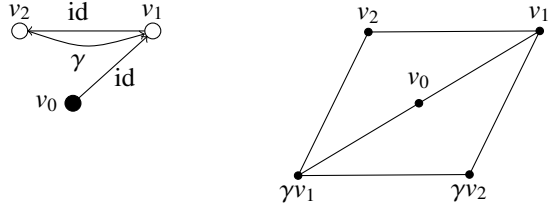

From now on, we will adopt the notation $\tilde{G}$ to denote the covering of a $\Gamma$-gain graph $(G,\psi)$. As before, given $v\in V(G)$, we use $v^{\star}$ to denote the representative of $v$ in $V(G)$. If $\psi$ is clear from the context then we often write $G$ for $(G,\psi)$. 

\subsection{Gain sparsity of a symmetric graph}\label{sec:gaincounts}
In this section we will introduce the criteria we will use to characterise infinitesimally rigid $\Gamma$-symmetric graphs. We start with the notions of balancedness and near-balancedness, which can also be found (for free group actions) in Section 4.1 of \cite{Ikeshita} and Section 1 of \cite{RIandST16} . The notion of balancedness can also be found in Section 2.2 of \cite{jzt2016}.

Let $(G,\psi)$ be a $\Gamma$-gain graph and let $W$ be a walk in $(G,\psi)$ of the form $v_1e_1v_2e_2\dots e_tv_{t+1}$. We say that the \textit{gain of $W$} is $\psi(W)=\prod_{i=1}^t\psi(e_i)^{\textrm{sign}(e_i)}$, where $\textrm{sign}(e_i)=1$ if $e_i$ is directed from $v_i$ to $v_{i+1}$, and $\textrm{sign}(e_i)=-1$ otherwise. Let $v\in \overline{V(G)}$. If $G$ is connected, we use the notation $\left<E(G)\right>_{\psi,v}$, or simply $\left<G\right>_{\psi,v}$, to denote the group generated by $\{\psi(W):W\text{ is a closed walk in }G\text{ starting at }v \text{ and with no fixed vertex}\}$. In \cite{jzt2016}, it was shown that, given two free vertices $u,v\in V(G)$, $\left<G\right>_{\psi,v}$ and $\left<G\right>_{\psi,u}$ are conjugate. Since the groups we work with are cyclic, the two subgroups are actually the same. When clear, we omit $\psi,v$ and write $\left<G\right>$. 

An edge set $E$ is \textit{balanced} if the edge set $E'$ obtained from $E$ by removing all fixed vertices and their incident edges, either has no cycle or every cycle in $E'$ has gain $\textrm{id}$. Otherwise, we say $E$ is \textit{unbalanced}.
We say $\left(G,\psi\right)$ is \textit{balanced} (respectively \textit{unbalanced}) if $E(G)$ is balanced (respectively unbalanced).
For $0\leq m\leq2,0\leq l\leq 3$, we say a $\Gamma$-gain graph $(G,\psi)$ is $(2,m,l)$-sparse if $|E(H)|\leq 2|\overline{V(H)}|+m|V_0(H)|-l$ for all subgraphs $H$ of $G$, with $E(H)\neq\emptyset$ in the case of $l=3$, and we say it is $(2,m,l)$-tight if it is $(2,m,l)$-sparse and $|E(G)|=2|\overline{V(G)}|+m|V_0(G)|-l$. We abbreviate $(2,2,l)$-sparse (equivalently, $(2,2,l)$-tight) to $(2,l)$-sparse (equivalently, $(2,l)$-tight).

\begin{remark}\label{rem:switching}
  The covering $\tilde G$ of a balanced $\Gamma$-gain graph $(G,\psi)$ will have a non-zero self-stress if $G$ is not $(2,3)$-sparse, regardless of the size of $V_0(G)$. In fact, if $|V_0(G)|\leq 1$, then clearly $\tilde G$ will have at least $|\Gamma|$  non-zero self-stresses with mutually disjoint support. See e.g. Figures~\ref{free-balanced example} (a) and (b).
    
    Note also that the choice of vertex orbit representatives for constructing the quotient $\Gamma$-gain graph has no effect on the gain of any edges incident to a fixed vertex, since those edges can have any gain. So, given a cycle $C$ of $(G,\psi)$ containing a fixed and a free vertex, for each $\gamma\in\Gamma$, there is a choice of representatives such that $\psi(C)=\gamma$ (see cycles $v_1v_3v_0$ and $v_4v_3v_0$ in Figures~\ref{free-balanced example} (b) and (c), respectively). 
\end{remark}

\begin{figure}[htp]
    \centering
    \begin{tikzpicture}[scale = 0.9]
    \filldraw(-4,0) circle (2pt);
    \filldraw(-3,2) circle (2pt);
    \filldraw(-5,1) circle (2pt);
    \filldraw(-7,0) circle (2pt);
    \filldraw(-6,2) circle (2pt);
    \filldraw(-7.5,1) circle (2pt);
    \filldraw(-2.5,1) circle (2pt);
  \draw (-4,0) -- (-5,1) -- (-6,2) -- (-7,0) -- (-5,1) -- (-3,2) -- (-4,0);
  \draw(-7.5,1) -- (-2.5,1) -- (-3,2);
  \draw(-2.5,1) -- (-4,0);
  \draw(-6,2) -- (-7.5,1) -- (-7,0);
  \node[above] at (-5,1.1){$v_0^{\star}$};
  \node[above] at (-3,2){$v_6$};
  \node[above] at (-6,2){$v_1^{\star}$};
  \node[below] at (-7,0){$v_3^{\star}$};
  \node[below] at (-4,0){$v_4$};
  \node[left] at (-7.5,1){$v_2^{\star}$};
  \node[right] at (-2.5,1){$v_5$};
  \node[below] at (-5,-1){(a)};
  \node[below] at (1,-1){(b)};
  \node[below] at (6,-1){(c)};
    
        \draw(0.5,0) circle (0.15cm);
        \draw(1.5,2) circle (0.15cm);
        \draw[fill=black](2.5,1) circle (0.15cm);
        \draw(0,1) circle (0.15cm);
  \draw[->] (0.55,0.15) -- (1.5,1.85); 
  \draw[->] (1.6,1.9) -- (2.4,1.1);
  \draw[->] (0.6,0.1) -- (2.4,0.9);
  \draw[->] (1.4,1.9) -- (0.1,1.1);
  \draw[->] (0.1,0.9) -- (0.4,0.1);
  \draw[->] (2.35,1) -- (0.15,1);
  \node at (1.5,0.22) {$\textrm{id}$};
  \node at (2.2,1.65) {$\textrm{id}$};
  \node at (0.85,1.25) {$\textrm{id}$};
  \node at (1.7,1.25) {$\textrm{id}$};
  \node at (0.6,1.75) {$\textrm{id}$};
  \node at (0,0.4) {$\textrm{id}$};
  \node[right] at (2.6,1) {$v_0$};
  \node[above] at (1.5,2.15) {$v_1$};
  \node[below] at (0.5,-0.15) {$v_3$};
  \node[left] at (-0.15,1) {$v_2$};

  \draw(5,0) circle (0.15cm);
        \draw(6,2) circle (0.15cm);
        \draw[fill=black](7,1) circle (0.15cm);
        \draw(4.5,1) circle (0.15cm);
  \draw[->] (5.05,0.15) -- (6,1.85); 
  \draw[->] (6.1,1.9) -- (6.9,1.1);
  \draw[->] (5.1,0.1) -- (6.9,0.9);
  \draw[->] (5.9,1.9) -- (4.6,1.1);
  \draw[->] (4.6,0.9) -- (4.9,0.1);
  \draw[->] (6.85,1) -- (4.65,1);
  \node at (6,0.22) {$\textrm{id}$};
  \node at (6.7,1.65) {$\textrm{id}$};
  \node at (5.45,1.25) {$\gamma$};
  \node at (6.2,1.25) {$\textrm{id}$};
  \node at (5.1,1.75) {$\gamma$};
  \node at (4.5,0.4) {$\textrm{id}$};
  \node[right] at (7.1,1) {$v_0$};
  \node[above] at (6,2.15) {$v_4$};
  \node[below] at (5,-0.15) {$v_3$};
  \node[left] at (4.35,1) {$v_2$};

    \end{tikzpicture}
    \caption{Equivalent balanced $\mathbb{Z}_2$-gain graphs (b,c) of a $\mathbb{Z}_2$-symmetric graph. (a) shows the representatives chosen in (b). In (c), we choose the same representatives, except for the representative for $\mathbb{Z}_2v_1$, which is now $v_4$.}
    \label{free-balanced example}
\end{figure}
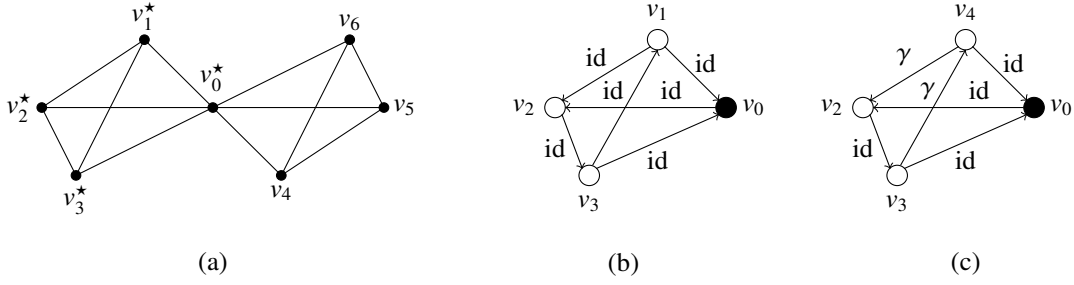

\begin{lemma} [Lemma 2.5(i) in \cite{jzt2016}]
\label{lemma: union of balanced is balanced}
Let $H_1,H_2$ be connected balanced $\Gamma$-gain graphs such that $V_0(H_1)=V_0(H_2)=\emptyset$. If $H_1\cap H_2$ is connected, then $H_1\cup H_2$ is balanced.
\end{lemma}

Since balancedness of a gain graph solely depends on its edges joining free vertices, it is easy to see that the following holds.

\begin{lemma}
    \label{lemma: union of balanced and free-balanced is free-balanced}
    Let $H_1,H_2$ be connected $\Gamma$-gain graphs. Assume that the graph obtained from $H_1\cap H_2$ by removing its fixed vertices is connected. If $H_1,H_2$ are balanced, then so is $H_1\cup H_2$. 
\end{lemma}

Suppose $F$ is a connected subset of $E(G)$ with $V_0(F)=\emptyset$. We say $F$ is \textit{near-balanced} if it is unbalanced, and there exist a vertex $v$ of $G[F]$, called the \textit{base vertex} of $G[F]$, and $\gamma\in\Gamma$ such that, for all closed walks $W$ in $F$ starting from $v$ and not containing $v$ as an internal vertex, $\psi(W)\in\{\textrm{id},\gamma,\gamma^{-1}\}$. A subgraph $H$ of $(G,\psi)$ with no fixed vertex is said to be \textit{near-balanced} if $E(H)$ is near-balanced. Figure~\ref{near-balanced image} shows a near-balanced $\mathbb{Z}_5$-gain graph and its covering. If $\left<H\right>\simeq\mathbb{Z}_2$ or $\left<H\right>\simeq\mathbb{Z}_3$, then $H$ is always near-balanced. Hence, we say $H$ (equivalently, $E(H)$) is \textit{proper near-balanced} if it is near-balanced and $\left<H\right>\not\simeq\mathbb{Z}_2,\mathbb{Z}_3$. 

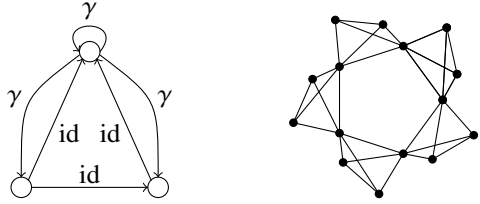
\begin{figure}[H]
    \centering
        \begin{tikzpicture}[scale = 0.9]
        \draw(0,0) circle (0.15cm);
        \draw(2,0) circle (0.15cm);
        \draw(1,2) circle (0.15cm);
  \draw[->] (1.15,2) .. controls (1.55,2.5) and (0.45,2.5) .. (0.85,2);
  \draw[->] (0.15,0) -- (1.85,0); 
  \draw[->] (1.9,0.1) -- (1.1,1.9);
  \draw[->] (1.15,1.95) .. controls (2,1.3) .. (2,0.15);
  \draw[->] (0.1,0.1) -- (0.9,1.9);
  \draw[->] (0.85,1.95) .. controls (0,1.3) .. (0,0.15);
  \node at (-0.1,1.3) {$\gamma$};
  \node at (2.1,1.3) {$\gamma$};
  \node at (1,2.6) {$\gamma$};
  \node at (0.7,0.8) {$\textrm{id}$};
  \node at (1.3,0.8) {$\textrm{id}$};
  \node at (1,0.2) {$\textrm{id}$};
    \end{tikzpicture}
    \hspace{1cm}
    \begin{subfigure}[b]{.2\textwidth}
   \centering
      \begin{tikzpicture}[scale = 0.5] 
        \foreach \x in {0,72,144,...,360} \draw (\x:1.5cm) -- (\x+72:1.5cm);
      \foreach \x in {0,72,144,...,360} \draw (\x:1.5cm) -- (\x+20:2cm) -- (\x+72:1.5cm);
      \foreach \x in {0,72,144,...,360} \draw (\x:1.5cm) -- (\x+50:2.5cm) -- (\x+72:1.5cm);
      \foreach \x in {0,72,144,...,288} \draw (\x+20:2cm) -- (\x+50:2.5cm);
     \foreach \x in {0,72,144,...,360} \node[inner sep=1pt,circle,draw,fill] at (\x:1.5cm) {};
     \foreach \x in {0,72,144,...,360} \node[inner sep=1pt,circle,draw,fill] at (\x+20:2cm) {};
     \foreach \x in {0,72,144,...,360} \node[inner sep=1pt,circle,draw,fill] at (\x+50:2.5cm) {};
      \end{tikzpicture}
\end{subfigure}
\caption{Near-balanced $\mathbb{Z}_5$-gain graph, and its covering.}
\label{near-balanced image}
\end{figure}

\begin{definition}
\label{sparsity defn.}
    Let $(G,\psi)$ be a $\Gamma$-gain graph. Let $m,l$ be non-negative integers such that $0\leq m\leq2,0\leq l\leq3,m\leq l$. $\left(G,\psi\right)$ is called \emph{$(2,m,3,l)$-gain-sparse} if 

    \begin{itemize}
    \item any balanced subgraph $H$ of $(G,\psi)$ with $E(H)\neq\emptyset$ is $(2,3)$-sparse; 
    \item $|E(H)|\leq 2|\overline{V(H)}|+m|V_0(H)|-l$ for any subgraph $H$ of $\left(G,\psi\right)$ with $E(H)\neq\emptyset$.
\end{itemize}
$(G,\psi)$ is called \textit{$(2,m,3,l)$-gain-tight} if it is  $(2,m,3,l)$-gain-sparse and $|E(G)|=2|\overline{V(G)}|+m|V_0(G)|-l$.
\end{definition}

\begin{remark}
    Let $(G,\psi)$ be a $\Gamma$-gain graph. Suppose that, for some $0\leq m\leq2,0\leq l\leq3$, $m\leq l$, $(G,\psi)$ is $(2,m,3,l)$-gain sparse, and let $H$ be a balanced subgraph of $(G,\psi)$ with $E(H)\neq\emptyset$. Then $H$ must satisfy $|E(H)|\leq2|V(H)|-3$, as well as $|E(H)|\leq2|\overline{V(H)}|+m|V_0(H)|-l$. It is easy to check that, whenever $(2-m)|V_0(H)|>3-l$, the latter condition is stronger than the former.
\end{remark}

An argument similar to the proof of Lemma 4.13 in \cite{Ikeshita} shows the following. 
\begin{lemma}
\label{lemma: the subgraph is connected}
    For $0\leq m\leq2,1\leq l\leq3$ such that $m\leq l$, any $(2,m,l)$-tight graph with non-empty edge-set has exactly one connected component with non-empty edge set (but may have other connected components consisting of isolated vertices).
\end{lemma}
\begin{proof}
    Fix $0\leq m\leq2,1\leq l\leq3$ such that $m\leq l$. Let $c_0\geq0,c\geq1$ be integers such that $c-c_0\geq1$, and $(G,\psi)$ be a $(2,m,l)$-tight graph with connected components $H_1,\dots,H_c$, of which $H_1,\dots,H_{c_0}$ are isolated vertices, and $H_{c_0+1},\dots,H_c$ have non-empty edge sets. Assume, by contradiction, that $c-c_0\geq2$. Then,
    \begin{equation*}
        |E(G)|=\sum_{i=c_0+1}^c|E(H_i)|
        \leq2\sum_{i=c_0+1}^c|\overline{V(H_i)}|+m\sum_{i=c_0+1}^c|V_0(H_i)|-(c-c_0)l
        \leq2|\overline{V(G)}|+m|V_0(G)|-(c-c_0)l,
    \end{equation*}
    where the last inequality holds with equality if $c_0=0$. Since $(c-c_0)\geq2,l\geq1$, this is strictly less than $2|\overline{V(G)}|+m|V_0(G)|-l$, which contradicts the fact that $(G,\psi)$ is $(2,m,l)$-tight. 
    \end{proof}

If $|\Gamma|\geq4,2\leq j\leq|\Gamma|-2$, then there are additional conditions that $\Gamma$-gain graphs must satisfy in order for their liftings to have $\rho_j$-symmetrically isostatic realisations. Hence, we introduce  more refined  sparsity conditions. First, we need the following notions, which may also be found in Section 4.3 of \cite{Ikeshita} and Section 2 of \cite{RIandST16}. See also \cite{KCandST20}. 

Let $k:=|\Gamma|\geq4$ and $(G,\psi)$ be a $\Gamma$-gain graph. For $2\leq j\leq k-2,-1\leq i\leq 1$, we define the following sets:
\begin{equation*}
    S_i(k,j)=\begin{cases}
        \{n\in\mathbb{N}:2\leq n, n|k,j\equiv i(\bmod n)\} & \text{if }j\text{ is even}\\
        \{n\in\mathbb{N}:2<n, n|k,j\equiv i(\bmod n)\} & \text{if }j\text{ is odd}
    \end{cases}
\end{equation*}
We say a connected subset $F$ of $E(G)$ (equivalently, a connected subgraph $H$ of $G$) is $S_0(k,j)$ if $\left<F\right>\simeq\mathbb{Z}_n$ (equivalently, $\left<H\right>\simeq\mathbb{Z}_n$) for some $n\in S_0(k,j).$ Similarly, we say $F$ (equivalently, $H$) is $S_{\pm1}(k,j)$ if $\left<F\right>\simeq\mathbb{Z}_n$ (equivalently, $\left<H\right>\simeq\mathbb{Z}_n$) for some $n\in S_{-1}(k,j)\cup S_1(k,j).$ We say $F$ (equivalently, $H$) is $S(k,j)$ if it is either $S_0(k,j)$ or $S_{\pm1}(k,j)$.

We also define the function $f_k^j$ on $2^{E(G)}$ by
\begin{equation*}
    f_k^j(F)=\sum_{X\in C(F)}\left\{2|V(X)|-3+\alpha_k^j(X)\right\},
\end{equation*}
where $F$ is a subset of $E(G)$, $C(F)$ denotes the set of connected components of $F$, and
\begin{equation*}
    \alpha_k^j(X)=\begin{cases}
        0 & \text{if }X\text{ is balanced}\\
        1 & \text{if }j\text{ is odd and }\left<X\right>\simeq\mathbb{Z}_2\\
        2-|V_0(X)| & \text{if }X\text{ is }S_{\pm1}(k,j)\\
        2-2|V_0(X)| & \text{if }X\text{ is }S_0(k,j)\text{ or } |V_0(X)|=0 \text{ and } X \text { is proper near-balanced}\\
        3-2|V_0(X)| & \text{otherwise}
    \end{cases}
\end{equation*}
Recall that the concept of near-balancedness is only defined on graphs with no fixed vertices. Hence, if $X$ is near-balanced, we assume by default that it has no fixed vertices. It was shown in \cite[Lemma 4.19(d)]{Ikeshita} that $S_0(k,j)\cap S_i(k,j)=\emptyset$ for $i=1,-1$, and hence the functions $\alpha_k^j$ and $f_k^j$ are well-defined. See also \cite{PhDThesis}.

\begin{definition}
\label{definition of sparsity, higher order rotation}
    Let $k:=|\Gamma|\geq4,2\leq j\leq k-2$ and $(G,\psi)$ be a $\Gamma$-gain graph. $(G,\psi)$ is said to be $\mathbb{Z}^{j}_k$-gain sparse if $|E(H)|\leq f_k^j(E(H))$ for all non-empty subgraphs $H$ of $G$. It is said to be $\mathbb{Z}^{j}_k$-gain tight if it is  $\mathbb{Z}^{j}_k$-gain sparse and $|E(G)|=f_k^j(E(G))$.
\end{definition} 

\subsection{Symmetric frameworks}\label{sec:symfwks}
Let $\Tilde{G}$ be a $\Gamma$-symmetric graph, and $\tau:\Gamma\rightarrow O(\mathbb{R}^d)$ be a faithful representation. We say a \textit{realisation} $(\Tilde{G},\Tilde{p})$ of $\Tilde{G}$ is \textit{$\tau(\Gamma)$-symmetric} if $\tau(\gamma)\Tilde{p}(v)=\Tilde{p}(\gamma v)$ for all $\gamma\in\Gamma,v\in V(\Tilde{G}).$ 

Notice that we are now realising the group $\Gamma$ geometrically. For instance, if $|\Gamma|=2$, then $\tau(\Gamma)$ can be identified either with the rotation group $\mathcal{C}_2:=\{\textrm{id},C_2\}$, where $C_2$ is the rotation of $\mathbb{R}^2$ by $\pi$ around the origin, or $\mathcal{C}_s=\{\textrm{id},\sigma\}$, where $\sigma$ is a reflection whose mirror line goes through the origin. By applying a rotation, we may always assume that the mirror line is the $y$-axis. The $\Gamma$-symmetric graph in Figure~\ref{first gain graph example}, for example, can be interpreted as a $\mathcal{C}_2$-symmetric framework. Similarly, for $k:=|\Gamma|\geq 3$, we  identify $\tau(\Gamma)$ with the group $\mathcal{C}_k$ which is generated by a counterclockwise rotation about the origin by $2\pi/k$. 

Consider a $\tau(\Gamma)$-symmetric framework $(\tilde{G},\tilde{p})$. By definition, $\tilde G$ must be a $\Gamma$-symmetric graph and so it has a quotient $\Gamma$-gain graph $(G,\psi)$. Then, for all $v\in V(G)$, we can define $p_v:=\tilde{p}\left(v^{\star}\right)$, where $v^{\star}$ is the representative of $v$ in $V(\tilde{G})$. This allows us to define the \textit{(quotient) $\tau(\Gamma)$-gain framework} of $(\Tilde{G},\Tilde{p})$ to be the triplet $(G,\psi,p)$. 

We say $p$ (or, equivalently $\tilde{p}$, $(\tilde{G},\tilde{p})$, $(G,\psi,p)$) is \textit{$\tau(\Gamma)$-generic} if $\textrm{rank}(R(\tilde{G},\Tilde{p}))\geq\textrm{rank}(R(\tilde{G},\Tilde{q}))$ for all $\tau(\Gamma)$-symmetric realisations $(\tilde{G},\Tilde{q})$ of $\Tilde{G}$. The set of all $\tau(\Gamma)$-generic configurations of $\tilde{G}$ is a dense, open subset of the set of $\tau(\Gamma)$-symmetric configurations of $\tilde{G}$. 

\subsection{Block-diagonalisation of the rigidity matrix}
Recall that $\Gamma=\left<\gamma\right>$ is isomorphic to $\mathbb{Z}_k$ through an isomorphism which sends $\gamma$ to 1. From group representation theory we know that $\Gamma$ has $k$ irreducible representations $\rho_0,\dots,\rho_{k-1}$ such that  $\rho_j:\Gamma\rightarrow\mathbb{C}\setminus\{0\}$ sends $\gamma^r\in\Gamma$ to $\omega^{rj}$, where $\omega=e^{\frac{2\pi i}{k}}$. For some faithful representation $\tau:\Gamma\rightarrow O(\mathbb{R}^d)$, let $(\tilde{G},\tilde{p})$ be a $\tau(\Gamma)$-symmetric framework. We define $P_{V(\tilde{G})}:\Gamma\rightarrow GL(\mathbb{R}^{|V(\tilde{G})|})$ to be the linear representation of $\Gamma$ that sends an element $\gamma'\in\Gamma$ to the matrix $\left[\delta_{\tilde{u},\gamma'\tilde{v}}\right]_{\tilde{u},\tilde{v}}$, where $\delta$ denotes the Kronecker delta symbol. We also define $P_{E(\tilde{G})}:\Gamma\rightarrow GL(\mathbb{R}^{|E(\tilde{G})|})$ to be the linear representation of $\Gamma$ that sends an element $\gamma'\in\Gamma$ to the matrix $\left[\delta_{\tilde{e},\gamma'\tilde{f}}\right]_{\tilde{e},\tilde{f}}$. Theorem 3.1 in \cite{sch10a} shows the following.

\begin{lemma}
\label{Lemma: rigidity matrix block diagonalises}
For all $\gamma\in\Gamma$,
\begin{equation*}
 P_{E(\tilde{G})}^{-1}(\gamma)R(\tilde{G},\tilde{p})(\tau\otimes P_{V(\tilde{G})})(\gamma)=R(\tilde{G},\tilde{p}).
\end{equation*}
\end{lemma}

By Schur's lemma, this implies that the rigidity matrix of $(\tilde{G},\tilde{p})$ block-decomposes with respect to suitable symmetry-adapted bases, which subdivides the column space into the direct sum of the spaces $V^0,\dots,V^{k-1}$, where each $V^j$ is the $(\tau\otimes P_{E(\tilde{V})})$-invariant subspace corresponding to $\rho_j$. Similarly, the row space can be written as the direct sum of $k$ spaces $W^0,\dots,W^{k-1}$, each being the $P_{E(\tilde{G})}$-invariant subspace corresponding to $\rho_j$ (for details, see Section 3.2 of \cite{sch10a}). Hence, we can write the rigidity matrix in the form

\begin{center}
$\tilde{R}(\tilde{G},\tilde{p})=\begin{pmatrix}
\tilde{R}_0(\tilde{G},\tilde{p})&&0\\
&\ddots& \\
0& &\tilde{R}_{k-1}(\tilde{G},\tilde{p})
\end{pmatrix}$,
\end{center}

where each $\tilde{R}_j(\tilde{G},\tilde{p})$ is determined by some $\rho_j$. This decomposition into subspaces also allows us to define the following.

\begin{definition}
With the notation above, we say an infinitesimal motion $\tilde{m}:V(\tilde{G})\rightarrow\mathbb{C}^{d\left|V(\tilde{G})\right|}$ is \textit{symmetric with respect to $\rho_j$} (or $\rho_j$-symmetric) if it lies in $V^j$.
\end{definition}

\begin{remark}
Since all irreducible representations $\rho_0,\dots,\rho_{k-1}$ of $\Gamma$ are 1-dimensional, an infinitesimal motion $\tilde{m}:V(\tilde{G})\rightarrow\mathbb{C}^{d|V(\tilde{G})|}$ is $\rho_j$-symmetric if and only if for all $\gamma\in\Gamma$ and all $v\in V(\tilde{G})$,
\begin{equation}
\label{defn of anti-symmetric motion}
    \tilde{m}\left(\gamma v\right)=\overline{\rho_j(\gamma)}\tau(\gamma)\tilde{m}(v),
\end{equation}
where $\overline{\rho_j(\gamma)}$ indicates the complex conjugate of $\rho_j(\gamma)$ (for details, see Section 4.1.2 in \cite{bt2015}). We usually refer to $\rho_0$-symmetric infinitesimal motions as \textit{fully-symmetric motions}, since they exhibit the full symmetry. If $k=2$, we refer to $\rho_1$-symmetric infinitesimal motions as \textit{anti-symmetric motions}, since the motion vectors are reversed by the non-trivial element of the group.
\end{remark}

\begin{definition}
    Let $k:=|\Gamma|\geq2$, $(\tilde{G},\tilde{p})$ be a $\tau(\Gamma)$-symmetric framework and $0\leq j\leq k-1$. We say $(\tilde{G},\tilde{p})$ is \textit{$\rho_j$-symmetrically isostatic} if  all $\rho_j$-symmetric infinitesimal motions of $(\tilde{G},\tilde{p})$ are trivial and $\tilde{R}_j(\tilde{G},\tilde{p})$ has no non-trivial row dependence. We usually refer to a $\rho_0$-symmetrically isostatic framework as a \textit{fully-symmetrically isostatic framework.} If $k=2$, we refer to a $\rho_1$-symmetrically isostatic framework as an \textit{anti-symmetrically isostatic framework.} 
\end{definition}

\begin{remark}
    For $k:=|\Gamma|\geq2,0\leq j\leq k-1$, let $(G,\psi,p)$ be a $\rho_j$-symmetrically isostatic $\tau(\Gamma)$-gain framework. Then, by definition of $\tau(\Gamma)$-genericity (recall Section~\ref{sec:symfwks}), any $\tau(\Gamma)$-generic realisation $(G,\psi,q)$ of $(G,\psi)$ is $\rho_j$-symmetrically isostatic. 
\end{remark}

We will now start working in $\mathbb{R}^2$. The following is a well known fact and can be read off standard character tables \cite{atk70}. See also the proof of Theorem 6.3 and Lemma 6.7 in \cite{bt2015}. 

\begin{proposition}\label{prop:trivmotions}
Let $\tau:\Gamma\rightarrow O(\mathbb{R}^2)$ be a faithful representation. Given a $\tau(\Gamma)$-symmetric framework $(\Tilde{G},\Tilde{p})$, the following hold:
\begin{itemize}
    \item[(i)] If $\tau(\Gamma)=\mathcal{C}_s$, all infinitesimal rotations of $(\Tilde{G},\Tilde{p})$ are $\rho_1$-symmetric. The space of infinitesimal translations  of $(\Tilde{G},\Tilde{p})$ decomposes into two $1$-dimensional subspaces, one consisting of $\rho_0$-symmetric and the other of $\rho_1$-symmetric translations (as shown in Figure~\ref{figure: trivial motions of reflection framework}). 
    \item[(ii)] If $\tau(\Gamma)=\mathcal{C}_2$, all infinitesimal rotations of $(\Tilde{G},\Tilde{p})$ are $\rho_0$-symmetric and all infinitesimal translations of $(\Tilde{G},\Tilde{p})$  are $\rho_1$-symmetric.
    \item[(iii)] If $\tau(\Gamma)=\mathcal{C}_k$ for some $k\geq3$, all infinitesimal rotations of $(\Tilde{G},\Tilde{p})$ are $\rho_0$-symmetric. The space of infinitesimal translations of $(\Tilde{G},\Tilde{p})$ decomposes into two $1$-dimensional subspaces, one consisting of $\rho_1$-symmetric  and the other of $\rho_{k-1}$-symmetric translations.
\end{itemize}
\end{proposition}

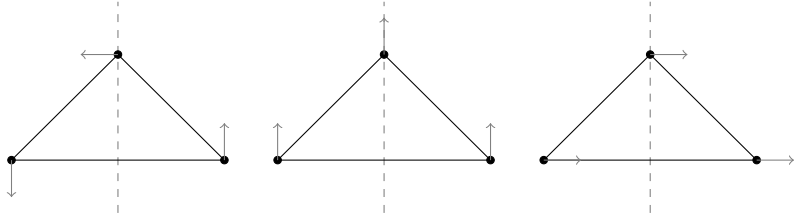
\begin{figure}[H]
    \centering
    \begin{tikzpicture}
  [scale=0.7,auto=left]
\draw[dashed, gray] (0,-1)--(0,3);
\filldraw(0,2) circle(2pt);
\filldraw(2,0) circle(2pt);
\filldraw(-2,0) circle(2pt);
\draw(0,2) -- (2,0) -- (-2,0) -- (0,2);
\draw[gray, ->](0,2) -- (-0.7,2);
\draw[gray, ->](-2,0) -- (-2, -0.7);
\draw[gray, ->](2,0) -- (2,0.7);

\draw[dashed, gray] (5,-1)--(5,3);
\filldraw(5,2) circle(2pt);
\filldraw(7,0) circle(2pt);
\filldraw(3,0) circle(2pt);
\draw(5,2) -- (7,0) -- (3,0) -- (5,2);
\draw[gray, ->](5,2) -- (5,2.7);
\draw[gray, ->](3,0) -- (3, 0.7);
\draw[gray, ->](7,0) -- (7,0.7);

\draw[dashed, gray] (10,-1)--(10,3);
\filldraw(10,2) circle(2pt);
\filldraw(12,0) circle(2pt);
\filldraw(8,0) circle(2pt);
\draw(10,2) -- (12,0) -- (8,0) -- (10,2);
\draw[gray, ->](10,2) -- (10.7,2);
\draw[gray, ->](8,0) -- (8.7,0);
\draw[gray, ->](12,0) -- (12.7,0);
\end{tikzpicture}
\caption{Trivial infinitesimal motions of a $\mathcal{C}_s$-symmetric framework $(\tilde{G},\tilde{p})$. From left to right $(\tilde{G},\tilde{p})$ is rotated, translated in the direction of the symmetry line and translated in the direction perpendicular to the symmetry line. The second motion maintains symmetry, the other two break the symmetry.}
\label{figure: trivial motions of reflection framework}
\end{figure}

\section{Phase-symmetric orbit matrices} \label{sec:orbitmat}

We now establish `orbit matrices' which are equivalent to the matrices composing the block-diagonalised version of $(\tilde{G},\tilde{p})$, and which we will denote $O_j(G,\psi,p)$. These matrices can be written down directly without using representation theory and they allow us to establish a recursive construction of $(2,m,3,l)$-sparse graphs. Moreover, their structure allows us to work with $\Gamma$-gain graphs rather than $\Gamma$-symmetric graphs, deleting any redundancies that would be present in the rigidity matrix.

\subsection{Dimensions of the block matrices}
Since the diagonalisation of the rigidity matrix follows from the fact that it intertwines $\tau\otimes P_{V(\tilde{G})}$ and $P_{E(\tilde{G})}$ (recall Lemma~\ref{Lemma: rigidity matrix block diagonalises}),  
the dimension of each block can be determined by studying the decomposition of $\tau\otimes P_{V(\tilde{G})}$ and $P_{E(\tilde{G})}$.

Let $|\Gamma|=k$. Let $(\tilde{G},\tilde{p})$ be a $\tau(\Gamma)$-symmetric framework, where $\tau:\Gamma\rightarrow O(\mathbb{R}^2)$ is a faithful representation.

Let $\rho_{\textrm{reg}}:\Gamma\rightarrow GL(\mathbb{R}^k)$ denote the regular representation of $\Gamma$, that sends $\gamma\in\Gamma$ to the matrix $\left[\delta_{g,\gamma h}\right]_{g,h}$, where $\delta$ again represents the Kronecker delta symbol. From representation theory, we know that $\rho_{\textrm{reg}}$ is the direct sum of the irreducible representations of $\Gamma$. It is also easy to see that $P_{V(\tilde{G})}$ is the direct sum of $|\overline{V(G)}|$ copies of $\rho_{\textrm{reg}}$ and $|V_0(G)|$ copies of the $1\times1$ identity matrix, and so
\begin{equation*}
    P_{V(\tilde{G})}\simeq|\overline{V(G)}|\rho_{\textrm{reg}}\oplus|V_0(G)|I_1\simeq\bigoplus_{j=0}^{k-1}|\overline{V(G)}|\rho_j\oplus |V_0(G)|I_1.
\end{equation*}

Given $0\leq j,j'\leq k-1$, the character of $\rho_j\otimes\rho_{j'}$ is the coordinate-wise product of $\rho_j$ and $\rho_{j'}$. So, the multiplicity of $\rho_j$ in $\tau\otimes\rho_{\textrm{reg}}$ is $\textrm{Tr}(\tau(\textrm{id}))=2$. Hence,              
              
\begin{equation*}
    \tau\otimes P_{V(\tilde{G})}\simeq|\overline{V(G)}|[\tau\otimes\rho_{\textrm{reg}}]\oplus|V_0(G)|\tau\simeq\bigoplus_{j=0}^{k-1}2|\overline{V(G)}|\rho_j\oplus|V_0(G)|\tau.
\end{equation*}
 So, for each free vertex $v\in V(G)$, every block matrix contains two columns. Where the columns corresponding to the fixed vertices lie depends on the map $\tau:\Gamma\rightarrow O(\mathbb{R}^2)$.

\begin{proposition}
\label{proposition: dimension of the blocks}
   Let $(\Tilde{G},\Tilde{p})$ be a $\tau(\Gamma)$-symmetric framework for some faithful representation $\tau:\Gamma\rightarrow O(\mathbb{R}^2)$. The following statements hold:
   \begin{itemize}
       \item[(i)] If $\tau(\Gamma)=\mathcal{C}_s$, $\tilde{R}_0(\tilde{G},\tilde{p})$ and $\tilde{R}_1(\tilde{G},\tilde{p})$ both have $2|\overline{V(G)}|+|V_0(\Tilde{G})|$ columns.
       \item[(ii)] If $\tau(\Gamma)=\mathcal{C}_2$, then $\Tilde{R}_0(\Tilde{G},\Tilde{p})$ has $2|\overline{V(G)}|$ columns and $\Tilde{R}_1(\Tilde{G},\Tilde{p})$ has $2|V(G)|$ columns.
       \item[(iii)] If $\tau(\Gamma)=\mathcal{C}_k$ for some $k\geq3$, then $\Tilde{R}_1(\Tilde{G},\Tilde{p})$ and $\Tilde{R}_{k-1}(\Tilde{G},\Tilde{p})$ both have $2|\overline{V(G)}|+|V_0(\Tilde{G})|$ columns, and all the other blocks have $2|\overline{V(G)}|$ columns. 
       \item[(iv)] For all $\tau(\Gamma)$ and all even $0\leq j\leq|\Gamma|-1$, $\Tilde{R}_j(\Tilde{G},\Tilde{p})$ has $|E(G)|$ rows and for all odd $0\leq j\leq|\Gamma|-1$, $\Tilde{R}_j(\Tilde{G},\Tilde{p})$ has $|\overline{E(G)}|$ rows (notice that when $|\Gamma|$ is odd, $E(G)=\overline{E(G)}$). 
   \end{itemize}
\end{proposition}

\begin{proof}
    First, let $|\Gamma|=2$ (so, $\tau(\Gamma)$ is either $\mathcal{C}_s$ or $\mathcal{C}_2$), and recall that $\Gamma$ has irreducible representations $\rho_0,\rho_1$, where $\rho_0$ is the identity representation, and $\rho_1$ sends the non-identity element $\gamma$ of $\Gamma$ to $-1$. Let $\tau_{\textrm{ref}}:\Gamma\rightarrow O(\mathbb{R}^2)$ be the reflection homomorphism that maps $\gamma$ to $\textrm{diag}(-1,1)$ and let $\tau_{\textrm{rot}}:\Gamma\rightarrow O(\mathbb{R}^2)$ be the two-fold rotation homomorphism that maps $\gamma$ to $\textrm{diag}(-1,-1)$. It is easy to see that $\tau_{\textrm{ref}}=\rho_0\oplus\rho_1$ and so
    \begin{equation*}
        \tau_{\textrm{ref}}\otimes P_{V(\tilde{G})}\simeq\bigoplus_{j=0}^{k-1}2|\overline{V(G)}|\rho_j\oplus|V_0(G)|\tau_\textrm{ref}\simeq\bigoplus_{j=0,1}(2|\overline{V(G)}|+|V_0(G)|)\rho_j.
    \end{equation*}
    Similarly, we have $\tau_{\textrm{rot}}=\rho_1\oplus\rho_1$. So, 
    \begin{equation*}
        \tau_{\textrm{rot}}\otimes P_{V(\tilde{G})}\simeq\bigoplus_{j=0}^{k-1}2|\overline{V(G)}|\rho_j\oplus|V_0(G)|\tau_\textrm{rot}\simeq(2|\overline{V(G)}|)\rho_0\oplus(2|V(G)|)\rho_1.
    \end{equation*}
    (i) and (ii) follow.

    Now, let $|\Gamma|=k\geq3$, so that $\tau(\Gamma)=\mathcal{C}_k$. Let $\alpha=\frac{2\pi}{k}$ and $\Gamma=\langle\gamma\rangle$, where $\gamma\in\Gamma\mapsto1\in\mathbb{Z}_k$ through an isomorphism. The standard $k$-fold rotation homomorphism $\tau:\Gamma\rightarrow O(\mathbb{R}^2)$ is given by  $\tau\left(\gamma\right)=\begin{pmatrix}\cos(\alpha)&-\sin(\alpha)\\ \sin(\alpha) &\cos(\alpha)\end{pmatrix}.$ We apply a complexification of the Euclidean plane with a change of basis from $\mathcal{B}_1=\left\{\begin{pmatrix}
1&0
\end{pmatrix}^T,\begin{pmatrix}
0&1
\end{pmatrix}^T\right\}$ to $\mathcal{B}_2=\left\{\frac{1}{2}\begin{pmatrix}
-1-i&1-i
\end{pmatrix}^T,\frac{1}{2}\begin{pmatrix}
-1+i&1+i
\end{pmatrix}^T\right\}$, with the change of basis matrix $$M_{1\rightarrow2}=\frac{1}{2}\begin{pmatrix}
    -1-i & -1+i \\ 1-i & 1+i
\end{pmatrix}.$$ Then,
\begin{equation*}    \tau(\gamma)_{\mathcal{B}_2}=M_{1\rightarrow2}\tau(\gamma)_{\mathcal{B}_1}M_{1\rightarrow2}^{-1}=\begin{pmatrix}
    \cos(\alpha)-i\sin(\alpha) & 0\\ 0 & \cos(\alpha)+i\sin(\alpha)
\end{pmatrix}=\begin{pmatrix}
    \overline{\omega} & 0 \\ 0 & \omega
\end{pmatrix}.
\end{equation*}
It follows that $\tau=\rho_1\oplus \rho_{k-1}$. Hence,
\begin{center}
 $\tau\otimes P_{V(\tilde{G})}\simeq\bigoplus_{j=0}^{k-1}2|\overline{V(G)}|\rho_j\oplus|V_0(G)|(\rho_1\oplus\rho_{k-1})$.   
\end{center}

Hence, (iii) holds. Finally, we prove (iv). Recall that, for some integer $k\geq2$, $\Gamma=\left<\gamma\right>\simeq\mathbb{Z}_k$ through the isomorphism which maps $\gamma$ to 1. Consider the edge set $E(\Tilde{G})$.
If $k=2$, any edge is clearly either free or fixed by that whole group. If $k\geq3$, any edge is either free or fixed uniquely by $\textrm{id}$ and $\delta:=\gamma^{k/2}$. It was shown in Section 4.3 of \cite{bt2015} that for each $e\in E(G)\setminus\overline{E(G)}$, all the blocks $\Tilde{R}_j(\Tilde{G},\Tilde{p})$ such that $\rho_j(\delta)=1$ have one row, and all the other blocks have no rows (this argument does not use the fact that the action is free on the vertex set). It was also shown that each block $\Tilde{R}_j(\Tilde{G},\Tilde{p})$ has a row for each edge in $|\overline{E(G)}|$. Since $\rho_j(\delta)=\rho_j(\gamma^{k/2})=\exp(\frac{2\pi ikj}{2k})=\exp(\pi ij)$, it follows that $\rho_j(\delta)$ is 1 if and only if $j$ is even. Hence, for all even $j$, $\Tilde{R}_j(\Tilde{G},\Tilde{p})$ has $|E(G)|$ rows, and for all odd $j$ $\Tilde{R}_j(\Tilde{G},\Tilde{p})$ has $|\overline{E(G)}|$ rows.
\end{proof}

\begin{example}
\label{remark: no edges on the symmetry line}
Let $\Gamma=\left<\gamma\right>$ be isomorphic to $\mathbb{Z}_k$ through the isomorphism which sends $\gamma$ to 1. Suppose $k$ is even. An edge $e=\{u,v\}$ of $\tilde{G}$ is non-free if either both $u,v\in V_0(\tilde{G})$, in which case $S_{\Gamma}(e)=\Gamma$, or if $v=\gamma^{k/2}u$, in which case $S_{\Gamma}(e)=\{\textrm{id},\gamma^{k/2}\}$. If $\tau(\Gamma)=\mathcal{C}_k$, $|V_0(\tilde{G})|\leq1$, so there are no edges between fixed vertices. 

In Figure~\ref{fixed edges image} we show realisations of non-free edges in $\mathcal{C}_s$-symmetric (a,b,c) and $\mathcal{C}_2$-symmetric frameworks (d). Figures (a,b,d) show $\rho_1$-symmetric motions  of such bars, whereas (c) shows a $\rho_0$-symmetric motion. For any $\rho_{1}$-symmetric velocity assignment $\tilde{m}$ to the vertices, the equation $(\tilde{m}(v)-\tilde{m}(u))\cdot(\tilde{p}(v)-\tilde{p}(u))=0$ always holds. Hence, the edge $e$ constitutes no constraint for $\rho_{1}$-symmetric infinitesimal rigidity. This is not the case for a $\rho_{0}$-symmetric velocity assignment. (Note that in (c) the equation only holds because the velocities are parallel to the mirror line.) The edge in (d) can also be seen as a subgraph of a $\mathcal{C}_k$-symmetric framework $(\tilde{G},\tilde{p})$, where $k\geq4$ is even. Given an odd $j$ with $2\leq j\leq k-2$, (d) shows a $\rho_j$-symmetric motion of $(\tilde{G},\tilde{p})$, restricted to that edge. 
\begin{figure}[H]
    \centering
    \begin{tikzpicture}
  [scale=.8,auto=left]
  \draw[dashed](-1,-0.5) -- (-1,2.5);    
    \draw(-1,0) -- (-1,2);
    \draw[->,gray](-1,0) -- (-1.7,0);
    \draw[->,gray](-1,2) -- (-0.3,2);
    \filldraw (-1,0) circle (2pt);
    \filldraw (-1,2) circle (2pt);
    \node[left] at (-1,2){$u$};
    \node[right] at (-1,0){$v$};
    \node[below] at (-1,-1) {(a)};

    \draw[dashed](2,-0.5) -- (2,2.5); 
    \draw[->, gray](1,1) -- (1.5,0.5);
    \draw[->, gray](3,1) -- (3.5,1.5);
    \draw(1,1) -- (3,1);
    \filldraw(1,1) circle (2pt);
    \filldraw(3,1) circle (2pt);
    \node[above] at (1,1) {$u$};
    \node[below] at (3,1) {$v$};
    \node[below] at (2,-1) {(b)};

    \draw[->, gray](5,1) -- (5,1.5);
    \draw[->, gray](7,1) -- (7,1.5);
    \draw[dashed](6,-0.5) -- (6,2.5);  
    \draw(5,1) -- (7,1);
    \filldraw(5,1) circle (2pt);
    \filldraw(7,1) circle (2pt);
    \node[below] at (5,1){$u$};
    \node[below] at (7,1){$v$};
    \node[below] at (6,-1) {(c)};

    \draw[->, gray] (9,0.5) -- (9.4,1);
    \draw[->, gray] (11,1.5) -- (11.4,2);
    \filldraw (9,0.5) circle (2pt);
    \filldraw (11,1.5) circle (2pt);
    \filldraw (10,1) circle (1pt);
    \draw (9,0.5) -- (11,1.5);
    \node[below] at (9,0.5) {$u$};
    \node[below] at (11,1.5) {$v$};
    \node[below] at (10.3,1) {$(0,0)$};
    \node[below] at (10,-1) {(d)};
\end{tikzpicture}
    \caption{(a,b) Fixed bars of $\mathcal{C}_s$-symmetric frameworks with anti-symmetric motions. (c) Fixed bar of a $\mathcal{C}_s$-symmetric framework with a fully-symmetric motion. (d) Fixed bar of a $\mathcal{C}_2$-symmetric framework and a $\rho_1$-symmetric motion of $(\tilde{G},\tilde{p})$ applied to the bar.}
    \label{fixed edges image}
\end{figure}
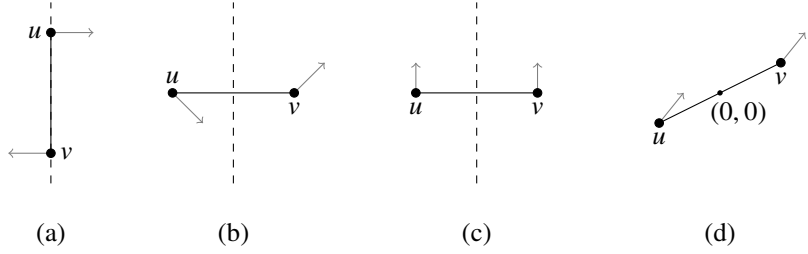
\end{example}

\subsection{Orbit matrices}\label{sec:orbitmatrices}
For each block matrix, we now construct an orbit matrix which has the same dimension and an isomorphic kernel. Thus, it will have the same nullity and rank and can be used for the corresponding symmetric rigidity analysis.

The orbit matrix for $j=0$ is given in Definition 5.1 of \cite{sw2011}. In Section 4.1.2 of \cite{bt2015}, the orbit matrix was defined for all other $1\leq j\leq|\Gamma|-1$, for the special case where $V_0(G)=\emptyset$. We model our definition of orbit matrices based on the definitions in \cite{bt2015} and \cite{sw2011}.

Let $k:=|\Gamma|$, and let $(G,\psi,p)$ be the $\tau(\Gamma)$-gain framework of a $\tau(\Gamma)$-symmetric framework $(\Tilde{G},\Tilde{p})$, with respect to a faithful representation $\tau:\Gamma\rightarrow O(\mathbb{R}^2)$. Let $U\subseteq V(G)$ be equal to $V_0(G)$ if $\tau(\Gamma)=\mathcal{C}_k$, and $0\leq j\leq k-2,j\neq1$, and $U=\emptyset$ otherwise. For $u\in V(G)\setminus U$, let 
\begin{equation*}
    M^{j}_{\tau(\Gamma)}(u)=\begin{cases}
    \begin{pmatrix}
        0 & 1
    \end{pmatrix}^T & \text{if }u\in V_0(G)\text{ and }\tau(\Gamma)=\mathcal{C}_s,j=0\\
        \begin{pmatrix}
            1 & 0
        \end{pmatrix}^T & \text{if }u\in V_0(G)\text{ and } \tau(\Gamma)=\mathcal{C}_s,j=1\\
        \begin{pmatrix}
            1 & -i
        \end{pmatrix}^T & \text{if }u\in V_0(G)\text{ and }\tau(\Gamma)=\mathcal{C}_k \text{ for some } k\geq3,j=1 \\
        \begin{pmatrix}
            1 & i
        \end{pmatrix}^T & \text{if }u\in V_0(G)\text{ and }\tau(\Gamma)=\mathcal{C}_k \text{ for some } k\geq3,j=k-1 \\
        I_2 & \text{otherwise}.
    \end{cases}
\end{equation*}

For each $0\leq j\leq k-1$ and $u\in V(G)\setminus U$, let $c_{\tau(\Gamma)}^j(u)=\textrm{rank}(M^{j}_{\tau(\Gamma)}(u))$.  When clear, we will often omit the symmetry group and simply write $M_u^j$ and $c_u^j$ for $ M^{j}_{\tau(\Gamma)}(u)$ and $c_{\tau(\Gamma)}^j(u)$, respectively. 

\begin{remark}
    Let $u\in V(G)\setminus U$ and $A_{\tau(\Gamma)}^j(u):=\big\{\tilde{m}(u^{\star}): \tilde{m}\text{ is a }\rho_j-\text{symmetric motion of }(\tilde{G},\tilde{p})\big\}$.  (For $\mathcal{C}_s$, for instance, if $u\in V_0(G)$, then $A_{\mathcal{C}_s}^0(u)=\Big\{\begin{pmatrix}0 & a\end{pmatrix}^T:a\in\mathbb{R}\Big\}$.) Then there is some basis $\mathcal{B}$ of $A_{\tau(\Gamma)}^j(u)$ such that $M_{\tau(\Gamma)}^j(u)$ is the matrix whose columns are the elements of $\mathcal{B}$. This can be easily checked by applying Equation~(\ref{defn of anti-symmetric motion}). For details, see \cite{PhDThesis}.
\end{remark}

\begin{definition}
\label{definition of orbit matrices}
    With the same notation as above, for $0\leq j\leq k-1$, the $\rho_j$-orbit matrix $O_j(G,\psi,p)$ of $(G,\psi,p)$ is a matrix with $c_u^j$ columns for each $u\in V(G)\setminus U$. If $j$ is even, $O_j(G,\psi,p)$ has $|E(G)|$ rows. Otherwise, it has $|\overline{E(G)}|$ rows. Each edge $e=(u,v)\in E(G)$ which has a row in $O_j(G,\psi,p)$, has row
    \begin{equation*}
    \begin{matrix}  
    \begin{matrix}
    u& \text{ } & \text{} & \text{} & \text{} & \text{} & \text{} & \text{} & \text{ } & \text{ } \text{ } & \text{ } & \text{ } & v  & \text{}
    \end{matrix}\\
    \begin{pmatrix}
    \dots & \left(p_u-\tau(\psi(e))\left(p_v\right)\right)^TM^{j}_u
    &\dots & \rho_j(\psi(e))(p_v-\tau(\psi^{-1}(e))\left(p_u\right))^TM^j_v&\dots
    \end{pmatrix}
\end{matrix}
\end{equation*}
if $u\neq v$, and it has row 
\begin{equation*}
\begin{matrix}
\begin{matrix}
    u
    \end{matrix}\\
        \begin{pmatrix}
    \dots & \left(p_u+\rho_j(\psi(e))p_u-\tau(\psi(e))p_u-\rho_j(\psi(e))\tau(\psi^{-1}(e))p_u\right)^T&\dots
    \end{pmatrix}.
\end{matrix}
\end{equation*}
otherwise. If $u$ (respectively, $v$) lies in $U$, then the columns corresponding to $u$ (respectively, $v$) vanish. 
\end{definition}

\begin{remark}
\label{remark: free-action case of orbit matrix}
    For $j=0$, Definition~\ref{definition of orbit matrices} coincides with the definition of the orbit matrix given in \cite{sw2011}. In the same paper, it was shown that $\tilde{R}_0(\tilde{G},\tilde{p})$ is equivalent to $O_0(G,\psi,p)$.    
    Similarly, if $e=(u,v)$ is such that $u,v\in\overline{V(G)}$, the row of $e$ under Definition~\ref{definition of orbit matrices} coincides with the row of $e$ under the definition of the phase-symmetric orbit matrix $O_j(G,\psi,p)$ defined in \cite{bt2015}, for all $0\leq j\leq k-1$. In the same paper, it was shown that $\tilde{R}_j(\tilde{G},\tilde{p})$ and $O_j(G,\psi,p)$ have the same size and $\ker \tilde{R}_j(\tilde{G},\tilde{p})=\ker O_j(G,\psi,p)$ for all $0\leq j\leq k-1$ whenever $V_0(G)=\emptyset$.
\end{remark}

\begin{lemma}
\label{lemma: kernels of orbit matrix and block equal}
    Let $k:=|\Gamma|$, let $(\tilde{G},\tilde{p})$ be a $\Gamma$-symmetric framework and let $0\leq j \leq k-1$. Let $\mathcal{O}$ denote the set of vertex orbit representatives of $\tilde{G}.$ Define the subset $\mathcal{O}'\subseteq V_0(\tilde{G})$ to be $V_0(\tilde{G})$ if $\tau(\Gamma)=\mathcal{C}_k$ and $0\leq j\leq k-2,j\neq1$, and $\emptyset$ otherwise.
    For some fixed $u^{\star}\in\mathcal{O}$ and some free $v^{\star}\in\mathcal{O}$, let $\{u^{\star},v^{\star}\}\in E(\tilde{G})$. For each $\delta\in\Gamma$, let $(G,\psi_{\delta},p)$ be a $\Gamma$-gain framework of $(\tilde{G},\tilde{p})$ such that the edge $e=(u,v)$ has gain $\delta$. Fix $\gamma\in \Gamma$.  Then a vector $m$ lies in $\ker O_j(G,\psi_{\gamma},p)$  if and only if $\tilde{m}':\mathcal{O}\setminus\mathcal{O'}\rightarrow\mathbb{R}^2$ defined by $\Tilde{m}'(w^{\star})= M^j_w m(w)$  is the restriction of a $\rho_j$-symmetric motion $\tilde{m}$ of $(\tilde{G},\tilde{p})$ to $\mathcal{O}\setminus\mathcal{O'}$. 
\end{lemma}

\begin{proof}
    Let $\tilde{m}:V(\tilde{G})\rightarrow\mathbb{C}^2$ be defined by 
    \begin{equation*}
        \tilde{m}(\gamma w^{\star})=\begin{cases}
        \overline{\rho_j(\gamma)}\tau(\gamma)\tilde{m}'(w^{\star}) & \text{for all }w^{\star}\in \mathcal{O}\setminus\mathcal{O'},\gamma\in\Gamma\\
        \begin{pmatrix}
            0 & 0
        \end{pmatrix}^T & \text{for all }w^{\star}\in\mathcal{O'},\gamma\in\Gamma.
        \end{cases}
    \end{equation*}
   Clearly, $\tilde{m}'$ is a restriction of $\tilde m$ to $\mathcal{O}\setminus\mathcal{O'}$. Moreover, it is easy to see that $\tilde{m}$ is a $\rho_j$-symmetric motion of $(\tilde{G},\tilde{p})$ if and only if it is an infinitesimal motion of $(\tilde{G},\tilde{p})$. 
   
   View $m$ as a column vector. For each row $r$ in $O_j(G,\psi_{\gamma},p)$ that represents an edge $e=(u_1,u_2)\in E(G)$, we check that $rm$ is zero if and only if $\tilde{m}$ satisfies the conditions of being an infinitesimal motion of the framework on the subgraph induced by the elements of the orbit $e$.
    
   If $u_1,u_2\in\overline{V(G)}$, this has been shown in Section 4.1.2 in \cite{bt2015}. If $u_1,u_2\in V_0(G)$, then $\tau(\Gamma)=\mathcal{C}_s$, since $\tau(\Gamma)=\mathcal{C}_k$ implies $|V_0(G)|\leq1$ by definition of a framework. Since the row corresponding to $\{u_1^{\star},u_2^{\star}\}$ in $\tilde{R}_1(\tilde{G},\tilde{p})$ is zero, we need only consider the case where $j=0$. However, by Remark~\ref{remark: free-action case of orbit matrix}, this case was already proven. Hence, we may assume that $u_1\in V_0(G),u_2\in\overline{V(G)}$. Without loss of generality, we consider the edge $e=(u,v)$, where $u^{\star},v^{\star}$ are as defined in the statement. Note that the orbit of $e$ is $\{\{u^{\star},\delta v^{\star}\}:\delta\in\Gamma\}$. Let $r$ be the row of $e$ in $O_j(G,\psi_{\gamma},p)$. 
    
    The map $\tilde{m}$ satisfies the conditions of being an infinitesimal motion of the framework on the subgraph induced by the elements of the orbit $e$ if and only if, for all $\delta\in\Gamma$
    \begin{equation*}
        \left<\tilde{p}(u^{\star})-\tilde{p}(\delta v^{\star}),\tilde{m}(u^{\star})-\tilde{m}(\delta v^{\star})\right>=0.
    \end{equation*}
    Since $\delta$ runs through all the elements of $\Gamma$, so does $\delta\gamma$. Hence, this is equivalent to saying that, for all $\delta\in\Gamma$,
    \begin{equation*}
        \left<\tilde{p}(u^{\star})-\tilde{p}(\delta\gamma v^{\star}),\tilde{m}(u^{\star})-\tilde{m}(\delta\gamma v^{\star})\right>=0.
    \end{equation*}
    Since $u$ is fixed, $\tilde m(u^{\star})=\tilde m(\delta u^{\star})$ and so, by the definitions of $\tilde{m}$ and a $\tau(\Gamma)$-symmetric framework, this is equivalent to saying that, for all $\delta\in\Gamma$,
    \begin{equation*}
        \left<p_u-\tau(\delta\gamma)p_v,\overline{\rho_j(\delta)}\tau(\delta)M_u^jm(u)-\overline{\rho_j(\delta\gamma)}\tau(\delta\gamma)m(v)\right>=0.
    \end{equation*}
    (If $M_u^j$ is not defined, we ignore terms involving $M_u^j$.) This is equivalent to saying that, for all $\delta\in\Gamma$,
    \begin{equation*}
        \overline{\rho_j(\delta)}\left(\left<p_u-\tau(\delta\gamma)p_v,\tau(\delta)M_u^jm(u)\right>+\left<\tau(\delta\gamma)p_v-p_u,\overline{\rho_j(\gamma)}\tau(\delta\gamma)m(v)\right>\right)=0.
    \end{equation*}
    Notice that, since $u$ is fixed, $p_u=\tau(\delta)p_u$ for all $\delta\in\Gamma$. Hence, since each $\tau(\delta)$ is an orthogonal matrix, we may remove the $\tau(\delta)$'s from the inner products, and multiply each equation by $\rho_j(\delta)$, to see that this set of equations holds if and only if
    \begin{equation*}
        \left<p_u-\tau(\gamma)p_v,M_u^jm(u)\right>+\left<\tau(\gamma)p_v-p_u,\overline{\rho_j(\gamma)}\tau(\gamma)m(v)\right>=0.
    \end{equation*}
    Similarly, since $u$ is fixed, $p_u=\tau(\gamma)p_u$, and so we may remove $\tau(\gamma)$ from the second inner product, and move the factor of $\overline{\rho_j(\gamma)}$ in the second inner product to the left, to obtain the equivalent equation
    \begin{equation*}
        \left<p_u-\tau(\gamma)p_v,M_u^jm(u)\right>+\left<\rho_j(\gamma)[p_v-p_u],m(v)\right>=0.
    \end{equation*}
    Again, since $u$ is fixed, $p_u=\tau(\gamma^{-1})p_u$, and hence the above equation  is equivalent to 
    \begin{equation*}
     [p_u-\tau(\gamma)p_v]^TM_u^jm(u)+\rho_j(\gamma)[p_v-\tau(\gamma^{-1})p_u]^Tm(v)=0,
    \end{equation*} 
    i.e. $rm=0$, as required.
\end{proof}

In Section 4.1, we saw that $O_j(G,\psi_{\gamma},p)$ and $\tilde{R}_j(\tilde{G},\tilde{p})$ have the same dimension. Then, by Lemma~\ref{lemma: kernels of orbit matrix and block equal},  the following holds.

\begin{corollary}
\label{corollary: rank of orbit matrix and blocks are equal}
    With the same notation as in Lemma~\ref{lemma: kernels of orbit matrix and block equal}, for any $0\leq j \leq k-1$, $\textrm{rank }(O_j(G,\psi_{\delta},p))$ and $\textrm{rank }(\tilde{R}_j(\tilde{G},\tilde{p}))$ coincide for all $\delta\in\Gamma$.
\end{corollary}

Recall that, when defining the gain graph associated to a symmetric graph, the gain of an edge incident to a fixed vertex and a free vertex could be chosen arbitrarily. As result of Lemma~\ref{lemma: kernels of orbit matrix and block equal} and Corollary~\ref{corollary: rank of orbit matrix and blocks are equal}, the choice of such a gain does not affect the rank of any of the orbit rigidity matrices.
For convenience, we usually choose gain $\textrm{id}$ for edges incident to fixed vertices.

\section{Necessity of the sparsity conditions} \label{sec:nec}

Let $(G,\psi)$ be a $\Gamma$-gain graph. A \textit{switching} at a free vertex $v$ with $\gamma\in\Gamma$ is an operation that generates a new function $\psi':E(G)\rightarrow\Gamma$ by letting $\psi'(e)=\gamma\psi(e)$ if $e$ is a non-loop edge directed from $v$, $\psi'(e)=\psi(e)\gamma^{-1}$ if $e$ is a non-loop edge directed to $v$, and $\psi'(e)=\psi(e)$ otherwise. We say two maps $\psi,\psi':E(G)\rightarrow\Gamma$ are \textit{equivalent} if one can be obtained from the other by applying a sequence of switchings, and/or by changing the gains of the edges incident to a fixed and a free vertex (recall the definition of equivalence given in Section~\ref{sec:gg}). In the proof of Lemma 5.2 in \cite{jzt2016}, it was shown that the rank of the fully-symmetric orbit  matrix is invariant under switchings whenever $V_0(G)=\emptyset$. Proposition 5.2 in \cite{bt2015} states that this is also true for all phase-symmetric  orbit matrices whenever $V_0(G)=\emptyset$. Since the proof is not explicitly given in \cite{bt2015}, we include it here for completeness, and we drop the restriction on $V_0(G)$. Together with Corollary~\ref{corollary: rank of orbit matrix and blocks are equal} this will then show that the rank of any orbit matrix is invariant under equivalence.

\begin{proposition}
\label{prop: switching maintains rank}
Let $|\Gamma|=k$, let $(G,\psi,p)$ be a $\tau(\Gamma)$-gain framework and let $\gamma\neq\textrm{id}\in\Gamma$. Let $\psi'$ be obtained from $\psi$ by applying a switching at a free vertex $v$ with $\gamma$. Let $p':V(G)\rightarrow\mathbb{R}^2$ be defined by $p'_v=\tau(\gamma)p_v$ and $p'_u=p_u$ for all $u\neq v$ in $V(G)$. Then for all $1\leq j=0\leq k-1$,
\begin{center}
    $\textrm{rank }O_j(G,\psi,p)=\textrm{rank }O_j(G,\psi',p').$
\end{center}
\end{proposition}

\begin{proof}
    Clearly, any edge non-incident with $v$  has the same row in $O_j(G,\psi,p)$ as in $O_j(G,\psi',p')$.

    Let $e:=(u,v)\in E(G)$ with $\psi(e)=\delta$ for some $\delta\in\Gamma$ and, for $0\leq j\leq k-1$, let $r_j$ be the row representing $e$ in $O_j(G,\psi',p')$. Notice that $\psi'(e)=\psi(e)\gamma^{-1}=\delta\gamma^{-1}$. Then, $$\tau(\psi'(e))p'_v=\tau(\delta\gamma^{-1})\tau(\gamma) p_v=\tau(\delta)p_v,$$ and so, if $u\neq v$,
    \begin{equation*}
    \begin{split}
        r_j&=\begin{pmatrix}
        \dots & (p_u-\tau(\delta)p_v)^TM_{u}^j & \dots & \rho_j(\delta\gamma^{-1})[\tau(\gamma)p_v-\tau((\delta\gamma^{-1})^{-1})p_u)]^T & \dots
    \end{pmatrix}\\
    &=\begin{pmatrix}
        \dots & (p_u-\tau(\delta)p_v)^TM_{u}^j & \dots & \rho_j(\gamma^{-1})\rho_j(\delta)[\tau(\gamma)(p_v-\tau(\delta^{-1})p_u)]^T & \dots
    \end{pmatrix}
    \end{split}
    \end{equation*}
   whenever $M_u^j$ is defined. (If $M_u^j$ is not defined, then there are no columns representing $u$; recall Section~\ref{sec:orbitmatrices}.) If $u=v$, then $\psi'(\delta)=\psi(\delta)$, and so
   \begin{equation*}
   \begin{split}
    r_j&=\begin{pmatrix}
        \dots & [\tau(\gamma)p_v-\tau(\delta\gamma)p_v+\rho_j(\delta)\tau(\gamma)p_v-\rho_j(\delta)\tau(\delta^{-1}\gamma)p_v]^T & \dots 
    \end{pmatrix}\\
    &=\begin{pmatrix}
        \dots & [p_v-\tau(\delta)p_v+\rho_j(\delta)p_v-\rho_j(\delta)\tau(\delta^{-1})p_v]^T\tau(\gamma)^T & \dots 
    \end{pmatrix}
    \end{split}
    \end{equation*}
    Multiply each row representing a loop at $v$ by the scalar $\rho_j(\gamma^{-1})$. Let $s$ be the number of columns of $O_j(G,\psi,p)$, and $t,t+1$ be the columns representing $v$ in $O_j(G,\psi,p)$. Define $A$ to be the square matrix of dimension $s$ such that the $2\times2$ submatrix with entries $A_{t,t},A_{t,t+1},A_{t+1,t},A_{t+1,t+1}$ is $\rho_j(\gamma)\tau(\gamma)$, all other diagonal entries of $A$ are 1, and all other entries 0. Then, $O_j(G,\psi',p')A=O_j(G,\psi,p)$. Since $A$ is an orthogonal matrix, this implies that $\textrm{rank }O_j(G,\psi,p)=\textrm{rank }O_j(G,\psi',p')$, as required.
\end{proof}

In addition, one can easily generalise the proof of Proposition 2.3 and Lemma 2.4 in \cite{jzt2016} to show Lemma~\ref{lemma: forests can have identity gain} (see \cite{PhDThesis} for details). 

\begin{lemma} 
\label{lemma: forests can have identity gain}
Let $(G,\psi)$ be a $\Gamma$-gain graph. 
\begin{itemize}
    \item[(i)] For any forest $T$ in $E(G)$, there is some $\psi'$ equivalent to $\psi$ such that $\psi'(e)=\textrm{id}$ for all $e\in T$.
    \item[(ii)] A subgraph $H$ of $G$ is balanced if and only if there is a gain $\psi'$ equivalent to $\psi$ such that $\psi'(e)=\textrm{id}$ for all $e\in E(H)$. 
\end{itemize}
\end{lemma}

\subsection{Frameworks symmetric with respect to a group of order 2}
\label{sec:nec(2)}

We first consider reflection symmetry $\mathcal{C}_s$.
\begin{proposition}
\label{Necessary conditions for reflection}
Let $(\Tilde{G},\Tilde{p})$ be a $\mathcal{C}_s$-symmetric framework with $\mathcal{C}_s$-gain framework $(G,\psi,p)$. Then:
\begin{enumerate}
    \item[(1)] If $(\Tilde{G},\Tilde{p})$ is fully-symmetrically isostatic, then $(G,\psi)$ is $(2,1,3,1)$-gain tight.
    \item[(2)] If $(\Tilde{G},\Tilde{p})$ is anti-symmetrically isostatic, then $(G,\psi)$ is $(2,1,3,2)$-gain-tight.
\end{enumerate}
\end{proposition}

\begin{proof}
If $(\Tilde{G},\Tilde{p})$ is fully-symmetrically isostatic, then $\textrm{null}(O_0(G,\psi,p))=1$, by Proposition~\ref{prop:trivmotions}(i).  Since $O_0(G,\psi,p)$ has dimension $2|\overline{V(G)}|+|V_0(G)|$, by the rank-nullity theorem, we deduce that $|E(G)|=2|\overline{V(G)}|+|V_0(G)|-1$. Moreover, there is no subgraph $(H,\psi|_{E(H)})$ of $(G,\psi)$ such that $|E(H)|>2\left|\overline{V(H)}\right|+\left|V_0(H)\right|-1,$ as this would imply a row dependency in the orbit matrix. 

Similarly, by Proposition~\ref{prop:trivmotions}(i), if $(\Tilde{G},\Tilde{p})$ is anti-symmetrically isostatic, then $|\overline{E(G)}|=2|\overline{V(G)}|+|V_0(G)|-2$ and 
$|\overline{E(H)}|\leq2\left|\overline{V(H)}\right|+\left|V_0(H)\right|-2,$ for all subgraphs $(H,\psi|_{E(H)})$ of $(G,\psi)$. 

Now, let $j=0,1$ and suppose by contradiction that $(\Tilde{G},\Tilde{p})$ is $\rho_j$-symmetrically isostatic and there is a balanced subgraph $(H,\psi|_{E(H)})$ of $(G,\psi)$ such that $|E(H)|>2|V(H)|-3$. Let $M$ be the submatrix of $O_j\left(G,\psi,p\right)$ obtained by removing all columns corresponding to the elements of $V(G)\setminus V(H)$, together with the rows corresponding to their incident edges. By Corollary~\ref{corollary: rank of orbit matrix and blocks are equal}, Proposition~\ref{prop: switching maintains rank} and Lemma~\ref{lemma: forests can have identity gain}(ii), we can assume that $\psi(e)=\textrm{id}$ for all $e\in E(H)$. 

$M$ is a submatrix of a standard rigidity matrix for a graph $F$ with $|E(F)|>2|V(F)|-3$, obtained by removing zero or more columns (depending on $|V_0(H)|$; one column is removed for each vertex in $V_0(H)$). But, by row independence, we must have $|E(F)|\leq2|V(F)|-3$, a contradiction. Hence, the result holds. 
\end{proof}

The proof of the following result for half-turn symmetry $\mathcal{C}_2$  is completely analogous to that of Proposition~\ref{Necessary conditions for reflection}, and uses Proposition~\ref{prop:trivmotions} (ii). 

\begin{proposition}
\label{necessary conditions for 2-fold rotation}
Let $(\Tilde{G},\Tilde{p})$ be a $\mathcal{C}_2$-symmetric framework with $\mathcal{C}_2$-gain framework $(G,\psi,p).$ Then:
\begin{enumerate}
    \item[(1)] If $(\Tilde{G},\Tilde{p})$ is fully-symmetrically isostatic, then $\left(G,\psi\right)$ is $(2,0,3,1)$-gain-tight.
    \item[(2)] If $(\Tilde{G},\Tilde{p})$ is anti-symmetrically isostatic, then $\left(G,\psi\right)$ is $(2,2,3,2)$-gain-tight.
\end{enumerate}
\end{proposition}

\subsection{Higher order rotational-symmetric frameworks}
In this subsection, we consider the case where $k\geq3$. First, recall  that near-balancedness is only defined for a graph with no fixed vertices, so we may directly use the following result from \cite{Ikeshita}.

\begin{lemma} (Lemma 5.5 in \cite{Ikeshita})
\label{lemma on near-balanced necessity}
    Let $k:=|\Gamma|\geq4,0\leq j\leq k-1$, $\tau:\Gamma\rightarrow\mathcal{C}_k$ be a faithful representation, $(G,\psi)$ be a $\Gamma$-gain graph, and $p:V(G)\rightarrow\mathbb{R}^2$. If $O_j(G,\psi,p)$ is row independent, $|E(H)|\leq2|V(H)|-1$ for any near-balanced subgraph $H$ of $G$.
\end{lemma}

We now give necessary conditions for the infinitesimal rigidity of $\mathcal{C}_k$-symmetric frameworks, where $k\geq3$. Recall that $\mathbb{Z}^j_k$-gain sparsity was defined in Definition~\ref{definition of sparsity, higher order rotation}.

\begin{proposition}
\label{necessary conditions for higher rotation}
For $k\geq3$, let $(\Tilde{G},\Tilde{p})$ be a $\mathcal{C}_k$-symmetric framework with $\mathcal{C}_k$-gain framework $(G,\psi,p).$ Then,
\begin{enumerate}
    \item[(1)] If $(\Tilde{G},\Tilde{p})$ is fully-symmetrically isostatic, then $(G,\psi)$ is $(2,0,3,1)$-gain tight.
    \item[(2)] If $(\Tilde{G},\Tilde{p})$ is $\rho_1$-symmetrically isostatic or $\rho_{k-1}$-symmetrically isostatic, then $(G,\psi)$ is $(2,1,3,1)$-gain tight.
    \item[(3)] If $k\geq4$ and $(\Tilde{G},\Tilde{p})$ is $\rho_j$-symmetrically isostatic for some $2\leq j\leq k-2$, then $(G,\psi)$ is $\mathbb{Z}^j_k$-gain tight. 
\end{enumerate}
\end{proposition}

\begin{proof}
The proof of (1) is the same as the proof of Proposition~\ref{necessary conditions for 2-fold rotation}(1), so we will only prove (2) and (3), starting with (2). Since the proofs for $\rho_1$-symmetrically isostatic and $\rho_{k-1}$-symmetrically isostatic frameworks are the same, we will only look at the former case.

So, suppose $(\Tilde{G},\Tilde{p})$ is $\rho_1$-symmetrically isostatic. Using the rank-nullity theorem, together with Proposition~\ref{prop:trivmotions} (iii), we see that $|E(G)|=2|\overline{V(G)}|+|V_0(G)|-1$, and that $|E(H)|\leq2|\overline{V(H)}|+|V_0(H)|-1$ for all subgraphs $H$ of $G$ with $E(H)\neq\emptyset$. Assume, by contradiction, that there is a balanced subgraph $\left(H,\psi|_{E(H)}\right)$ of $(G,\psi)$ such that $|E(H)|>2|V(H)|-3.$ Let $M$ be the submatrix of $O_1\left(G,\psi,p\right)$ obtained by removing all the columns representing the vertices that are not in $V(H)$, together with the rows corresponding to their incident edges. By Corollary~\ref{corollary: rank of orbit matrix and blocks are equal}, Proposition~\ref{prop: switching maintains rank} and Lemma~\ref{lemma: forests can have identity gain}(ii), we may assume that $\psi(e)=\textrm{id}$ for all $e\in E(H).$ If $V(H)=\emptyset$, $M$ is a standard rigidity matrix for a graph $F$ with $|E(F)|>2|V(F)|-3$, contradicting the row independence of $O_1(G,\psi,p)$. So, we may assume that $V_0(H)=\{v_0\}$. Let $v_1,\dots,v_t$ be the vertices that are incident with $v_0$ in $H$ and, for $1\leq i\leq t$, let $p_i:=p(v_i)=\begin{pmatrix}x_i&y_i\end{pmatrix}^T$.
Then, $M$ has the form

\begin{center}
$\left(\begin{array}{@{}c|c@{}}
\hspace{2mm}\begin{matrix}
-x_1+iy_1\\
\vdots\\
-x_t+iy_t
\end{matrix}
& \begin{matrix}
p_1 & \dots & 0 &\\
\vdots  & \ddots & \vdots &\\
0 & \dots & p_t &
\end{matrix}
\hspace{0.1mm}\vline\hspace{0.1mm}
\begin{matrix}
& 0 & \dots & 0\\
& \vdots & \ddots & \vdots\\
& 0 & \dots & 0
\end{matrix}\\
\cmidrule[0.4pt]{1-2}
\hspace{2mm}\begin{matrix}
0\\
\vdots\\
0\end{matrix}&\begin{matrix}
     & \vdots & & \vdots\\
     \ldots & p_i-p_j & \dots & p_j-p_i & \dots\\
     & \vdots & & \vdots
\end{matrix}
\end{array}\right)$.
\end{center}
Let $M'$ be the matrix obtained from $M$ by replacing the first column with the following two columns:
\begin{equation*}
    \begin{pmatrix}
    x_1 & y_1 \\
    \vdots & \vdots \\
    x_t & y_t \\
    0 & 0\\
    \vdots & \vdots\\
    0 & 0 
\end{pmatrix}.
\end{equation*}
Since $M$ is row independent, so is $M'$. But $M'$ is a standard rigidity matrix for a graph $F$ with $|E(F)|>2|V(F)|-3$, contradicting the row independence of $O_1(G,\psi,p)$. This proves (2). 

For (3), let $2\leq j\leq k-2$ and assume $(\tilde{G},\tilde{p})$ is $\rho_j$-symmetrically isostatic. By the rank-nullity theorem and by Proposition~\ref{prop:trivmotions}(iii), $|E(G)|=2|\overline{V(G)}|$ and $|E(H)|\leq2|\overline{V(H)}|$ for all subgraphs $H$ of $G$ with $E(H)\neq\emptyset$. The same argument as that in the proof of Proposition~\ref{necessary conditions for 2-fold rotation} also shows that all balanced subgraphs $H$ of $G$ must satisfy $|E(H)|\leq2|V(H)|-3$. By Lemma~\ref{lemma on near-balanced necessity}, all near-balanced subgraphs $H$ of $G$ must satisfy $|E(H)|\leq2|V(H)|-1$. So we only need to consider the subgraphs of $G$ which are $S(k,j)$ and, in the case where $j$ is odd, the subgraphs $H$ of $G$ with $\left<H\right>\simeq\mathbb{Z}_2$.

So, suppose that $H$ is a subgraph of $G$ with $\left<H\right>\simeq\mathbb{Z}_n$ for some $n\in S_0(k,j)\cup S_{-1}(k,j)\cup S_1(k,j)\cup\{2\}$, where $n=2$ only if $j$ is odd. Recall that $\mathbb{Z}_k\simeq\Gamma$ with the isomorphism mapping $1$ to $\gamma$. So the group $\left<H\right>$ is the group $\Gamma'$ of order $n$ generated by $\gamma^{k/n}$. Moreover, $j\equiv i(\bmod n)$, where $i=0$ if $n\in S_0(k,j)$ and $i=\pm1$ otherwise. Hence, there is some integer $m\geq1$ such that $j=i+mn$. 

Let $\rho_i'$ be the irreducible representation of $\Gamma'$ which sends the generator $\gamma^{k/n}$ to $\exp\Big(\frac{2\pi i\sqrt{-1}}{n}\Big)$, and let $\tau':\Gamma'\rightarrow\mathcal{C}_n$ be the homomorphism which sends $\gamma^{k/n}$ to the rotation $C_n$. Let $e=(u,v)\in E(H)$. Then, $\psi(e)=\gamma^{sk/n}$ for some $0\leq s\leq n-1$. Since $j=i+mn$, we have
\begin{equation*}
    \rho_j(\psi(e))=\textrm{exp}\Big(\frac{2\pi(i+mn)\sqrt{-1}}{k}\frac{sk}{n}\Big)=\textrm{exp}\Big(\frac{2\pi i\sqrt{-1}}{n}s\Big)\textrm{exp}\Big(2\pi ms\sqrt{-1}\Big)=\textrm{exp}\Big(\frac{2\pi i\sqrt{-1}}{n}s\Big)=\rho'_i(\psi(e)).
\end{equation*}
Thus, we have 
\begin{equation*}
    p_u-\tau(\psi(e))p_v=p_u-\tau'(\psi(e))p_v
\end{equation*}
and
\begin{equation*}
    \rho_j(\psi(e))(p_v-\tau(\psi(e))^{-1}p_u)=\rho'_i(\psi(e))(p_v-\tau'(\psi(e))^{-1}p_u).
\end{equation*}

(See also the proofs of Lemma 5.4 in \cite{Ikeshita} and Lemma 6.13 in \cite{bt2015} for the free action case.) Hence, $O_j(H,\psi,p)$ is the $\rho_i'$-orbit matrix of a $\mathcal{C}_n$-symmetric framework. If $i\equiv0\bmod n$, this implies that $H$ must satisfy $|E(H)|\leq2|\overline{V(H)}|-1$ by (1) and by Proposition~\ref{necessary conditions for 2-fold rotation}(1). If $i\equiv\pm1\bmod n$, this implies that $H$ must satisfy $|E(H)|\leq2|V(H)|-2$ when $n=2$ (see Proposition~\ref{necessary conditions for 2-fold rotation}(2)), and it must satisfy $|E(H)|\leq2|\overline{V(H)}|+|V_0(H)|-1$ when $k\geq3$, by (2). This gives the result.
\end{proof}

\section{Gain graph extensions}\label{sec:red}
When proving the sufficiency of the sparsity conditions, we will use an inductive argument on the order of the gain graph. To do so, we introduce certain operations on gain graphs, called \textit{extensions}. As the name suggests, extensions add vertices to the gain graph. Each extension has an inverse operation, called \textit{reduction}. For the inductive arguments to hold, extensions must maintain the symmetry-generic isostatic properties of a gain graph, and reductions must maintain the relevant sparsity counts. In this section, we will consider the extension operations. The corresponding reductions will be considered in Section~\ref{sec:ext}.

Throughout this section, we let $(G,\psi)$ be a $\Gamma$-gain graph. We will construct a $\Gamma$-gain graph $(G',\psi')$ by applying an extension to $(G,\psi)$. Depending on the extension we are working with, we may apply restrictions on the order of $\Gamma$, in which case we will specify it.

\subsection{Adding a vertex of degree 1}

The following move will only be used to study the infinitesimal rigidity of $\mathcal{C}_s$-symmetric frameworks.
\begin{definition}
\label{defn: fix-0-extension}
A \textit{fix-0-extension} chooses a vertex $u\in V(G)$, adds a new fixed vertex $v$ to $V_0(G)$, and connects it to $u$ with a new edge $e$. We label $e$ arbitrarily, unless $u\in V_0(G)$, in which case $\psi'(e)=\textrm{id}$, and we let $\psi'(f)=\psi(f)$ for all $f\in E(G)$. The inverse operation of a fix-0-extension is called a \textit{fix-0-reduction}. See Figure~\ref{fix-0-extension image} for an illustration.
\end{definition}

\begin{figure}[H]
    \centering
    \begin{tikzpicture}
  [scale=.9,auto=left]
  \draw (2.75,0) circle (1.5cm);
  \draw(2.75,0) circle (0.15cm);
  \node[below] at (2.75,-0.15){$u$};
  \node[left] at (1.5, 1.2) {$(G,\psi)$};

  \draw[->, very thick] (4.5,0.25) -- (7.5,0.25);
  \node[above] at (6, 0.25) {\textit{Extension}};
  \draw[->, very thick] (7.5,-0.25) -- (4.5,-0.25);
  \node[below] at (6, -0.25) {\textit{Reduction}};

  \draw(9.25,0) circle (0.15cm);  
  \draw (9.25,0) circle (1.5cm);
  \draw[fill=black] (9.25,2) circle (0.15cm);
  \node[below] at (9.25,-0.15){$u$};
  \node[right] at (9.4, 2){$v$};
  \draw[->] (9.25, 2) -- (9.25, 0.15);
  \node[left] at (9.25,0.85) {$\alpha$};
  \node[left] at (12.3, 1.2) {$(G',\psi')$};
\end{tikzpicture}
    \caption{Example of a fix-0-extension, where $u$ is free and $\alpha$ is an arbitrary gain.}
    \label{fix-0-extension image}
\end{figure}
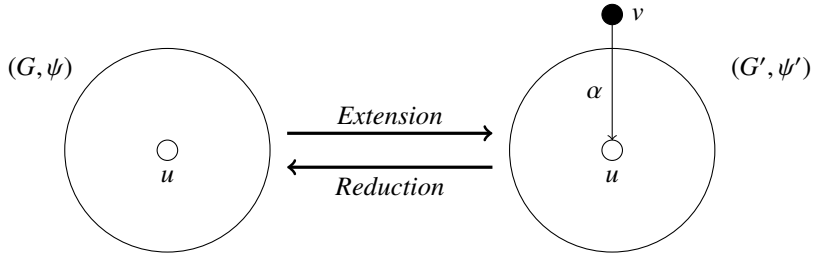

\begin{lemma}
\label{lemma: fix-0-extensions don't change the rank}
For $0\leq j\leq1$, let $(G,\psi,p)$ be a $\rho_j$-symmetrically isostatic $\mathcal{C}_s$-gain framework. Suppose $(G',\psi')$ is obtained by applying a fix-0-extension to $(G,\psi)$. Suppose further that, whenever $j=1$, the fix-0-extension from which we obtain $(G',\psi')$ connects the new fixed vertex to a free vertex. Then there is a map $p':V(G')\rightarrow\mathbb{R}^2$ such that $(G',\psi',p')$ is a $\rho_j$-symmetrically isostatic $\mathcal{C}_s$-gain framework. 
\end{lemma}

\begin{proof}
With the same notation as in Definition~\ref{defn: fix-0-extension}, we define $p':V(G')\rightarrow\mathbb{R}^2$ so that $p'_w=p_w$ for all $w\in V(G)$, $p'_v$ lies on the $y$-axis, and the $y$-coordinates of $p'_v,p_u$ differ. Let $p'_v=\begin{pmatrix}0 & y_v\end{pmatrix}$ and $p'_u=\begin{pmatrix}x_u & y_u\end{pmatrix}$. If $j=1$, then $x_u\neq0$ since $u\in\overline{V(G)}$. We may assume that $\psi(e)=\textrm{id}$, by Proposition~\ref{prop: switching maintains rank} and Lemma~\ref{lemma: forests can have identity gain}(i). We have
\begin{center}
$O_0\left(G',\psi',p'\right)=\left(\begin{array}{@{}c|c@{}}
\begin{matrix}
\text{      }y_{v}-y_u
\end{matrix}
& \begin{matrix}
\star
\end{matrix}\\
\cmidrule[0.4pt]{1-2}
\begin{matrix}
0\end{matrix}
& \begin{matrix}
O_0\left(G,\psi,p\right)
\end{matrix}
\end{array}\right)$ \hspace{3mm}and\hspace{3mm}
$O_1\left(G',\psi',p'\right)=\left(\begin{array}{@{}c|c@{}}
\begin{matrix}
\text{ }-x_u
\end{matrix}
& \begin{matrix}
\star
\end{matrix}\\
\cmidrule[0.4pt]{1-2}
\begin{matrix}
\text{  }0\end{matrix}
& \begin{matrix}
O_1\left(G,\psi,p\right)
\end{matrix}
\end{array}\right)$.
\end{center}
For $j=0,1$, we have added one row and one column to $O_j(G,\psi,p)$. Hence, it suffices to show that the rows of the new matrices are independent. This follows from the fact that $y_u\neq y_v$ and $x_u\neq0$. 
\end{proof}

\begin{remark}
    Notice that Lemma~\ref{lemma: fix-0-extensions don't change the rank} does not take into consideration the case where $j=1$ and $u\in V_0(G)$. This is because, by Proposition~\ref{Necessary conditions for reflection} (2), if $(\Tilde{G},\tilde{p})$ is a $\rho_1$-symmetrically isostatic $\mathcal{C}_s$-symmetric framework, then its $\Gamma$-gain graph $(G,\psi)$ is $(2,1,3,2)$-gain tight.  In particular, any two vertices in $V_0(G)$ cannot be joined by an edge. (Recall also Example~\ref{remark: no edges on the symmetry line}.) Hence, when proving the sufficiency of the sparsity conditions for this case, if we apply a fix-0-reduction at a fixed vertex $v\in V_0(G)$, we may assume that the vertex $v$ is adjacent to is free.
\end{remark}

\subsection{Adding a vertex of degree 2}
\begin{definition}
\label{defn: 0-extension}
A \textit{0-extension} chooses two vertices $v_1,v_2\in V(G)$ (we may choose $v_1=v_2$ provided that $v_1\in\overline{V(G)}$) and adds a free vertex $v$, together with two edges $e_1=(v,v_1),e_2=(v,v_2)$. We let $\psi'(e)=\psi(e)$ for all $e\in E(G)$. If $v_1,v_2$ coincide, we choose $\psi'$ such that $\psi'(e_1)\neq\psi'(e_2)$. In all other cases, we label $e_1,e_2$ freely. The inverse operation of a $0$-extension is called a \textit{$0$-reduction}. See Figures~\ref{0-extension image} and~\ref{0-extension image pt.2} for an illustration. 
\end{definition}

\begin{figure}[H]
    \centering
    \begin{tikzpicture}
  [scale=.9,auto=left]
  \draw (2.75,0) circle (1.5cm);
  \draw(2.75,0) circle (0.15cm);
  \node[below] at (2.75,-0.15){$v_1=v_2$};
  \node[left] at (1.5, 1.2) {$(G,\psi)$};

  \draw[->, very thick] (4.5,0.25) -- (7.5,0.25);
  \node[above] at (6, 0.25) {\textit{Extension}};
  \draw[->, very thick] (7.5,-0.25) -- (4.5,-0.25);
  \node[below] at (6, -0.25) {\textit{Reduction}};

  \draw(9.25,0) circle (0.15cm);  
  \draw (9.25,0) circle (1.5cm);
  \draw (9.25,2) circle (0.15cm);
  \node[below] at (9.25,-0.15){$v_1=v_2$};
  \node[right] at (9.4, 2){$v$};
  \draw[->] (9.25, 1.85) .. controls (9,1.115) .. (9.1, 0.13);
  \draw[->] (9.25, 1.85) .. controls (9.5,1.115) .. (9.4, 0.13);
    \node[left] at (9.05,0.85) {$\alpha$};
    \node[right] at (9.4,0.85) {$\beta$};
  \node[left] at (12.3, 1.2) {$(G',\psi')$};
\end{tikzpicture}
    \caption{Example of a 0-extension where $v_1$ and $v_2$ coincide. Here we must  have $\alpha\neq \beta$.} 
    \label{0-extension image}
\end{figure}
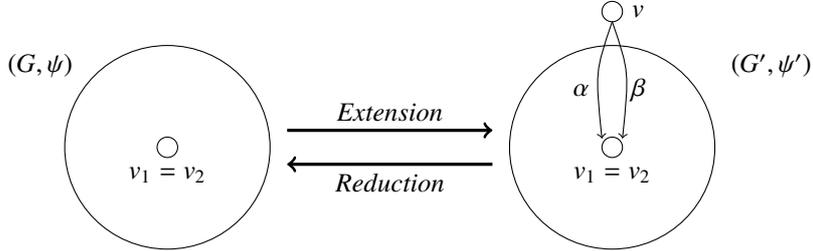

\begin{figure}[H]
    \centering
    \begin{tikzpicture}
  [scale=.9,auto=left]
  \draw (2.75,0) circle (1.5cm);
  \draw(2,0) circle (0.15cm);
  \draw[fill = black](3.5,0) circle (0.15cm);
  \node[below] at (2,-0.15){$v_1$};
  \node[below] at (3.5,-0.15){$v_2$};
  \node[left] at (1.5, 1.2) {$(G,\psi)$};

  \draw[->, very thick] (4.5,0.25) -- (7.5,0.25);
  \node[above] at (6, 0.25) {\textit{Extension}};
  \draw[->, very thick] (7.5,-0.25) -- (4.5,-0.25);
  \node[below] at (6, -0.25) {\textit{Reduction}};

  \draw(8.5,0) circle (0.15cm);
  \draw[fill = black](10,0) circle (0.15cm);
  \draw (9.25,0) circle (1.5cm);
  \draw (9.25,2) circle (0.15cm);
  \node[below] at (8.5,-0.15){$v_1$};
  \node[below] at (10,-0.15){$v_2$};
  \node[right] at (9.4, 2){$v$};
  \draw[->] (9.28, 1.88) -- (9.9, 0.15);
  \draw[->] (9.22, 1.88) -- (8.6, 0.15);
  \node[left] at (8.9,0.85) {$\alpha$};
  \node[right] at (9.65,0.85) {$\beta$};
  \node[left] at (12.3, 1.2) {$(G',\psi')$};
\end{tikzpicture}
    \caption{Example of a 0-extension where $v_1$ is free and $v_2$ is fixed. Here the gains $\alpha$ and $\beta$ are arbitrary (the cases where $v_1$ and $v_2$ are both free or both fixed are also allowed).}
    \label{0-extension image pt.2}
\end{figure}
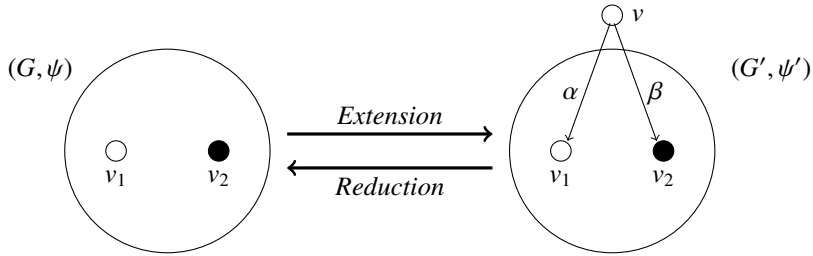

Defining $p':V(G')\rightarrow\mathbb{R}^2$ such that $p'_w=p_w$ for all $w\in V(G)$ and $p'_v$ does not lie on the line through $\tau(\psi(e_1))p(v_1)$ and $\tau(\psi(e_2))p(v_2)$, we can prove the following result in a similar way  as we proved Lemma~\ref{lemma: fix-0-extensions don't change the rank}. 

\begin{lemma}
\label{lemma 0-extension maintain rank}
Given an irreducible representation $\rho$ of $\Gamma$ and a faithful representation $\tau:\Gamma\rightarrow O(\mathbb{R}^2)$, let $(G,\psi,p)$ be a $\rho$-symmetrically isostatic $\tau(\Gamma)$-gain framework. If $(G',\psi')$ is obtained by applying a 0-extension to $(G,\psi)$, then there is a map $p':V(G')\rightarrow\mathbb{R}^2$ such that $(G',\psi',p')$ is a $\rho$-symmetrically isostatic $\tau(\Gamma)$-gain framework.
\end{lemma}

The following extension will only be used to study the infinitesimal rigidity of $\mathcal{C}_s$-symmetric frameworks.

\begin{definition}
\label{fix-1-extension defn.}
  A \textit{fix-1-extension} chooses two distinct vertices $u_1,u_2\in E(G)$ and an edge $e\in E(G)$ which can either be $(u_1,u_2)$ or, if $u_1$ (respectively, $u_2$) is free, a loop at $u_1$ (respectively, $u_2$). It removes $e$, and adds a fixed vertex $v$, together with the edges $e_1=(v,u_1),e_2=(v,u_2)$. We label $e_1$ and $e_2$ freely, and we let $\psi'(f)=\psi(f)$ for all $f\in E(G)$. The inverse operation of a fix-1-extension is called a \textit{fix-1-reduction}.  See Figure~\ref{fix-1-extension image} for an illustration.
\end{definition}

\begin{figure}[H]
    \centering
    \begin{tikzpicture}
  [scale=.9,auto=left]
  \draw (2.75,0) circle (1.5cm);
  \draw(2,0) circle (0.15cm);
  \draw(3.5,0) circle (0.15cm);
  \node[below] at (2,-0.15){$u_1$};
  \node[below] at (3.5,-0.15){$u_2$};
  \draw[->] (2.15,0) -- (3.35,0);
  \node[left] at (1.5, 1.2) {$(G,\psi)$};

  \draw[->, very thick] (4.5,0.25) -- (7.5,0.25);
  \node[above] at (6, 0.25) {\textit{Extension}};
  \draw[->, very thick] (7.5,-0.25) -- (4.5,-0.25);
  \node[below] at (6, -0.25) {\textit{Reduction}};

  \draw(8.5,0) circle (0.15cm);
  \draw(10,0) circle (0.15cm);
  \draw (9.25,0) circle (1.5cm);
  \draw[fill=black] (9.25,2) circle (0.15cm);
  \node[below] at (8.5,-0.15){$u_1$};
  \node[below] at (10,-0.15){$u_2$};
  \node[right] at (9.4, 2){$v$};
  \draw[->] (9.28, 1.88) -- (9.9, 0.15);
  \draw[->] (9.22, 1.88) -- (8.6, 0.15);
    \node[left] at (8.85,0.85) {$\alpha$};
    \node[right] at (9.65,0.85) {$\beta$};
  \node[left] at (12.3, 1.2) {$(G',\psi')$};
\end{tikzpicture}
    \caption{Example of a fix-1-extension, where $\alpha$ and $\beta$ are arbitrary gains. The vertices $u_1,u_2$ are allowed to be fixed, although for a $\rho_1$-symmetrically isosatic framework, there is no edge joining fixed vertices 
    (recall  Example~\ref{remark: no edges on the symmetry line}).} 
    \label{fix-1-extension image}
\end{figure}
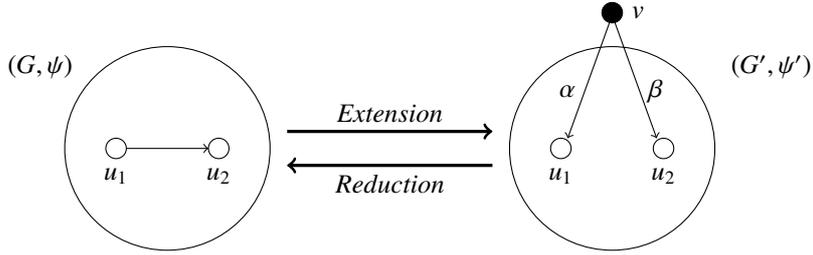

\begin{lemma}
\label{lemma fix-1-extension maintain rank}
Let $\Gamma=\left<\gamma\right>$ be the cyclic group of order 2, $0\leq j\leq1$, and let $(G,\psi,p)$ be a $\rho_j$-symmetrically isostatic $\mathcal{C}_s$-gain framework. Let $(G',\psi')$ be obtained by applying a fix-1-extension to $(G,\psi)$. With the same notation as in Definition~\ref{fix-1-extension defn.}, assume that if $e=(u_1,u_2)$, then the line through $p(u_1)$ and $\tau(\psi(e))p(u_2)$ and the line through $\sigma p(u_1)$ and $\sigma\tau(\psi(e))p(u_2)$ meet in at least one point. Assume further that if $e$ is a loop, then $p(u_1),p(u_2)$ do not share the same $y$-coordinate. Then there is a map $p':V(G')\rightarrow\mathbb{R}^2$ such that $(G',\psi',p')$ is a $\rho_j$-symmetrically isostatic $\mathcal{C}_s$-gain framework.
\end{lemma}

\begin{proof}
    Throughout the proof, we use the same notation as that in Definition~\ref{fix-1-extension defn.} and, for $1\leq i\leq2$, we let $x_i$ and $y_i$ be, respectively, the $x$-coordinate and $y$-coordinate of $p(u_i)$. We let $H$ be the subgraph obtained from $G$ by removing $e$. Since $v$ is fixed, we may assume that $\psi(e_1)=\psi(e_2)=\text{id}$.
    
    We first show the result holds when $e$ is a loop. Assume, without loss of generality, that $u_1$ is free and that $e$ is a loop at $u_1$, and notice that $\psi(e)=\gamma$. By assumption, $y_1-y_2\neq0$. Moreover, since $(G,\psi)$ is $\rho_j$-symmetrically isostatic and $G$ contains a loop edge, we know that $j=0$ (recall Example~\ref{remark: no edges on the symmetry line}). Define $p':V(G')\rightarrow\mathbb{R}^2$ such that $p'_u=p_u$ for all $u\in V(G)$ and $p'_v$ be the mid-point of the segment between $p(u_1)$ and $\sigma p(u_1)$. Then, $p'_v$ lies on the $y$-axis and has $y$-coordinate $y_1$, so that
    \begin{equation*}
        O_0(G',\psi',p')=\left(\begin{array}{@{}c|c@{}}
\hspace{2.5mm}\begin{matrix}
0\\
y_1-y_2
\end{matrix}
& \begin{matrix}
x_1 & 0 &  0\\
0 & 0 & [p(u_2)-p'_v]^TM_{u_2}^0
\end{matrix}\\
\cmidrule[0.4pt]{1-2}
\hspace{2.5mm}\begin{matrix}
0\end{matrix}
& \begin{matrix}
O_0\left(H,\psi|_{E(H)},p|_{V(H)}\right)
\end{matrix}
\end{array}\right).
    \end{equation*}
    Multiplying the first row by $4$, we obtain the row corresponding to $e$ which, added to the bottom right block, forms $O_0(G,\psi,p)$. Since $O_0(G',\psi',p')$ is obtained by adding one row and one column to $O_0(G,\psi,p)$, it suffices to show that the additional row does not add a dependence. This follows from the fact that $y_1-y_2\neq0$. Hence, the result holds whenever $e$ is a loop.

    Now, assume that $e=(u_1,u_2)$. Let $t:=0$ if $\psi(e)=\gamma$ and $t:=1$ if $\psi(e)=\textrm{id}$. Since the line through $p(u_1)$ and $\tau(\psi(e))p(u_2)$ and the line through $\sigma p(u_1)$ and $\sigma\tau(\psi(e))p(u_2)$ meet, they must meet in a point $P$ that lies on the $y$-axis. Simple calculations show that the $y$-coordinate of $P$ is 
    \begin{equation}
    \label{values of P}
        y=-\frac{y_1-y_2}{x_1+(-1)^{t}x_2}x_1+y_1=(-1)^{t+1}\frac{y_2-y_1}{x_1+(-1)^{t}x_2}x_2+y_2.
    \end{equation}
Define $p':V(G')\rightarrow\mathbb{R}^2$ such that $p'_u=p_u$ for all $u\in V(G)$ and $p'_v=P$. Then, we have

\begin{center}
    $O_j(G',\psi',p')=\left(\begin{array}{@{}c|c@{}}
\hspace{2.5mm}\begin{matrix}
\begin{pmatrix}
-x_1 & y-y_1 
\end{pmatrix}M_v^j\\
\begin{pmatrix}
-x_2 & y-y_2 
\end{pmatrix}M_v^j
\end{matrix}
& \begin{matrix}
[p(u_1)-P]^TM_{u_1}^j &  0\\
0 & [p(u_2)-P]^TM_{u_2}^j
\end{matrix}\\
\cmidrule[0.4pt]{1-2}
\hspace{2.5mm}\begin{matrix}
0\end{matrix}
& \begin{matrix}
O_j\left(H,\psi|_{E(H)},p|_{V(H)}\right)
\end{matrix}
\end{array}\right)$.
\end{center}

So, multiplying the row corresponding to $e_i$ by $\frac{x_1+(-1)^{t}x_2}{x_i}$ for $1\leq i\leq2$, and using (\ref{values of P}), we see that $O_j(G',\psi',p')$ is
\begin{center}
$\left(\begin{array}{@{}c|c@{}}
\begin{matrix}
\begin{pmatrix}
-x_1+(-1)^{t+1}x_2 & y_2-y_1 
\end{pmatrix}M_v^j\\
\begin{pmatrix}
-x_1+(-1)^{t+1}x_2 & (-1)^{t+1}(y_2-y_1) 
\end{pmatrix}M_v^j
\end{matrix}
& \begin{matrix}
x_1+(-1)^{t}x_2 & y_1-y_2 & 0 & 0\\
0 & 0 & x_1+(-1)^{t}x_2 & (-1)^{t}(y_2-y_1)
\end{matrix}\\
\cmidrule[0.4pt]{1-2}
\begin{matrix}
0\end{matrix}
& \begin{matrix}
O_j\left(H,\psi|_{E(H)},p|_{V(H)}\right)
\end{matrix}
\end{array}\right)$
\end{center}
where the first column corresponding to $u_1$ (respectively, $u_2$) in $O_0(G',\psi',p')$ vanishes if $u_1$ (respectively, $u_2$) is fixed, and the second column corresponding to $u_1$ (respectively, $u_2$) in $O_1(G',\psi',p')$ vanishes if $u_1$ (respectively, $u_2$) is fixed. Apply the following row operations: if $j=t=0$, add the second row to the first; in all other cases, subtract the second row from the first. Then, we obtain the row corresponding to $e$ which, added to the bottom right block, forms $O_j(G,\psi,p)$. 
Similarly as in the case where $e$ is a loop, it suffices to show that the second row does not add a dependence to $O_j(G,\psi,p)$. This follows from the fact that the line through $p(u_1)$ and $\tau(\psi(e))p(u_2)$ and the line through $\sigma p(u_1)$ and $\sigma\tau(\psi(e))p(u_2)$ meet at a point, which implies that the entry in the leftmost column is not zero.
\end{proof}

\subsection{Adding a vertex of degree 3}
\begin{definition}
\label{defn: loop-1-extension}
A \textit{loop-1-extension} adds a free vertex $v$ to $V(G)$ together with an edge $e=(v,u)$ for some $u\in V(G)$ and a loop $e_{L}=(v,v)$. We let $\psi'(f)=\psi(f)$ for all $f\in E(G)$. $\psi(e_{L})$ can be any non-identity element of $\Gamma$ and $\psi(e)$ can be chosen freely. The inverse operation of a loop-1-extension is called a \textit{loop-1-reduction}. See Figure~\ref{loop-1-extension image} for an illustration.
\end{definition}

\begin{figure}[H]
    \centering
    \begin{tikzpicture}
  [scale=.9,auto=left]
  \draw (2.75,0) circle (1.5cm);
  \draw(2.75,0) circle (0.15cm);
  \node[below] at (2.75,-0.15){$u$};
  \node[left] at (1.5, 1.2) {$(G,\psi)$};

  \draw[->, very thick] (4.5,0.25) -- (7.5,0.25);
  \node[above] at (6, 0.25) {\textit{Extension}};
  \draw[->, very thick] (7.5,-0.25) -- (4.5,-0.25);
  \node[below] at (6, -0.25) {\textit{Reduction}};

  \draw(9.25,0) circle (0.15cm);  
  \draw (9.25,0) circle (1.5cm);
  \draw (9.25,2) circle (0.15cm);
  \node[below] at (9.25,-0.15){$u$};
  \node[right] at (9.4, 2){$v$};
  \draw[->] (9.25, 1.85) -- (9.25, 0.15);
  \node[left] at (9.25,0.85) {$\beta$};
  \draw[->] (9.4,2) .. controls (9.8,2.5) and (8.7,2.5) .. (9.1,2); 
  \node[left] at (9,2.2) {$\alpha$};
  \node[left] at (12.3, 1.2) {$(G',\psi')$};
\end{tikzpicture}
    \caption{Example of a loop-1-extension, where $\alpha$ is a non-identity gain, and $\beta$ is an arbitrary gain. (The case where $u$ is fixed is also allowed).}
    \label{loop-1-extension image}
\end{figure}
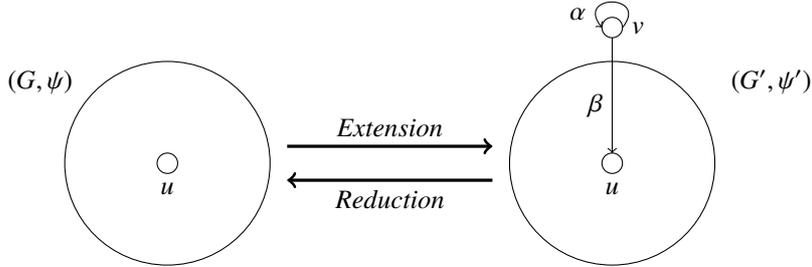

\begin{lemma}
\label{lemma: loop-1-extensions don't change the rank}
Let $\Gamma$ be a cyclic group of order $k\geq2$, $\tau:\Gamma\rightarrow O(\mathbb{R}^2)$ be a faithful representation and $(G,\psi,p)$ be a $\rho_j$-symmetrically isostatic $\tau(\Gamma)$-gain framework for some $0\leq j\leq k-1$ such that $j=0$ when $k=2$. Let  $\gamma\in \Gamma$ correspond to the $k$-fold rotation (or the reflection) under $\tau$. Let $(G',\psi')$ be obtained from $(G,\psi)$ by applying a loop-1-extension. With the same notation as Definition~\ref{defn: loop-1-extension}, let $g:=\psi'(e_L)$ and $h:=\psi'(e)$. Assume the following hold:
\begin{itemize}
    \item[(i)] if $k$ is even and $j$ is odd, then $g\neq\gamma^{k/2}$;
    \item[(ii)] if $\tau(\Gamma)=\mathcal{C}_k$ and $j=0$, then $u$ is free;
    \item[(iii)] if $k\geq4,2\leq j\leq k-2$ and $u$ is fixed, then there is no $n\in S_0(k,j)$ such that $\left<g\right>\simeq\mathbb{Z}_n$.
\end{itemize}
Then there is a map $p':V(G')\rightarrow\mathbb{R}^2$ such that $(G',\psi',p')$ is a $\rho_j$-symmetrically isostatic $\tau(\Gamma)$-gain framework.
\end{lemma}

\begin{proof}
    With the same notation as in Definition~\ref{defn: loop-1-extension}, let $p':V(G')\rightarrow\mathbb{R}^2$ be defined such that $p'_w=p_w$ for all $w\in V(G)$. We have 
    \begin{center}
$O_j(G',\psi',p')=\left(\begin{array}{@{}c|c@{}}
\begin{matrix}
[I_2+\rho_j(g)I_2-\tau(g)-\rho_j(g)\tau(g^{-1})](p'_v)^T\\
[p_v'-\tau(h)p_u]^T
\end{matrix}
& \begin{matrix}
0\\
\star
\end{matrix}\\
\cmidrule[0.4pt]{1-2}
\begin{matrix}
0\end{matrix}
& \begin{matrix}
O_j\left(G,\psi,p\right)
\end{matrix}
\end{array}\right)$.
\end{center}
   So $O_j(G',\psi',p')$ is obtained from $O_j(G,\psi,p)$ by adding two rows and two columns. Since $O_j(G,\psi,p)$ has full rank by assumption, it is enough to show that the first two rows of $O_j(G',\psi',p')$ are linearly independent for some choice of $p'_v$. Let $A$ be the matrix $I_2+\rho_j(g)I_2-\tau(g)-\rho_j(g)\tau(g^{-1})$. If $\tau(\Gamma)=\mathcal{C}_s$, then $j=0$ and $\tau(g)=\sigma$ by assumption, so $A$ is the $2\times2$ matrix whose only non-zero entry is $(A)_{1,1}=4$. If we choose $p'_v$ such that it does not share the same $y$-coordinate as $p'_u$, it is then easy to see that the first two rows of the matrix are linearly independent. Hence, $O_j(G',\psi',p')$ has full rank, as required.
   
   So, we may assume that $\tau(\Gamma)=\mathcal{C}_k$. Let $g=\gamma^t$ for some $1\leq t\leq k-1$, $\alpha=\frac{2\pi t}{k},$ and $\omega=\exp \frac{2\pi i }{k}$. Then,
   \begin{equation*}
       A=\begin{pmatrix}
           (1-\cos(\alpha))(1+\omega^{jt}) & \sin(\alpha)(1-\omega^{jt}) \\
           -\sin(\alpha)(1-\omega^{jt})&
           (1-\cos(\alpha))(1+\omega^{jt})
       \end{pmatrix}.
   \end{equation*}
   We show that $A$ is not the zero matrix. Assume, by contradiction, that $A$ is the zero matrix. Since $1\leq t\leq k-1$, we know $\cos(\alpha)\neq1$ and so $\omega^{jt}=-1$, i.e. there is some odd integer $m$ such that $\frac{2\pi jt}{k}=m\pi$. Moreover, $\sin(\alpha)(1-\omega^{jt})=2\sin(\alpha)=0$. Since $1\leq t\leq k-1$, $\alpha=\pi$, and so $t=\frac{k}{2}$. It follows that $j=m$, and so $j$ is odd. This contradicts (i), so, as claimed, $A$ is not the zero matrix. 
   
   If $u$ is free, then, by the injectivity of $p$, the vector $p_u$, and hence also the vector $\tau(h)p_u$, cannot be zero. So unless $q$ is a multiple of $\tau(h)p_u$, the affine map $q\mapsto q-\tau(h)p_u$ applied to $\lambda q$ gives vectors of different directions for each scalar $\lambda$. The linear map $A$ applied to $\lambda q$, however, only produces vectors that are multiples of the vector $Aq$.
      This implies that $\{Aq,q-\tau(h)p_u\}$ is linearly independent for some $q$, and so we may choose $p'_v$ to be such $q$. 
   Then $O_j(G',\psi',p')$ is linearly independent, as required.
   
   So, assume that $u$ is fixed. By (ii), we may also assume that $1\leq j\leq k-1$. In particular, this implies, by assumption, that $k\neq2$.  Assume, by contradiction, that there is no choice of $p'_v$ such that the first two rows of $O_j(G',\psi',p')$ are linearly independent. This implies that $A$ is a scalar multiple of $I_2$. This happens exactly when $\sin(\alpha)(1-\omega^{jt})=0$. If $\sin(\alpha)=0$, then $\alpha=\pi$, and $t=\frac{k}{2}$. Hence, $\omega^{jt}=\exp(\pi ij)$. By (i), $j$ must be even, and so $\omega^{jt}=1$. If $\sin(\alpha)\neq0$, then clearly $\omega^{jt}=1$. Hence, in both cases we have $\omega^{jt}=1$, i.e. $jt=mk$ for some integer $m$. If $j=1$, this implies that $t$ is a multiple of $k$, contradicting the fact that $1\leq t\leq k-1$. Hence, $j\neq1$. Similarly, $j\neq k-1$: if $j=k-1$, then $\omega^{jt}=\omega^{-t}$. Since $\omega^{jt}$ is real, this equals $\omega^t=1$, and so $t=k$, a contradiction. Hence, $k\geq4$ and $2\leq j\leq k-2$. We show that there is an integer $n\in S_0(k,j)$ such that $\left<g\right>\simeq\mathbb{Z}_n$, contradicting (iii).

   Let $n=\frac{k}{\gcd(k,t)}=\frac{\textrm{lcm}\,(k,t)}{t}$. Then, we know from group theory (see e.g. \cite{dummit}) that $\left<g\right>=\left<\gamma^t\right>\simeq\mathbb{Z}_n$, and that $m'=\frac{mk}{\textrm{lcm}\,(k,t)}$ is an integer (since $mk$ is a multiple of both $k$ and $t$), and so, since $j=\frac{mk}{t}=nm'$, we have $j\equiv0\mod n$. Moreover, $k=n\gcd(k,t)$, so $n|k$. Hence, $n\in S_0(k,j)$, as required. This contradicts (iii). Thus, there is a choice of $p'_v$ such that the first two rows of $O_j(G',\psi',p')$ are linearly independent. It follows that $O_j(G',\psi',p')$ has full rank.    
\end{proof}

\begin{remark} It was shown in \cite[Section 4.3]{mtforced} that for the  $\Gamma$-gain graph of  a rotationally symmetric framework in the plane, an edge joining the fixed vertex $u$ with a free vertex $v$ gives the same constraint in the fully-symmetric orbit rigidity matrix as a loop edge on $v$  which corresponds  to a regular $|\Gamma|$-polygon in the covering framework. Hence we have condition (ii) in Lemma~\ref{lemma: loop-1-extensions don't change the rank}. This is clear geometrically, because both of these edges force the vertices in the orbit of $v$ to keep their distance to the origin in any symmetry-preserving motion. Thus, for analysing fully-symmetric infinitesimal rigidity, one may always reduce the problem to the case when the group acts freely on the vertices. However, Lemma~\ref{lemma: loop-1-extensions don't change the rank} shows that this simple reduction is  not possible for the reflection group nor for analysing  ``incidentally symmetric" infinitesimal rigidity for any rotational group. 
\end{remark}

\begin{definition}
\label{defn: 1-extension}
A \textit{1-extension} chooses a vertex $u\in V(G)$ and an edge $e=\left(v_1,v_2\right)\in E(G)$ (any pair of free vertices in $\{v_1,v_2,u\}$ are allowed to coincide; further, $v_1,v_2,u$ are all allowed to coincide, provided they are not fixed and $|\Gamma|\geq3$), removes $e$ and adds a new free vertex $v$ to $V(G),$ together with three edges $e_1=(v,v_1),e_2=(v,v_2),e_3=(v,u)$. We let $\psi'(f)=\psi(f)$ for all $f\in E(G)$. The edges $e_1,e_2$ are labelled such that $\psi'(e_1)^{-1}\psi'(e_2)=\psi(e)$. The label of $e_3$ is chosen such that it is locally unbalanced, i.e. every two-cycle $e_ie_j^{-1},$ if it exists, is unbalanced. The inverse operation of a $1$-extension is called a \textit{$1$-reduction}. See Figure~\ref{1-extension image} for an illustration.
\end{definition}

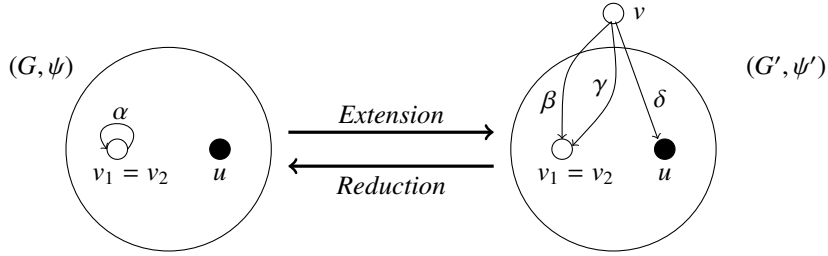
\begin{figure}[H]
    \centering
    \begin{tikzpicture}
  [scale=.9,auto=left]
  \draw (2.75,0) circle (1.5cm);
  \draw(2,0) circle (0.15cm);
  \draw[fill = black](3.5,0) circle (0.15cm);
  \node[below] at (3.5,-0.15){$u$};
  \node[left] at (1.5, 1.2) {$(G,\psi)$};
  \draw[->] (2.15,0) .. controls (2.55,0.5) and (1.45,0.5) .. (1.85,0);
  \node[below] at (2.2,-0.15){$v_1=v_2$};
  \node[above] at (2.05, 0.3){$\alpha$};

  \draw[->, very thick] (4.5,0.25) -- (7.5,0.25);
  \node[above] at (6, 0.25) {\textit{Extension}};
  \draw[->, very thick] (7.5,-0.25) -- (4.5,-0.25);
  \node[below] at (6, -0.25) {\textit{Reduction}};

  \draw(8.5,0) circle (0.15cm);  
  \draw (9.25,0) circle (1.5cm);
  \draw[fill=black] (10,0) circle (0.15cm);
  \draw (9.25, 2) circle (0.15cm);
  \node[right] at (9.4, 2){$v$};
  \node[below] at (8.7,-0.15){$v_1=v_2$};
  \node[below] at (10, -0.15){$u$};
  \node[left] at (12.5, 1.2) {$(G',\psi')$};
  \draw[->] (9.28, 1.88) -- (9.9, 0.15);
  \draw[->] (9.22,1.88) .. controls (9.35,0.8) .. (8.65, 0.05);
  \draw[->] (9.22,1.88) .. controls (8.5,1.2) .. (8.5, 0.15);
  \node[right] at (9.7,0.75) {$\delta$};
  \node[left] at (8.55,0.7) {$\beta$};
  \node[left] at (9.3,0.9) {$\gamma$};
\end{tikzpicture}
    \caption{Example of a 1-extension, where $\alpha=\beta\gamma^{-1}$ and $\delta$ is an arbitrary gain. In this example, we can see that $v_1$ and $v_2$ are allowed to coincide.}
    \label{1-extension image}
\end{figure}

\begin{lemma}
\label{lemma: 1-extension maintain rank}
Let $\Gamma$ be a cyclic group of order $k$, $\tau:\Gamma\rightarrow O(\mathbb{R}^2)$ be a faithful representation, and $(G,\psi,p)$ be a $\rho_j$-symmetrically isostatic $\tau(\Gamma)$-gain framework for some $0\leq j\leq k-1$. With the same notation as in Definition~\ref{defn: 1-extension}, assume that the points $\tau(\psi(e_1))p(v_1),\tau(\psi(e_2))p(v_2)$ and $\tau(\psi(e_3))p(u)$ do not lie on the same line. If $(G',\psi')$ is obtained from $(G,\psi)$ by applying a 1-extension, then there is a map $p':V(G')\rightarrow\mathbb{R}^2$ such that $(G',\psi',p')$ is a $\rho_j$-symmetrically isostatic $\tau(\Gamma)$-gain framework.
\end{lemma}

\begin{proof}
With the same notation as that of Definition~\ref{defn: 1-extension}, let $H$ be the subgraph of $G$ obtained by removing $e$. If $v_1,v_2$ are free, then an analogous proof to that of Lemma 6.1 in \cite{bt2015} gives the result. So, without loss of generality, assume $v_1\in V_0(G)$. In particular, $v_1$ cannot coincide with either $v_2$ or $u$, and we may assume $\psi(e_1)=\psi(e_2)=\psi(e)=\textrm{id}$. Let $\psi(e_3)=\delta$.

Define $p':V(G')\rightarrow\mathbb{R}^2$ such that $p'_w=p_w$ for all $w\in V(G)$ and $p'_v$ lies on the midpoint of the line through $p(v_1)$ and $p(v_2)$. Then, $O_j(G',\psi',p')$ is
\begin{center}
        $\left(\begin{array}{@{}c|c@{}}
\begin{matrix}
\rho_j(\delta)[p_v'-\tau(\delta)p_u]^T\\
1/2[p(v_2)-p(v_1)]^T\\
1/2[p(v_1)-p(v_2)]^T
\end{matrix}
& \begin{matrix}
 0 & \star & \star \\
1/2[p(v_1)-p(v_2)]^TM_{v_1}^j & 0 & 0 \\
 0 & 1/2[p(v_2)-p(v_1)]^TM_{v_2}^j & 0
\end{matrix}\\
\cmidrule[0.4pt]{1-2}
\begin{matrix}
0\end{matrix}
& \begin{matrix}
O_j\left(H,\psi_{E(H)},p|_{V(H)}\right).
\end{matrix}
\end{array}\right)$,
    \end{center} 
where, for $1\leq i\leq2$, the columns representing $v_i$ vanish if $M_{v_i}^j$ is not defined. Adding the second row to the third, and multiplying this latter by 2, we obtain the row representing $e$ in $O_j\left(H,\psi_{E(H)},p|_{V(H)}\right)$. Now, $O_j(G',\psi',p')$ is obtained by adding two rows and two columns to $O_j(G,\psi,p)$, so it suffices to show that the first two entries of the two added rows are independent. Since $p(v_1),p(v_2)$ and $\tau(\delta)p_u$ do not lie on the same line, the line through $p'_v$ and $\tau(\delta)p_u$ is not parallel to the line through $p(v_1)$ and $p(v_2)$. Hence, the upper left $2\times 2$ matrix has full rank, and $O_j(G',\psi',p')$ has full rank. If $\tau(\Gamma)=\mathcal{C}_s$, and $p(v_1),p(v_2)$ both lie on the symmetry line, then $p_v'$ also lies on the symmetry line. In such a case, we may perturb  $p'_v$ slightly without changing the rank of $O_j(G',\psi',p')$, in order to avoid placing the free vertex $v$ on the symmetry line. 
\end{proof}

\subsection{Adding two vertices of degree 3}
The following extension is defined on $\Gamma$-gain graphs with $|V_0(G)|=1$ and with $|\Gamma|\geq2$ even. Recall that, for $k:=|\Gamma|$, the cyclic group $\Gamma=\left<\gamma\right>$ is isomorphic to $\mathbb{Z}_k$, through the isomorphism which maps $\gamma$ to 1. 
\begin{definition}
    A \textit{2-vertex-extension} adds two free vertices $v_1,v_2$ and connects them to the fixed vertex. Then, it adds two parallel edges $e_1,e_2=(v_1,v_2)$ between $v_1$ and $v_2$. We define $\psi'$ such that $\psi'(e)=\psi(e)$ for all $e\in E(G)$, the new edges incident with the fixed vertex are labelled arbitrarily, and $\psi'(e_1)=\textrm{id},\psi'(e_1)=\gamma^{k/2}$. The inverse operation of a 2-vertex-extension is called a \textit{2-vertex-reduction}. See Figure~\ref{Example of a 2-vertex extension.} for an illustration.
\end{definition}

\begin{figure}[htp]
    \centering
    \begin{tikzpicture}
  [scale=.9,auto=left]
  \draw (2.75,0) circle (1.5cm);
  \draw[fill = black](2.75,0) circle (0.15cm);
  \node[below] at (2.75,-0.15){$v_0$};
  \node[left] at (1.5, 1.2) {$(G,\psi)$};

  \draw[->, very thick] (4.5,0.25) -- (7.5,0.25);
  \node[above] at (6, 0.25) {\textit{Extension}};
  \draw[->, very thick] (7.5,-0.25) -- (4.5,-0.25);
  \node[below] at (6, -0.25) {\textit{Reduction}};

  \draw[fill = black](9.25,0) circle (0.15cm);  
  \draw (9.25,0) circle (1.5cm);
  \draw (8.5,2) circle (0.15cm);
  \draw (10,2) circle (0.15cm);
  \node[below] at (9.25,-0.15){$v_0$};
  \node[left] at (12.3, 1.2) {$(G',\psi')$};
  \draw[->] (8.53, 1.88) -- (9.15, 0.15);
  \draw[->] (9.97, 1.88) -- (9.35, 0.15);
  \node[right] at (9.55,0.75) {$\alpha$};
  \node[left] at (8.95,0.75) {$\beta$};
  \draw[->] (8.65,2) -- (9.85,2);
  \draw[->](8.6,2.1) .. controls (9.25,2.3) ..  (9.9,2.1);
  \node[left] at (8.4,2){$v_1$};
  \node[right] at (10.1,2){$v_2$};
  \node[above] at (9.25,2.3){$\gamma^{k/2}$};
  \node[below] at (9.25,2){$\textrm{id}$};
\end{tikzpicture}
    \caption{Example of a 2-vertex extension, where $\alpha,\beta$ are arbitrary gains.}
    \label{Example of a 2-vertex extension.}
\end{figure}
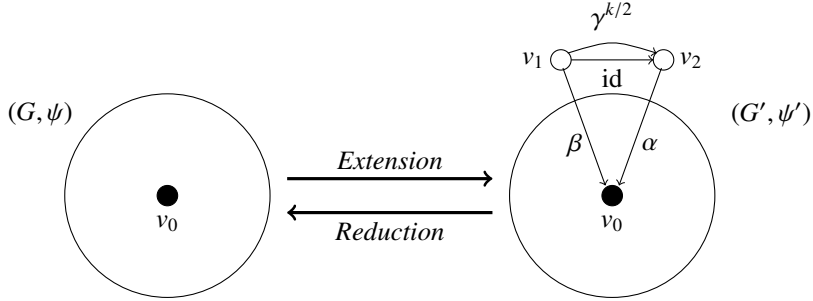

In a similar way as for Lemma~\ref{lemma: fix-0-extensions don't change the rank}, we can show that the following result holds. In this case, we define $p':V(G')\rightarrow\mathbb{R}^2$ such that $p'(v_1)$ and $p'(v_2)$ are not scalar multiples of each other. See \cite{PhDThesis} for details.

\begin{lemma}
\label{2-vertex extension maintains rank}
    Let $k\geq2$ be even. Let $(G,\psi,p)$ be a $\rho_j$-symmetrically isostatic $\mathcal{C}_k$-gain framework with $V_0(G)=\{v_0\}$. If $(G',\psi')$ is obtained by applying a 2-vertex-extension to $(G,\psi)$, then there is a map $p':V(G')\rightarrow\mathbb{R}^2$ such that $(G',\psi',p')$ is a $\rho_j$-symmetrically isostatic $\mathcal{C}_k$-gain framework.
\end{lemma}

\section{Sufficiency of the sparsity conditions}\label{sec:ext}
In this section, we will establish the characterisations of symmetry-generic infinitesimally rigid bar-joint frameworks. We first need some combinatorial preliminaries.

\subsection{General combinatorial results}
\begin{lemma}
\label{lemma: there is a vertex of degree 2 or 3}
    Let $0\leq m\leq2,1\leq l\leq2$ be such that $0\leq l-m\leq1$. 
    Let $(G,\psi)$ be a $\Gamma$-gain graph with at least one free vertex, and let $s,t\in\mathbb{N}$ be the number of free vertices in $G$ of degree 2 and 3, respectively. Assume $(G,\psi)$ is $(2,m,3,l)$-gain tight. The following hold:
    \begin{itemize}
        \item[(i)] Each free vertex has degree at least 2, and each fixed vertex has degree at least $m$.
        \item[(ii)] If each fixed vertex has degree at least $d$ for some $d\geq0$, then $2s+t\geq|V_0(G)|(d-2m)+2l$.
    \end{itemize}
\end{lemma}

\begin{proof}
    For (i), let $v\in V(G)$. By the sparsity of $(G,\psi)$, the subgraph $H$ obtained from $G$ by removing $v$ satisfies 
    \begin{equation*}
        |E(H)|\leq\begin{cases}
            2|\overline{V(G)}|+m|V_0(G)|-l-2 & \text{if }v\text{ is free}\\
            2|\overline{V(G)}|+m|V_0(G)|-l-m & \text{if }v\text{ is fixed.}
        \end{cases}
    \end{equation*}
    But $|E(G)|=2|\overline{V(G)}|+m|V_0(G)|-l$. So there are at least 2 edges in $G$ that are not in $H$ when $v$ is free, and there are at least $m$ edges in $G$ that are not in $H$ when $v$ is fixed. (i) follows.  
    
    For (ii), the average degree of $G$ is 
    \begin{equation*}
        \hat{\rho}=\frac{2|E(G)|}{|V(G)|}=\frac{4|\overline{V(G)}|+2m|V_0(G)|-2l}{|V(G)|}.
    \end{equation*}
    The minimum average degree $\rho_{min}$ of $G$ is attained when all free vertices, which are not the $s$ and $t$ vertices of degree 2 and 3, have degree 4, and all fixed vertices have degree $d$. So
    \begin{equation*}
        \rho_{min}=\frac{2s+3t+d\left|V_0\left(G\right)\right|+4(|\overline{V(G)}|-s-t)}{|V(G)|}.
    \end{equation*}
    By minimality, $\rho_{min}\leq\hat{\rho}$, and (ii) follows.
\end{proof}

\begin{proposition}
\label{lemma: there is no tight subgraph containing all three neighbours}
    Let $0\leq m\leq 2,0\leq l\leq3$, let $(G,\psi)$ be a $\Gamma$-gain graph and suppose there is some $v\in \overline{V(G)}$ of degree 3 with no incident loops. If $G$ is $(2,m,l)$-sparse, then there is no $(2,m,l)$-tight subgraph of $G-v$ which contains all neighbours of $v$ (the neighbours of $v$ need not be distinct).
\end{proposition}

\begin{proof}
    Suppose such a subgraph $H$ exists. Then the subgraph $H'$ of $G$ obtained from $H$ by adding $v$ and its incident edges satisfies
    \begin{equation*}
        |E(H')|=|E(H)|+3=2|\overline{V(H)}|+m|V_0(H)|-l+3=2|\overline{V(H')}|+m|V_0(H')|-l+1
    \end{equation*}
    a contradiction.
\end{proof}

It is straightforward to check that all except two of the reductions are \textit{admissible}, i.e. they maintain the relevant sparsity counts. However, when applying a 1-reduction or a fix-1-reduction, we add an edge. This edge might give rise to a subgraph that violates the sparsity count. We call such subgraphs blockers.

\begin{definition}
\label{blocker definition}
    Let $(G,\psi)$ be a $\Gamma$-gain graph. Let $m,l$ be non-negative integers such that $m\leq2,l\leq3,m\leq l$, and suppose $(G,\psi)$ is $(2,m,3,l)$-gain tight. Let $v\in V(G)$ be a free vertex of degree 3, or a fixed vertex of degree 2. Let $(G',\psi')$ be obtained from $(G,\psi)$ by applying a 1-reduction or a fix-1-reduction at $v$, and let $e=(v_1,v_2)$ be the edge we add when we apply such reduction. Let $H$ be a subgraph of $G-v$ with $v_1,v_2\in E(H)$ and $E(H)\neq\emptyset$. We say $H$ is
    \begin{enumerate}
        \item a \textit{general-count blocker} of $v_1,v_2$ (equivalently, of $(G',\psi')$) if $H+e$ is connected and $H$ is $(2,m,l)$-tight.
        \item a \textit{balanced blocker} of $e$ (equivalently, of $(G',\psi')$) if $H$ is $(2,3)$-tight and $H+e$ is balanced under $\psi'$.
    \end{enumerate}
    Both general-count blockers and balanced blockers are referred to as \textit{blockers} of $(G',\psi')$.
\end{definition}

The following result states that, given two blockers $H_1,H_2$ with $E(H_1\cap H_2)\neq\emptyset$, their union $H_1\cup H_2$ can also be seen as a blocker. It will be used in Section~\ref{section on 1-reduction} to show that a vertex of degree 3 always admits a 1-reduction, except for two special cases (see Theorem~\ref{theorem on 1-reductions: (2,m,3,l)-tight}).

\begin{lemma}
\label{lemma: what i need for 1-red.}
   Let $m,l$ be non-negative integers such that $m\leq2,1\leq l\leq2,m\leq l$. Let $(G,\psi)$ be a $\Gamma$-gain graph, and suppose there is some $v\in \overline{V(G)}$ of degree 3 with no incident loops. Assume further that $|V_0(G)|\leq1$ if $m=2$. Suppose $(G,\psi)$ is $(2,m,3,l)$-gain tight. Assume there are graphs $(G_1,\psi_1),(G_2,\psi_2)$ obtained from $(G,\psi)$ by applying two different $1$-reductions at $v$, which add the edges $f_1$ and $f_2$, respectively. Assume that, for $i=1,2$, $(G_i,\psi_i)$ has a blocker $H_i$, and that $E(H_1\cap H_2)\neq\emptyset$. Let $H:=H_1\cup H_2$. The following hold:
    \begin{itemize}
        \item[(i)] The blockers $H_1,H_2$ are not general-count blockers.
        \item[(ii)] $H+f_1+f_2$ is balanced and $H$ is $(2,3)$-tight.
    \end{itemize}
\end{lemma}

\begin{proof}   
    Notice that $H_1\cup H_2$ always contains all neighbours of $v$. To see this, we consider $N(v).$ If $N(v)=1$ this is clear. If $N(v)=2$, let $v_1,v_2$ be the neighbours of $v$ and $e_1=(v,v_1),e_1'=(v,v_1),e_2=(v,v_2)\in E(G)$. By Corollary~\ref{corollary: rank of orbit matrix and blocks are equal}, Proposition~\ref{prop: switching maintains rank} and Lemma~\ref{lemma: forests can have identity gain}(i), we may assume that $\psi(e_1)=\psi(e_2)=\textrm{id}$ and that $\psi(e'_1)\neq\textrm{id}$. Then, at most one 1-reduction at $v$ adds a loop at $v_1$ (with gain $\psi(e'_1)$) and no 1-reduction at $v$ adds a loop at $v_2$. It follows that one of $H_1,H_2$ contains $v_1$ and $v_2$, and so $v_1,v_2\in V(H_1\cup H_2)$. Finally, let $N(v)=3$. For $1\leq i\leq3$, let $e_i=(v,v_i)\in E(G)$.
    By Corollary~\ref{corollary: rank of orbit matrix and blocks are equal}, Proposition~\ref{prop: switching maintains rank} and Lemma~\ref{lemma: forests can have identity gain}(i) we may assume that $\psi(e_i)=\textrm{id}$ for all $1\leq i\leq3$. Then, for each pair $1\leq i\neq j\leq3$, there is at most one 1-reduction at $v$ which adds an edge between $v_i$ and $v_j$ (with gain $\textrm{id}$). It follows that $v_1,v_2,v_3\in V(H_1\cup H_2)$. 

    Throughout the proof, we let $H'=H_1\cap H_2$ and we let $H'_1,\dots, H'_c$ be the connected components of $H'$. Let $c_0\leq c-1$ be the number of isolated vertices of $H'$, so that $H'_1,\dots,H_{c_0}'$ are the isolated vertices of $H'$, and $H_{c_0+1}',\dots,H_c'$ are the connected components of $H'$ with non-empty edge set.
    
    We first prove (i). Assume, by contradiction, that one of $H_1,H_2$ is a general-count blocker. Without loss of generality, let it be $H_1$. If $H_2$ is also a general-count blocker, then, since $|E(H')|\leq2|\overline{V(H')}|+m|V_0(H')|-l$, it is easy to check that 
    \begin{equation*}
            |E(H)|=|E(H_1)|+|E(H_2)|-|E(H')|\geq2|\overline{V(H)}|+m|V_0(H)|-l.
    \end{equation*}
    By Proposition~\ref{lemma: there is no tight subgraph containing all three neighbours}, this is a contradiction.  Hence, we may assume that $H_2$ is a balanced blocker. It follows that $H'$ is balanced. Then, for each $c_0+1\leq i\leq c$, $H_i'$ must be $(2,3)$-sparse, and so
    \begin{equation*}
    \begin{split}
        |E(H')|&=\sum_{i=1}^c|E(H_i')|\leq\sum_{i=1}^{c_0}[2|V(H_i')|-2]+\sum_{i=c_0+1}^c[2|V(H_i')|-3]=2|V(H')|-(2c_0+3(c-c_0)).
    \end{split}
    \end{equation*}
     Hence, 
    \begin{equation*}
        \begin{split}
            |E(H)|&=|E(H_1)|+|E(H_2)|-|E(H')|\\
            &\geq(2|\overline{V(H_1)}|+m|V_0(H_1)|-l)+(2|V(H_2)|-3)-(2|V(H')|-(2c_0+3(c-c_0)))\\
            &=(2|\overline{V(H_1)}|+m|V_0(H_1)|-l)+(2|\overline{V(H_2)}|+2|V_0(H_2)|-3)-(2|\overline{V(H')}|+2|V_0(H')|-(2c_0+3(c-c_0)))\\
            &=(2|\overline{V(H_1)}|+m|V_0(H_1)|-l)+(2|\overline{V(H_2)}|+m|V_0(H_2)|+(2-m)|V_0(H_2)|-3)\\
            &-(2|\overline{V(H')}|+m|V_0(H')|+(2-m)|V_0(H')|-(2c_0+3(c-c_0)))\\
            &=2|\overline{V(H)}|+m|V_0(H)|-l+(2-m)(|V_0(H_2)|-|V_0(H')|)+2c_0+3(c-c_0-1)\\
            &\geq2|\overline{V(H)}|+m|V_0(H)|-l,
        \end{split}
    \end{equation*}
    where the last inequality holds because $0\leq c_0\leq c-1,m\leq2$ and $V(H')\subseteq V(H_2)$. By Proposition~\ref{lemma: there is no tight subgraph containing all three neighbours}, this is a contradiction. So $H_1,H_2$ must be balanced blockers, as required.
    
    \medskip
    We now prove (ii). By (i), $H_1,H_2$ are both balanced blockers. Hence, $H'_i$ is $(2,3)$-sparse for all $c_0+1\leq i\leq c$. It follows that $|E(H')|\leq2|V(H')|-(2c_0+3(c-c_0))$ (see the proof of (i) for details). Hence, we have
    \begin{equation}
    \label{equation (ii)}
    \begin{split}
        |E(H)|&=|E(H_1)|+|E(H_2)|-|E(H')|\\
        &\geq(2|V(H_1)|-3)+(2|V(H_2)|-3)-(2|V(H')|-(2c_0+3(c-c_0)))\\
        &=2|V(H)|+2c_0+3(c-c_0)-6=2|V(H)|+3c-c_0-6.
    \end{split}
    \end{equation}
    We show that $c=1$. Assume, by contradiction, that $c\geq2$. Then, since $c_0\leq c-1$, $3c-c_0-6\geq2c-5$ and so, since $m\leq2,1\leq l$, we have $|E(H)|\geq2|V(H)|+2c-5\geq2|V(H)|-1\geq2|\overline{V(H)}|+m|V_0(H)|-l$. By Proposition~\ref{lemma: there is no tight subgraph containing all three neighbours} and the sparsity of $(G,\psi)$, this is a contradiction. So $c=1$, and $H'$ is connected. Hence, Equation~\eqref{equation (ii)} becomes $|E(H)|\geq2|V(H)|-3$. 

    We now show that $H'$ has at most one fixed vertex. Assume, by contradiction, that $|V_0(H')|\geq2$. By assumption, this implies that $m\neq2$. By Equation~\eqref{equation (ii)}, $|E(H)|\geq2|V(H)|-3\geq2|\overline{V(H)}|+1$. If $m=0$, this contradicts the sparsity of $(G,\psi)$. Hence, $m=1$. By Equation~\eqref{equation (ii)}, 
    \begin{equation*}
        |E(H)|\geq2|V(H)|-3=2|\overline{V(H)}|+|V_0(H)|+(|V_0(H)|-3)\geq2|\overline{V(H)}|+|V_0(H)|-1\geq2|\overline{V(H)}|+|V_0(H)|-l,
    \end{equation*}
    where the last inequality holds because $m\leq l$ and $m=1$. This contradicts Proposition~\ref{lemma: there is no tight subgraph containing all three neighbours}. Hence, $H'$ has at most one fixed vertex. We look at the cases where $V_0(H')=\emptyset$ and $V_0(H')=\{v_0\}$ separately. In both cases, we show that (ii) holds.

    First, assume that $V_0(H')=\emptyset$. Then, by Lemma~\ref{lemma: union of balanced and free-balanced is free-balanced}, $H+f_1+f_2$ is balanced, since $H'$ is connected. By the sparsity of $(G,\psi)$ and Equation~\eqref{equation (ii)}, it follows that $H$ is $(2,3)$-tight, and so (ii) holds.

    Now, suppose that $V_0(H')=\{v_0\}$. Then, by Equation~\eqref{equation (ii)}, $|E(H)|\geq2|V(H)|-3=2|\overline{V(H)}|-1\geq2|\overline{V(H)}|-l$, since $l\geq1$. If $m=0$, this contradicts Proposition~\ref{lemma: there is no tight subgraph containing all three neighbours}, so $m\neq0$. By Equation~\eqref{equation (ii)}, we also know that $|E(H)|\geq2|V(H)|-3=2|\overline{V(H)}|+|V_0(H)|-2$. If $(m,l)=(1,2)$, this contradicts Proposition~\ref{lemma: there is no tight subgraph containing all three neighbours}, so $(m,l)\neq(1,2)$. Hence, $(m,l)$ is one of $(1,1)$ and $(2,2)$. We claim that, in both cases, $v_0$ is not a cut vertex of $H'$. So, assume by contradiction, that $v_0$ is a cut vertex of $H'$.

    For $t\geq2$, let $I_1,\dots,I_t$ be the connected components of $H'-v_0$. For each $1\leq i\leq t$, let $I_i+v_0$ be the graph obtained from $I_i$ by adding $v_0$, together with all edges in $H'$ incident to $v_0$ and a vertex of $I_i$. Notice that, since $H'$ is connected, for all $1\leq i\leq t$, the graph $I_i+v_0$ contains an edge incident to $v_0$. Since $I_i+v_0$ is balanced for all $1\leq i\leq t$, we have 
    \begin{equation*}        |E(H')|=\sum_{i=1}^t|E(I_i+v_0)|\leq2\sum_{i=1}^t|V(I_i+v_0)|-3t=2|V(H')|+2(t-1)-3t\leq2|V(H')|-t-2\leq 2|V(H')|-4.
    \end{equation*}
     But then, 
    \begin{equation*}
    \begin{split}
        |E(H)|&=|E(H_1)|+|E(H_2)|-|E(H')|\geq(2|V(H_1)|-3)+(2|V(H_2)|-3)-(2|V(H')|-4)\\
        &=2|V(H)|-2=2|\overline{V(H)}|+|V_0(H)|-1.
    \end{split}
    \end{equation*}
    This contradicts Proposition~\ref{lemma: there is no tight subgraph containing all three neighbours}, both when $(m,l)=(1,1)$ and when $(m,l)=(2,2)$. Hence, $v_0$ is not a cut vertex of $H'$. Then, $H+f_1+f_2$ is balanced by Lemma~\ref{lemma: union of balanced and free-balanced is free-balanced}. By the sparsity of $(G,\psi)$ and Equation~\eqref{equation (ii)}, it follows that $H$ is $(2,3)$-tight, and so (ii) holds.
    \end{proof}

\subsection{Applying a 1-reduction to a (2,m,3,l)-gain tight graph}
\label{section on 1-reduction}
The following result is crucial for our combinatorial characterisations of symmetry-generic infinitesimal rigidity. It states that, except for two specific cases, there is always an admissible reduction  at a vertex $v$ of degree 3 of a $(2,m,3,l)$-gain tight graph, where $(m,l)=(0,2),(1,1),(1,2),(2,2).$
\begin{theorem}
\label{theorem on 1-reductions: (2,m,3,l)-tight}
Let $(G,\psi)$ be a $\Gamma$-gain graph with a free vertex $v$ of degree 3 which has no loop. Suppose that $(G,\psi)$ is $(2,m,3,l)$-tight, where $(m,l)$ is one of the pairs $(0,1),(1,1),(1,2)$ or $(2,2)$. Suppose further that if $|V_0(G)|\geq2$, then $m=1$. If there is not an admissible 1-reduction at $v$, then exactly one of the following holds.
\begin{itemize}
    \item[(i)] $(G,\psi)$ is $(2,2,3,2)$-tight and $v$ has exactly one free neighbour $v_1$ and exactly one fixed neighbour $v_2$ (see Figure~\ref{degree 2 vertex problem image} (a),(b)) . 
    \item[(ii)] $(G,\psi)$ is $(2,1,3,2)$-tight and $v$ has three neighbours, all of which are fixed (see Figure~\ref{degree 2 vertex problem image} (c)). 
\end{itemize}
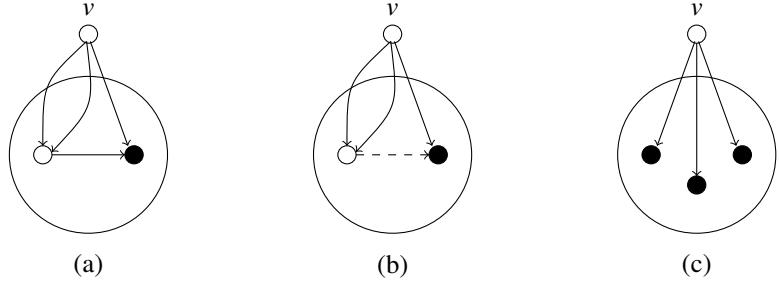
\begin{figure}[H]
    \centering
    \begin{tikzpicture}[scale=.8]
        \draw(3.5,0) circle (0.15cm);  
  \draw (4.25,0) circle (1.3cm);
  \draw[fill=black] (5,0) circle (0.15cm);
  \draw (4.25, 2) circle (0.15cm);
  \node[above] at (4.25, 2.15){$v$};
  \draw[->] (4.28, 1.88) -- (4.9, 0.15);
  \draw[->] (4.22,1.88) .. controls (4.35,0.8) .. (3.65, 0.05);
  \draw[->] (4.22,1.88) .. controls (3.5,1.2) .. (3.5, 0.15);
  \draw[->] (3.65,0) -- (4.85,0);
  \node[below] at (4.25, -1.5) {(a)};

  \draw(8.5,0) circle (0.15cm);  
  \draw (9.25,0) circle (1.3cm);
  \draw[fill=black] (10,0) circle (0.15cm);
  \draw (9.25, 2) circle (0.15cm);
  \node[above] at (9.25, 2.15){$v$};
  \draw[->] (9.28, 1.88) -- (9.9, 0.15);
  \draw[->] (9.22,1.88) .. controls (9.35,0.8) .. (8.65, 0.05);
  \draw[->] (9.22,1.88) .. controls (8.5,1.2) .. (8.5, 0.15);
  \draw[->, dashed] (8.65,0) -- (9.85,0);
  \node[below] at (9.25, -1.5) {(b)};

  \draw[fill=black](13.5,0) circle (0.15cm);  
  \draw (14.25,0) circle (1.3cm);
  \draw[fill=black](14.25,-0.5) circle (0.15cm);
  \draw[fill=black] (15,0) circle (0.15cm);
  \draw (14.25, 2) circle (0.15cm);
  \node[above] at (14.25, 2.15){$v$};
  \draw[->] (14.28, 1.88) -- (14.9, 0.15);
  \draw[->] (14.22, 1.88) -- (13.6, 0.15);
  \draw[->] (14.25, 1.85) -- (14.25, -0.35);
  \node[below] at (14.25, -1.5) {(c)};
    \end{tikzpicture}
    \caption{Three instances of a vertex $v$ of degree 3. In (a,b), $v$ has two neighbours, one of which is fixed. In (a) there is an edge between the neighbours of $v$, whereas in (b) there isn't. In (c), $v$ has three neighbours, all of which are fixed.}
    \label{degree 2 vertex problem image}
\end{figure}
\end{theorem}

\begin{proof}
   We will look at the cases where $N(v)=1,N(v)=2$, and $N(v)=3$ separately. We prove that, in each case, we may always apply a reduction at $v$, unless one of (i),(ii) holds.

   \medskip
    \textit{\textbf{Case 1: }$N(v)=1.$}

    Let $u$ be the neighbour of $v$, and $e_1,e_2,e_3$ be the edges incident to $u$ and $v$. By Corollary~\ref{corollary: rank of orbit matrix and blocks are equal}, Proposition~\ref{prop: switching maintains rank} and Lemma~\ref{lemma: forests can have identity gain}(i), we may assume that $\psi(e_1)=\textrm{id}$. Moreover, $\psi(e_2)\psi(e_3)^{-1}\neq\textrm{id}$ by the definition of gain graph. Let $(G',\psi')$ be obtained from $G-v$ by adding a loop $f$ at $u$ with gain $\psi(e_2)\psi(e_3)^{-1}$. We show that this 1-reduction is admissible.

    Suppose, by contradiction, that $H$ is a blocker of $(G',\psi')$. Since $H+f$ contains the loop $f$, $H$ is not a balanced blocker. By Proposition~\ref{lemma: there is no tight subgraph containing all three neighbours}, $H$ is not a general-count blocker. This contradicts the fact that $H$ is a blocker. Hence, there is an admissible 1-reduction at $v$.

    \medskip
    \textit{\textbf{Case 2: }$N(v)=2.$}
    Let $v_1,v_2$ be the neighbours of $v$, let $e_1,e'_1:=(v,v_1)$ and $e_2:=(v,v_2)$, and let $g=\psi(e'_1)$. By Corollary~\ref{corollary: rank of orbit matrix and blocks are equal}, Proposition~\ref{prop: switching maintains rank} and Lemma~\ref{lemma: forests can have identity gain}(i), we may assume that $\psi(e_1)=\psi(e_2)=\textrm{id}$ and $g\neq \textrm{id}$. We will look at the cases where $v_2$ is free and fixed separately. 

    \begin{itemize}
        \item\textbf{Sub-case 2a: }$v_2$ is free.

        Let $(G_1,\psi_1),(G_2,\psi_2),(G_3,\psi_3)$ be obtained from $G-v$ by adding, respectively, the edge $f_1=(v_1,v_2)$ with gain $\textrm{id}$, the edge $f_2=(v_1,v_2)$ with gain $g$, and a loop $f_3$ at $v_1$ with gain $g$. Assume, by contradiction, that $H_1,H_2$ and $H_3$ are blockers for $(G_1,\psi_1),(G_2,\psi_2)$ and $(G_3,\psi_3)$, respectively. Let $H=H_1\cup H_2\cup H_3$ and $H'=H_1\cap H_2\cap H_3$.

        By Proposition~\ref{lemma: there is no tight subgraph containing all three neighbours}, $H_1,H_2$ are balanced blockers. Moreover, by Lemma~\ref{lemma: what i need for 1-red.}(ii), $E(H_1\cap H_2)=\emptyset$: otherwise $H_1\cup H_2+f_1+f_2$ is balanced, a contradiction since $f_1,f_2^{-1}$ is an unbalanced 2-cycle. Since $H_3+f_3$ contains the loop $f_3$, $H_3$ is a general-count blocker. It follows, from Lemma~\ref{lemma: what i need for 1-red.}(i), that $E(H_1\cap H_3)=E(H_2\cap H_3)=\emptyset$. Moreover, by Proposition~\ref{lemma: there is no tight subgraph containing all three neighbours}, $v_2\not\in V(H_3)$. Since $m\leq2$, for $i=1,2,$ $|E(H_i)|=2|V(H_i)|-3\geq2|\overline{V(H_i)}|+m|V_0(H_i)|-3$. Hence,
        \begin{equation*}
            \begin{split}
                |E(H)|&=|E(H_1)|+|E(H_2)|+|E(H_3)|\\
                &\geq(2|\overline{V(H_1)}|+m|V_0(H_1)|-3)+(2|\overline{V(H_2)}|+m|V_0(H_2)|-3)+(2|\overline{V(H_3)}|+m|V_0(H_3)|-l)\\
                &=2\Big(|\overline{V(H_1)}|+|\overline{V(H_2)}|+|\overline{V(H_3)}|\Big)+m\Big(|V_0(H_1)|+|V_0(H_2)|+|V_0(H_3)|\Big)-6-l\\
&=2\Big(|\overline{V(H)}|+\sum_{1\leq i\neq j\leq3}|\overline{V(H_i\cap H_j)}|-|\overline{V(H')}|\Big)+m\Big(|V_0(H)|+\sum_{1\leq i\neq j\leq3}|V_0(H_i\cap H_j)|-|V_0(H')|\Big)-6-l\\
&\geq2\Big(|\overline{V(H)}|+\sum_{1\leq i\neq j\leq3}|\overline{V(H_i\cap H_j)}|-|\overline{V(H')}|\Big)+m|V_0(H)|-6-l,
            \end{split}
        \end{equation*}
        where the last inequality holds because $V_0(H')\subseteq V_0(H_i\cap H_t)$ for all pairs $1\leq i\neq t\leq3$ and $m\geq0$.
        Since the free vertex $v_2$ is not contained in $H'$, it is easy to see that $\sum_{1\leq i\neq j\leq3}|\overline{V(H_i\cap H_j)}|-|\overline{V(H')}|\geq3$. Hence, $|E(H)|\geq2|\overline{V(H)}|+m|V_0(H)|-l$. This contradicts Proposition~\ref{lemma: there is no tight subgraph containing all three neighbours} and so there is an admissible 1-reduction at $v$.

    \item\textbf{Sub-case 2b: }$v_2$ is fixed.

    If $(G,\psi)$ is $(2,2,3,2)$-tight, then (i) holds. So, assume $(G,\psi)$ is not $(2,2,3,2)$-tight.
    
    Let $(G_1,\psi_1),(G_2,\psi_2)$ be the graphs obtained from $G-v$ by adding, respectively, an edge $f_1=(v_1,v_2)$ with gain $\textrm{id}$, and a loop $f_2$ at $v_1$ with gain $g$ (see Figure~\ref{1-reduction image}). Notice that, if there is already an edge $(v_1,v_2)\in E(G)$, $(G_1,\psi_1)$ is not a well-defined gain graph. We look at the cases where $(v_1,v_2)\in E(G)$ and $(v_1,v_2)\not\in E(G)$ separately. 

     \begin{figure}[H]
    \centering
    \begin{tikzpicture}
  [scale = 0.8, auto=left]
  \draw(8.5,0) circle (0.15cm);  
  \draw[fill=black] (10,0) circle (0.15cm);
  \draw (9.25, 2) circle (0.15cm);
  \node[right] at (9.4, 2){$v$};
  \node[below] at (8.15,0.15){$v_1$};
  \node[below] at (10.4, 0.15){$v_2$};
  \draw[->] (9.29, 1.86) -- (9.95, 0.15);
  \draw[->] (9.22,1.85) .. controls (9.35,0.8) .. (8.65, 0.05);
  \draw[->] (9.15,1.9) .. controls (8.5,1.2) .. (8.5, 0.15);
  \node[right] at (9.65,0.9) {$e_2$};
  \node[left] at (8.55,0.9) {$e_1$};
  \node[left] at (9.25,0.9) {$e'_1$};

  \draw(13.5,0) circle (0.15cm);  
  \draw[fill=black] (15,0) circle (0.15cm);
  \node[below] at (13.15,0.15){$v_1$};
  \node[below] at (15.4, 0.15){$v_2$};
  \draw[->] (13.65,0) -- (14.85,0);
  \node[above] at (14.25,0){$f_1$};

  \draw(18.5,0) circle (0.15cm);  
  \draw[fill=black] (20,0) circle (0.15cm);
  \node[below] at (18.15,0.15){$v_1$};
  \node[below] at (20.4, 0.15){$v_2$};
  \draw[->] (18.65,0) .. controls (19.05,0.5) and (17.95,0.5) .. (18.35,0); 
  \node[above] at (18.5,0.3) {$f_2$};
\end{tikzpicture}
    \caption{Two possible 1-reductions at $v$.}
    \label{1-reduction image}
\end{figure}
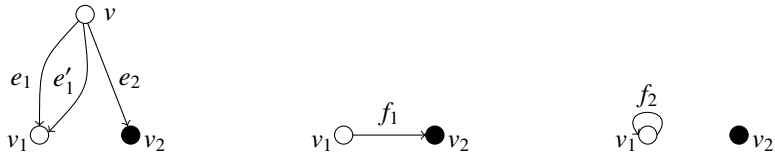

    \medskip
    First, assume $(v_1,v_2)\in E(G)$. Then it is easy to check that the graph induced by $v,v_1,v_2$ violates both $(2,0,3,1)$-gain sparsity and $(2,1,3,2)$-gain sparsity. Hence, we may assume that $(G,\psi)$ is $(2,1,3,1)$-gain tight. 

    Assume, by contradiction, that $(G_2,\psi_2)$ has a blocker $H_2$. Since $H_2+f_2$ contains the loop $f_2$, it is unbalanced. Hence, $H_2$ is a general-count blocker. It follows, from Proposition~\ref{lemma: there is no tight subgraph containing all three neighbours}, that $v_2\not\in V(H_2)$. Hence, the graph $H$ obtained from $H_2$ by adding $v,v_2,$ together with the edges $e_1,e_1',e_2,(v_1,v_2)$, satisfies
    \begin{equation*}
            |E(H)|=|E(H_2)|+4=2|\overline{V(H_2)}|+|V_0(H_2)|+3=2|\overline{V(H)}|+|V_0(H)|,
        \end{equation*}
        contradicting the sparsity of $(G,\psi)$. Thus, the 1-reduction at $v$ which yields $(G_2,\psi_2)$ is admissible.

    \medskip
    Now, let $(v_1,v_2)\not\in E(G)$. Assume that $H_1$ and $H_2$ are blockers for $(G_1,\psi_1)$ and $(G_2,\psi_2)$, respectively. By Proposition~\ref{lemma: there is no tight subgraph containing all three neighbours}, $H_1$ is a balanced blocker. Since $H_2+f_2$ contains the loop $f_2$, it is a general-count blocker. By Proposition~\ref{lemma: there is no tight subgraph containing all three neighbours}, $v_2\not\in V(H_2)$. Moreover, by Lemma~\ref{lemma: what i need for 1-red.}(i), $E(H_1\cap H_2)=\emptyset$. Let $H:=H_1\cup H_2$ and $H':=H_1\cap H_2$. We have
    \begin{equation*}
    \begin{split}
         |E(H)|&=(2|V(H_1)|-3)+(2|\overline{V(H_2)}|+m|V_0(H_2)|-l)\\
         &=(2|\overline{V(H_1)}|+2|V_0(H_1)|-3)+(2|\overline{V(H_2)}|+m|V_0(H_2)|-l)\\
         &=2|\overline{V(H)}|+m|V_0(H)|-l+[2|\overline{V(H')}|+m|V_0(H')|+(2-m)|V_0(H_1)|-3]\\
         &=2|\overline{V(H)}|+m|V_0(H)|-l,
    \end{split}
    \end{equation*}
    where the last inequality holds because $|\overline{V(H')}|\geq 1$, $|V_0(H_1)|\geq 1$ and $m\leq1$. This contradicts Proposition~\ref{lemma: there is no tight subgraph containing all three neighbours}. Hence, there is an admissible 1-reduction at $v$.
    \end{itemize}

    So for Case 2, we have shown that there always is an admissible 1-reduction at $v$, unless (i) holds. 
    
   \medskip
   \textit{\textbf{Case 3: }$N(v)=3$.}

    For $i=1,2,3$, let $e_i=(v,v_i)$ be the edges incident with $v$. Let $f_1=(v_1,v_2),f_2=(v_2,v_3)$ and $f_3=(v_3,v_1)$. By Corollary~\ref{corollary: rank of orbit matrix and blocks are equal}, Proposition~\ref{prop: switching maintains rank} and Lemma~\ref{lemma: forests can have identity gain}(i), we may assume $\psi(e_1)=\psi(e_2)=\psi(e_3)=\textrm{id}$.
    
    For $1\leq i\leq3$, let $(G_i,\psi_i)$ be obtained by applying a 1-reduction at $v$, during which we add the edge $f_i$ with gain $\textrm{id}$, and assume that $(G_i,\psi_i)$ has a blocker $H_i$. Let $H:=H_1\cup H_2\cup H_3$ and $H':=H_1\cap H_2\cap H_3$.

    We first show that $E(H_i\cap H_j)=\emptyset$ for all pairs $i\neq j$. As a first step, we show that  $E(H_i\cap H_j)\neq\emptyset$ for at most one pair $i\neq j$. Without loss of generality, let $E(H_1\cap H_2)\neq\emptyset$ and $E(H_1\cap H_3)\neq\emptyset$. By Lemma~\ref{lemma: what i need for 1-red.}(ii), $H_1\cup H_2$ is $(2,3)$-tight, and $H_1\cup H_2+f_1+f_2$ is balanced. Moreover, 
    \begin{equation*}
        |E(H_1\cup H_2+v)|=|E(H_1\cup H_2)|+3=2|V(H_1\cup H_2)|=2|V(H_1\cup H_2+v)|-2.
    \end{equation*}
    If $H_1\cup H_2+v$ is balanced, this contradicts the sparsity of $(G,\psi)$. So, we may assume that $H_1\cup H_2+v$ (equivalently, $H_1\cup H_2+f_1+f_2+f_3$) is unbalanced. The group $\left<H_1\cup H_2+v\right>\simeq\left<H_1\cup H_2+f_1+f_2+f_3\right>$ is given by the elements of $\left<H_1\cup H_2+f_1+f_2\right>$, together with the gains of the walks from $v_1$ to $v_3$ which do not contain fixed vertices. Since $H_1\cup H_2+f_1+f_2$ is balanced and $H_1\cup H_2+v$ is unbalanced, there must be a path $P$ from $v_3$ to $v_1$ in $H_1\cup H_2$ with gain $g\neq\text{id}$, which contains only free vertices.
       In particular, $v_1,v_3$ are free. Moreover, $v_2$ must be fixed, for otherwise, $f_1,f_2,P$ is a closed path in $H_1\cup H_2+f_1+f_2$ with gain $g\neq\text{id}$ and with 
       no fixed vertex, contradicting the fact that $H_1\cup H_2+f_1+f_2$ is balanced. 

    Applying the same argument to $H_1\cup H_3$, we may conclude that $v_1$ is fixed and $v_2,v_3$ are free. But this contradicts the fact that $v_1$ is free and $v_2$ is fixed. Hence, $E(H_i\cap H_j)\neq\emptyset$ for at most one pair $1\leq i\neq j\leq3$. Without loss of generality, let $E(H_1\cap H_2)\neq\emptyset$ and
     $E(H_1\cap H_3)=E(H_2\cap H_3)=\emptyset$. Note that in this case, as shown above, $v_2$ is fixed and $v_1$ and $v_3$ are free, which is a fact we will use later.    

    If $H_3$ is a balanced blocker, then
    \begin{equation}
    \label{eq. gcb}
        \begin{split}
            |E(H)|&=|E(H_1\cup H_2)|+|E(H_3)|=(2|V(H_1\cup H_2)|-3)+(2|V(H_3)|-3)\\
            &=2|V(H)|+2|V((H_1\cup H_2)\cap H_3)|-6.
        \end{split}
    \end{equation}
    If $|V((H_1\cup H_2)\cap H_3)|\geq3$, then $|E(H)|\geq2|V(H)|>2|\overline{V(H)}|+m|V_0(H)|-l$, contradicting the sparsity of $(G,\psi)$. Hence, $V((H_1\cup H_2)\cap H_3)=\{v_1,v_3\}$, and so $H$ is balanced (every closed walk in $H$ is composed of closed walks in $H_1\cup H_2$, of closed walks in $H_3$, and of concatenations of walks from $v_1$ to $v_3$ in $H_1\cup H_2$ together with walks from $v_3$ to $v_1$ in $H_3$; all such walks must have identity gain, since $H_1\cup H_2+f_1+f_2,H_3+f_3$ are balanced). However, by Equation~\eqref{eq. gcb}, we have $|E(H)|=2|V(H)|-2$, contradicting the sparsity of $(G,\psi)$. So, we may assume that $H_3$ is a general-count blocker. Then, it is easy to see that
    \begin{equation*}
        \begin{split}
            |E(H)|&=|E(H_1\cup H_2)|+|E(H_3)|=(2|V(H_1\cup H_2)|-3)+(2|\overline{V(H_3)}|+m|V_0(H_3)|-l)\\
            &=2|\overline{V(H)}|+m|V_0(H)|-l+(2-m)|V_0(H_1\cup H_2)|+m|V_0((H_1\cup H_2)\cap H_3)|+2|\overline{V((H_1\cup H_2)\cap H_3)}|-3\\&\geq2|\overline{V(H)}|+m|V_0(H)|-l+1,
        \end{split}
    \end{equation*}
    where the inequality holds because $v_1,v_3\in\overline{V((H_1\cup H_2)\cap H_3)}$ and $0\leq m\leq2$. This contradicts the sparsity of $(G,\psi)$. Hence, $E(H_i\cap H_j)=\emptyset$ for all pairs $i\neq j$. 
        
    We now show that $H_1,H_2,H_3$ cannot all be balanced blockers. Assume, by contradiction, that $H_1,H_2,H_3$ are all balanced blockers. If $|V(H_i\cap H_j)|=1$ for all $1\leq i\neq j\leq3$, then it is easy to see that $H+v$ is balanced (since every closed walk in $H$ is composed of closed walks in $H_i$ for $1\leq i\leq 3$, and the concatenation of walks from $v_1$ to $v_2$ in $H_1$,  walks from $v_2$ to $v_3$ in $H_2$, and walks from $v_3$ to $v_1$ in $H_3$; all such walks either include a fixed vertex or must have identity gain) and
            \begin{equation}
            \label{eq: bal.block.}
            \begin{split}
                |E(H)|&=\sum_{i=1}^3|E(H_i)|=2\sum_{i=1}^3|V(H_i)|-9=2|V(H)|+2\sum_{1\leq i\neq j\leq3}|V(H_i\cap H_j)|-2|V(H')|-9\\
                &=2|V(H)|+2(1+1+1)-9=2|V(H)|-3.
            \end{split}
            \end{equation}    
            This contradicts Proposition~\ref{lemma: there is no tight subgraph containing all three neighbours}. Hence, $|V(H_i\cap H_j)|\geq2$ for some $1\leq i\neq j\leq3$. Without loss of generality, let $|V(H_1\cap H_2)|\geq2$. Then,
              \begin{equation*}
            |E(H_1\cup H_2)|=2|V(H_1\cup H_2)|+2|V(H_1\cap H_2)|-6\geq2|V(H_1\cup H_2)|-2\geq2|\overline{V(H_1\cup H_2)}|+m|V_0(H_1\cup H_2)|-2,
        \end{equation*}
        where the last inequality holds because $m\leq2$. If $l=2$, this contradicts Proposition~\ref{lemma: there is no tight subgraph containing all three neighbours}. Hence, $l=1.$ Moreover, rearranging Equation~\eqref{eq: bal.block.}, we know that 
           \begin{equation*}
            \begin{split}
                |E(H)|&=2|V(H)|-1+2\left(\sum_{1\leq i\neq j\leq3}|V(H_i\cap H_j)|-|V(H')|-4\right)\\
                &\geq2|\overline{V(H)}|+m|V_0(H)|-1+2\left(\sum_{1\leq i\neq j\leq3}|V(H_i\cap H_j)|-|V(H')|-4\right),
            \end{split}
            \end{equation*}
        since $m\leq2$.
        If we show that $f:=\sum_{1\leq i\neq j\leq3}|V(H_i\cap H_j)|-|V(H')|\geq4$, this contradicts Proposition~\ref{lemma: there is no tight subgraph containing all three neighbours}. Notice that $|V(H')|$ is at most the minimum of $|V(H_i\cap H_j)|$, where $i\neq j$ run from 1 to 3. Call this number $min$. Hence, $f\geq\sum_{1\leq i\neq j\leq3}|V(H_i\cap H_j)|-min$. If $min=|V(H_1\cap H_2)|$, then $|V(H_2\cap H_3)|,|V(H_1\cap H_3)|\geq2$, and so $f\geq2+2\geq4$. So assume, without loss of generality, that $min=|V(H_2\cap H_3)|$, and hence that $f\geq|V(H_1\cap H_2)|+|V(H_1\cap H_3)|.$ If $|V(H_1\cap H_3)|\geq2$, then $f\geq4$. So, assume that $V(H_1\cap H_3)=\{v_1\}$. By minimality, $V(H_2\cap H_3)=\{v_3\}.$ It follows that $V(H')=\emptyset$, and so $f\geq2+1+1=4$. Since we reached a contradiction, $H_1,H_2,H_3$ are not all balanced blockers. Assume, without loss of generality, that $H_1$ is a general-count blocker.

        \begin{claim}
            For $2\leq i\leq3$, we have $|V(H_1\cap H_i)|=1$.
        \end{claim}
        
        \begin{claimproof}
            If $H_i$ is also a general-count blocker, then, since $|E(H_1\cup H_i)|=
                |E(H_1)|+|E(H_i)|$, we have
            \begin{equation*}
        \label{claim 1 equation}
            \begin{split}
                |E(H_1\cup H_i)|
                &=2|\overline{V(H_1\cup H_i)}|+m|V_0(H_1\cup H_i)|-l+(2|\overline{V(H_1\cap H_i)}|+m|V_0(H_1\cap H_i)|-l)\\
                &\geq2|\overline{V(H_1\cup H_i)}|+m|V_0(H_1\cup H_i)|-l+(2|\overline{V(H_1\cap H_i)}|+m|V_0(H_1\cap H_i)|-2),
            \end{split}
        \end{equation*}
        since $l\leq2$.
        If $|\overline{V(H_1\cap H_i)}|\geq1$, or if $|V_0(H_1\cap H_i)|\geq2$ (and so $m=1$ by assumption), it is easy to see that this is at least $2|\overline{V(H_1\cup H_i)}|+m|V_0(H_1\cup H_i)|-l$. This contradicts Proposition~\ref{lemma: there is no tight subgraph containing all three neighbours}. Hence, $|V(H_1\cap H_i)|=|V_0(H_1\cap H_i)|=1$, and the claim holds. If $H_i$ is a balanced blocker, it is easy to see that
        \begin{equation*}
        \begin{split}
             |E(H_1\cup H_i)|&=2|\overline{V(H_1\cup H_i)}|+m|V_0(H_1\cup H_i)|-l+(2|\overline{V(H_1\cap H_i)}|+m|V_0(H_1\cap H_i)|+(2-m)|V_0(H_i)|-3)\\
             &\geq2|\overline{V(H_1\cup H_i)}|+m|V_0(H_1\cup H_i)|-l+(2|\overline{V(H_1\cap H_i)}|+2|V_0(H_1\cap H_i)|-3)\\
             &=2|\overline{V(H_1\cup H_i)}|+m|V_0(H_1\cup H_i)|-l+(2|V(H_1\cap H_i)|-3),
        \end{split}
        \end{equation*}
        where the inequality holds because $V_0(H_1\cap H_i)\subseteq V_0(H_i)$ and $m\leq2$. If $|V(H_1\cap H_i)|\geq2$, then this contradicts the sparsity of $(G,\psi)$. Hence, the claim holds.
        \end{claimproof} 
        
        \medskip
        By the Claim, $V(H_1\cap H_2)=\{v_2\}$ and $V(H_1\cap H_3)=\{v_1\}$. Hence, $v_1,v_2,v_3$ do not lie in $V(H')$. 
        Let $n$ be the number of free vertices in $\{v_1,v_2,v_3\}$. Since each vertex in $\{v_1,v_2,v_3\}$ lies in $H_i\cap H_j$ for some $0\leq i\neq j\leq1$, this implies that
            \begin{eqnarray*}
                \overline{S}:=\sum_{1\leq i\neq j\leq3}|\overline{V(H_i\cap H_j)}|-|\overline{V(H'})|\geq n \qquad\textrm{and} \qquad
                S_0:=\sum_{1\leq i\neq j\leq3}|V_0(H_i\cap H_j)|-|V_0(H')|\geq 3-n.
            \end{eqnarray*}
        We look at the following sub-cases separately: $H_2,H_3$ are balanced blockers; $H_2$ is a general-count blocker and $H_3$ is a balanced blocker; $H_2,H_3$ are general-count blockers;

    \begin{itemize}
        \item \textbf{Sub-case 3a: } $H_2,H_3$ are balanced blockers. Then,
        \begin{equation*}
            \begin{split}
|E(H)|&=(2|\overline{V(H_1)}|+m|V_0(H_1)|-l)+(2|V(H_2)|-3)+(2|V(H_3)|-3)\\
&=2[|\overline{V(H)}|+\overline{S}]+m[|V_0(H)|+S_0]+(2-m)(|V_0(H_2)|+|V_0(H_3)|)-6-l\\
    &\geq2[\overline{V(H)}|+n]+m[|V_0(H)|+(3-n)]+(2-m)(|V_0(H_2)|+|V_0(H_3)|)-6-l\\
    &=2|\overline{V(H)}|+m|V_0(H)|-l+[2n+m(3-n)+(2-m)(|V_0(H_2)|+|V_0(H_3)|)-6].
           \end{split}
           \end{equation*}
           Let $f:=2n+m(3-n)+(2-m)(|V_0(H_2)|+|V_0(H_3)|)-6$. If $f\geq0$, Proposition \ref{lemma: there is no tight subgraph containing all three neighbours} leads to a contradiction, and so there is an admissible 1-reduction at $v$. We will show that indeed $f\geq0$.

           This is clear if $n=3$. Suppose $n=2$, so that $f=m+(2-m)(|V_0(H_2)|+|V_0(H_3)|)-2$. Since $n=2$, there is at least one fixed vertex in $\{v_1,v_2,v_3\}$, and so $|V_0(H_2)|+|V_0(H_3)|\geq1$. Hence, $f\geq m+2-m-2=0$. 

           So, we may assume $n\leq1$. Hence, there are at least two fixed vertices in $\{v_1,v_2,v_3\}\subset V(G)$, and so $|V_0(H_2)|+|V_0(H_3)|\geq 2$. By assumption, this implies that $m=1$. Hence, $f=n-3+|V_0(H_2)|+|V_0(H_3)|\geq n-1$. When $n=1$, $f\geq0$. So, let $n=0$. Then $|V_0(H_2)|,|V_0(H_3)|\geq2$, so $f\geq1$. 

           \item\textbf{Sub-case 3b: } $H_2$ is a general-count blocker and $H_3$ is a balanced blocker.

           We have
           \begin{equation*}
            \begin{split}
|E(H)|&=(2|\overline{V(H_1)}|+|V_0(H_1)|-l)+(2|\overline{V(H_2)}|+|V_0(H_2)|-l)+(2|V(H_3)|-3)\\
&=2[|\overline{V(H)}|+\overline{S}]+m[|V_0(H)|+S_0]+(2-m)|V_0(H_3)|-3-2l\\
    &\geq2[|\overline{V(H)}|+n]+m[|V_0(H)|+(3-n)]+(2-m)|V_0(H_3)|-3-2l\\
    &=2|\overline{V(H)}|+m|V_0(H)|-l+[2n+m(3-n)+(2-m)|V_0(H_3)|-3-l]
           \end{split}
           \end{equation*}
           If  $f:=2n+m(3-n)+(2-m)|V_0(H_3)|-3-l\geq0$, then we obtain a contradiction by Proposition~\ref{lemma: there is no tight subgraph containing all three neighbours}. We will show that indeed $f\geq0$. If $n=3$, then $f\geq3-l>0$, since $l\leq2$. If $n=2$, then $f\geq1+m-l\geq0$, since $l-m\leq1$. Hence, we may assume that $n\leq1$. So, at least two of the elements in $\{v_1,v_2,v_3\}\subset V(G)$ are fixed. It follows that $m=1$ and $f=n-l+|V_0(H_3)|$. If $n=1$, then $|V_0(H_3)|\geq1$ and $f=1-l+|V_0(H_3)|\geq2-l\geq0$, since $l\leq2$. If $n=0$, then $|V_0(H_3)|\geq2$ and $f\geq2-l\geq0$. 

           \item \textbf{Sub-case 3c: }$H_2,H_3$ are general-count blockers.
        
           We have
            \begin{equation*}
                \begin{split}
                |E(H)|&=\sum_{i=1}^3|E(H_i)|=2\sum_{i=1}^3|\overline{V(H_i)}|+m\sum_{i=1}^3|V_0(H_i)|-3l\\
                &=2(|\overline{V(H)}|+\overline{S})+m(|V_0(H)|+S_0)-3l\\
                &\geq2|\overline{V(H)}|+m|V_0(H)|-l+[2n+m(3-n)-2l].
            \end{split}
            \end{equation*}
            If  $f:=2n+m(3-n)-2l\geq0$, then we obtain a contradiction by Proposition~\ref{lemma: there is no tight subgraph containing all three neighbours}. We will show that $f\geq0$ unless (ii) holds.
            If $n=3$, $f=6-2l=2(3-l)>0$, since $l\leq2$. If $n=2$, then $f=4+m-2l=2(2-l)+m\geq0$, since $l\geq2$ and $m\geq0$. So, we may assume that $n\leq1$, which implies that $m=1$. Hence, $f=n+3-2l$. If $n=1$, then
            $f=2(2-l)\geq0$. If $n=0$, then $f=3-2l$. So if $l\leq1$, then $f\geq1$. This leaves the case that $l=2$. In this case (ii) holds. 
    \end{itemize}
    This proves the result.
\end{proof}

\subsection{Reflection group}
Let $(\tilde{G},\tilde{p})$ be a $\mathcal{C}_s$-generic framework. Recall that the $\Gamma$-gain graph $(G,\psi)$ of $\tilde{G}$ is $(2,1,3,1)$-gain tight whenever $(\tilde{G},\tilde{p})$ is fully-symmetrically isostatic, and $(G,\psi)$ is $(2,1,3,2)$-gain tight whenever $(\tilde{G},\tilde{p})$ is anti-symmetrically isostatic (see Proposition~\ref{Necessary conditions for reflection} in Section~\ref{sec:nec(2)}). In this section, we show that the converse statements are also true. 

To do so, we employ a proof by induction on $|V(G)|$, which uses the vertex extension and reduction moves described in Section~\ref{sec:red}. Hence, we first need to show that there is an admissible reduction of $(G,\psi)$, whose corresponding extension does not break fully-symmetric or anti-symmetric infinitesimal rigidity. Let $v\in V(G)$ be a free vertex of degree 3 with no loop. By Theorem~\ref{theorem on 1-reductions: (2,m,3,l)-tight}, there is always an admissible 1-reduction at $v$, unless all neighbours of $v$ are fixed and $(G,\psi)$ is $(2,1,3,2)$-gain tight. Lemma~\ref{lemma: 1-extension maintain rank} shows that a 1-extension maintains the fully-symmetric and anti-symmetric infinitesimal rigidity of a framework. However, the result assumes that all neighbours of the added vertex do not lie on the same line, and hence they cannot all be fixed. This issue arises both in the fully-symmetric and the anti-symmetric cases. Hence, our proof by induction cannot rely on applying a 1-reduction to a vertex whose neighbours are all fixed.

In the following result we show that, if $G$ has at least two free vertices, and all free vertices of degree 3 in $V(G)$ have three fixed neighbours, then there is another vertex in $V(G)$ at which we may apply an admissible reduction.

\begin{lemma}
\label{there is an admissible reduction: reflection}
    For $1\leq l\leq 2$, let $(G,\psi)$ be a $(2,1,3,l)$-gain tight graph with $|\overline{V(G)}|\geq2$. Then there is a reduction of $(G,\psi)$ which yields a $(2,1,3,l)$-gain tight graph $(G',\psi')$. The reduction which yields $(G',\psi')$ is one of the following: a fix-0-reduction, a 0-reduction, a loop-1-reduction, a 1-reduction at a vertex with at least one free neighbour, or a fix-1-reduction.
\end{lemma}

\begin{proof}
    The case where there are no fixed vertices is known (see e.g., Theorem 6.3 in \cite{bt2015}), so we may assume $V_0(G)\neq\emptyset$. Suppose $G$ has a vertex $v$ which is either a fixed vertex of degree 1, or a free vertex of degree 2, or a free vertex of degree 3 with a loop (notice that if $v$ has a loop, then $l=1$). Then, we may apply a fix-0-reduction, or a 0-reduction, or a loop-1-reduction at $v$. All such reductions are clearly admissible. Hence, we may assume that there is no such vertex $v$. By Theorem~\ref{theorem on 1-reductions: (2,m,3,l)-tight}, we may also assume that all free vertices of degree 3 in $V(G)$ have three distinct neighbours, all of which are fixed. Let $n$ be the number of vertices of degree 2 in $V_0(G)$.

    \begin{claim}
        Under the above assumptions, we have $n\geq3$.
    \end{claim}

    \medskip

    \textit{Proof. }To see this, let $v_1,\dots,v_t$ be the free vertices in $G$ of degree 3 and assume that for all $1\leq i\leq t$, the edges incident with $v_i$ are directed to $v_i$. Notice that $t\geq2$, by Lemma~\ref{lemma: there is a vertex of degree 2 or 3}. Define the set $V':=\{v\in V_0(G):(v,v_i)\in E(G) \text{ for some }1\leq i\leq t\}$. Let $n'=|V'|$ and consider the subgraph $H$ of $G$ induced by $\{v_1,\dots,v_t\}\cup V'$. By the sparsity of $(G,\psi)$, $3t\leq|E(H)|\leq2t+n'-l$ and hence $n'\geq t+l$. 
    
    Now, the average degree of $G$ is $\hat{\rho}=\frac{4\left|\overline{V(G)}\right|+2|V_0(G)|-2l}{|V(G)|}.$ 
    This average is smallest when all vertices in $\overline{V(G)}\setminus\{v_1,\dots,v_t\}$ have degree 4, and all fixed vertices in $V(G)$ which do not have degree 2, have degree 3. This gives $\hat{\rho}\geq \frac{4\overline{|V(G)|}+3|V_0(G)|-n-t}{|V(G)|}$. Hence, 
    \begin{equation*}
        n\geq|V_0(G)|+2l-t\geq n'+2l-t\geq(t+l)+2l-t=3l\geq3,
    \end{equation*}
    as required. $\square$

    \medskip
    So, there is a fixed vertex $v$ of degree 2. Let $u_1,u_2$ be the neighbours of $v$. Notice that there is no $(2,1,3,l)$-gain tight subgraph $H$ of $G$ with $u_1,u_2\in V(H),v\not\in V(H)$, as otherwise the graph $H':=H+v$ satisfies
    \begin{equation}
    \label{eq: no tight graph}
        |E(H')|=|E(H)|+2=2|\overline{V(H)}|+|V_0(H)|-l+2=2|\overline{V(H')}|+|V_0(H')|-l+1,
    \end{equation}
    contradicting the sparsity of $(G,\psi)$. We show that there is an admissible fix-1-reduction at $v$.
    
    First, suppose that $l=2$. By the sparsity of $(G,\psi)$, $u_1,u_2$ are free. Let $(G_1,\psi_1),(G_2,\psi_2)$ be obtained from $(G,\psi)$ by removing $v$ and adding the edge $e=(u_1,u_2)$ with gains $\text{id}$ and $\gamma$, respectively. Assume, by contradiction, that for $1\leq i\leq2$, $(G_i,\psi_i)$ has a blocker $H_i$. By Equation~\eqref{eq: no tight graph}, $H_1,H_2$ are balanced blockers. If $E(H_1\cap H_2)=\emptyset$, then 
    \begin{equation*}
    \begin{split}
        |E(H_1\cup H_2)|&=2|V(H_1)|-3+2|V(H_2)|-3=2|V(H_1\cup H_2)|+2|V(H_1\cap H_2)|-6\geq2|V(H_1\cup H_2)|-2,
    \end{split}
    \end{equation*}
    where the inequality holds because $u_1,u_2\in V(H_1\cap H_2)$. But then the graph $H':=H_1\cup H_2+v$ satisfies
    \begin{equation*}
        |E(H')|=2|V(H')|-2=2|\overline{V(H')}|+|V_0(H')|-2+|V_0(H')|\geq2|\overline{V(H')}|+|V_0(H')|-1,
    \end{equation*}
    where the inequality holds because $v\in V_0(H')$. This contradicts the sparsity of $(G,\psi)$. So $E(H_1\cap H_2)\neq\emptyset$. Since $H_1,H_2$ are balanced blockers, all paths from $u_1$ to $u_2$ in $H_1$ have gain $\textrm{id}$ and all paths from $u_1$ to $u_2$ in $H_2$ have gain $\gamma$. By the sparsity count, $H_1\cap H_2$ is connected (see, e.g., the proof of Lemma~\ref{lemma: what i need for 1-red.}(ii)). So, there is a path from $u_1$ to $u_2$ in $H_1\cap H_2$ with two different gains, a contradiction. Hence, at least one of $(G_1,\psi_1),(G_2,\psi_2)$ is $(2,1,3,2)$-gain tight.

    \medskip
    Now, let $l=1$. Let $(G_1,\psi_1)$ be obtained from $(G,\psi)$ by removing $v$ and adding the edge $e_1=(u_1,u_2)$ with gain $\text{id}$. Assume that $(G_1,\psi_1)$ has a blocker $H_1$. By Equation~\eqref{eq: no tight graph}, $H_1$ is a balanced blocker. Hence, $H_1+v$ satisfies $|E(H_1+v)|=2|V(H_1+v)|-3=(2|\overline{V(H_1+v)}|+|V_0(H_1+v)|-3)+|V_0(H_1+v)|.$
    If $H_1+v$ contains three fixed vertices, this contradicts the sparsity of $(G,\psi)$. Since $v$ is fixed, this implies that at most one of $u_1,u_2$ is fixed. Assume, without loss of generality, that $u_1$ is free.

    Let $(G_2,\psi_2)$ be obtained from $(G,\psi)$ by removing $v$ and adding a loop $e_2$ at $u_1$ with gain $\gamma$. Assume that $(G_2,\psi_2)$ has a blocker $H_2$. Since $H_2+e_2$ contains the unbalanced loop $e_2$, $H_2$ is a general-count blocker. Hence, by Equation~(\ref{eq: no tight graph}), $u_2\not\in V(H_2)$. If $E(H_1\cap H_2)\neq\emptyset$, then $|E(H_1\cap H_2)|\leq2|V(H_1\cap H_2)|-3$ and so $H_{12}:=H_1\cup H_2$ satisfies
    \begin{equation*}
    \begin{split}
        |E(H_{12})|       &\geq2|V(H_1)|-3+2|\overline{V(H_2)}|+|V_0(H_2)|-1-2|V(H_1\cap H_2)|+3\\
        &=2|\overline{V(H_{12})}|+|V_0(H_{12})|-1+(|V_0(H_1)|-|V_0(H_1\cap H_2)|)\geq2|\overline{V(H_{12})}|+|V_0(H_{12})|-1.
    \end{split}
    \end{equation*}
    By Equation~(\ref{eq: no tight graph}), this contradicts the sparsity of $(G,\psi)$. So, $E(H_1\cap H_2)=\emptyset$. Hence, 
    \begin{equation*}
    \begin{split}
        |E(H_{12})|&=2|V(H_1)|-3+2|\overline{V(H_2)}|+|V_0(H_2)|-1\\
        &=2|\overline{V(H_{12})}|+|V_0(H_{12})|-4+(|V_0(H_1)|+2|\overline{V(H_1\cap H_2)}|+|V_0(H_1\cap H_2)|)\\
        &\geq2|\overline{V(H_{12})}|+|V_0(H_{12})|-2+(|V_0(H_1)|+|V_0(H_1\cap H_2)|).
    \end{split}
    \end{equation*}
    where the inequality holds because $u_1\in\overline{V(H_1\cap H_2)}$. If $|V_0(H_1)|\geq1$, then $H_{12}$ is $(2,1,3,1)$-gain tight which, by Equation~(\ref{eq: no tight graph}), contradicts the sparsity of $(G,\psi)$. Hence, $V_0(H_1)=\emptyset$. In particular, $u_2$ is free.
    
    Let $(G_3,\psi_3)$ be obtained from $(G,\psi)$ by removing $v$ and adding a loop $e_3$ at $u_2$ with gain $\gamma$. Assume that $(G_3,\psi_3)$ has a blocker $H_3$. Similarly as we did with $H_2$, it is easy to see that $H_3$ is a general-count blocker, that $u_1\not\in V(H_3)$ and that $E(H_1\cap H_3)=\emptyset$. Moreover, $E(H_2\cap H_3)=\emptyset$, as otherwise $H_2\cup H_3$ is $(2,1,3,1)$-gain tight which, by Equation~(\ref{eq: no tight graph}), contradicts the sparsity of $(G,\psi)$. Let $\overline{S}=\sum_{1\leq i\neq j\leq3}|\overline{V(H_i\cap H_j)}|-|\overline{V(H_1\cap H_2\cap H_3)}|$ and $S_0=\sum_{1\leq i\neq j\leq3}|V_0(H_i\cap H_j)|-|V_0(H_1\cap H_2\cap H_3)|$. Since $u_1,u_2\not\in \overline{V(H_1\cap H_2\cap H_3)}$, we have $\overline{S}\geq2$. So the graph $H:=H_1\cup H_2\cup H_3$ satisfies
    \begin{equation*}
        \begin{split}
            |E(H)|&=2|V(H_1)|-3+2|\overline{V(H_2)}|+|V_0(H_2)|-1+2|\overline{V(H_3)}|+|V_0(H_3)|-1\\
            &=2|\overline{V(H)}|+|V_0(H)|-5+(|V_0(H_1)|+2\overline{S}+S_0)\geq2|\overline{V(H)}|+|V_0(H)|-1.
        \end{split}
    \end{equation*}
    By Equation~(\ref{eq: no tight graph}), this contradicts the sparsity of $(G,\psi)$. Hence, there is an admissible fix-1-reduction at $v$.
\end{proof}

\begin{theorem}
\label{theorem: final result for reflection}
Let $\Gamma$ be a cyclic group of order 2, let $(G,\psi)$ be a $\Gamma$-gain framework, $\tau:\Gamma\rightarrow\mathcal{C}_s$ be a faithful representation, and $p:V(G)\rightarrow\mathbb{R}^2$ be $\mathcal{C}_s$-generic. The following hold:
\begin{itemize}
    \item If $(G,\psi)$ is $(2,1,3,1)$-gain-tight, then the covering framework $(\Tilde{G},\tilde{p})$ is fully-symmetrically isostatic. 
    \item If $(G,\psi)$ is $(2,1,3,2)$-gain-tight, then the covering framework $(\Tilde{G},\tilde{p})$ is anti-symmetrically isostatic.
    \end{itemize}
\end{theorem}

\begin{proof}
We use a proof by induction on $|V(G)|$. 
First, assume that $V(G)$ has no free vertex.

If $(G,\psi)$ is $(2,1,3,1)$-gain-tight, then $G$ is a tree. The base case consists of exactly one single vertex and no edge, which is clearly fully-symmetrically isostatic. Assume that the statement is true for all graphs on $m$ vertices and let $G$ be a graph on $m+1$ vertices. Since $G$ is a tree, it has a vertex $v$ of degree 1. Thus, we may apply a fix-0-reduction at $v$ to obtain a $(2,1,3,1)$-gain tight graph $(G',\psi')$ on $m$ vertices. By the inductive hypothesis, all $\mathcal{C}_s$-generic realisations of $\tilde{G}'$ are fully-symmetrically isostatic. Choose a $\mathcal{C}_s$-generic realisation $(\tilde{G'},\tilde{q}')$ of $\tilde{G}'$. By Lemma~\ref{lemma: fix-0-extensions don't change the rank}, there is a $\mathcal{C}_s$-symmetric realisation $(\tilde{G},\tilde{q})$ of $\tilde{G}$ which is fully-symmetrically isostatic. By $\mathcal{C}_s$-genericity, $(\tilde{G},\tilde{p})$ is also fully-symmetrically isostatic.

If $(G,\psi)$ is $(2,1,3,2)$-gain tight, then $G$ consists of exactly two isolated vertices, with no edges, since any edge would violate the sparsity count. In this case, $(\tilde{G},\tilde{p})$ is clearly anti-symmetrically isostatic, since any anti-symmetric motion must be trivial.

Hence, we may assume $|\overline{V(G)}|\geq1$.
  All base graphs are given in Figure~\ref{Base graphs for reflection.}. It is easy to check that 
$\mathcal{C}_s$-symmetric realisations of these base graphs are fully-symmetrically and anti-symmetrically isostatic, respectively.

\begin{figure}[H]
    \centering
    \begin{tikzpicture}
  [scale=.8,auto=left]

  \draw(-2.25,1.8) -- (9.65,1.8);
  \node at (0.25,1.4) {Fully-symmetric};
  \node at (6.255,1.4) {Anti-symmetric};
  \draw(-2.25,1) -- (9.65,1);
  \draw(-2.25,-1) -- (-2.25,1.8);
  
  \draw[fill =  black](-1,0) circle (0.15cm);
  
  \draw(1.5,0) circle (0.15cm);
  \draw[->] (1.65,0) .. controls (2.05,0.5) and (0.95,0.5) .. (1.35,0);

  \draw (3,-1) -- (3,1.8);
  \draw (0.25,-1) -- (0.25,1);

  \draw(4.25,0) circle (0.15cm);

  \draw (5.5,-1) -- (5.5,1);

  \draw[fill = black](6.75,0) circle (0.15cm);
  \draw[fill = black](8.5,0) circle (0.15cm);

  \draw (9.65,-1) -- (9.65,1.8);
  \draw(-2.25,-1) -- (9.65,-1);
\end{tikzpicture}
    \caption{Base graphs for reflection.} 
    \label{Base graphs for reflection.}
\end{figure}

For the inductive step, assume the result holds whenever $|V(G)|=m$ for some $m\in\mathbb{N}$. Let $1\leq l\leq2$ and suppose $(G,\psi)$ is $(2,1,3,l)$-gain tight with $|V(G)|=m+1$. If $G$ has a fixed vertex $v$ of degree 1, then we may apply a fix-0-reduction at $v$ to obtain a $(2,1,3,l)$-gain tight graph $(G',\psi')$ on $m$ vertices. 
By induction, all $\mathcal{C}_s$-generic realisations of $\tilde{G'}$ are fully-symmetrically isostatic if $l=1$, and they are anti-symmetrically isostatic if $l=2$. Then, our result follows from Lemma~\ref{lemma: fix-0-extensions don't change the rank}. So, assume that all fixed vertices of $G$ have degree at least $2$. 

Suppose that $\overline{V(G)}=\{u\}$, and let $V_0(G)=\{v_1,\dots,v_t\}$ for some $t\geq1$. The average  degree of $G$, denoted $\hat{\rho}$, satisfies 
$\hat{\rho}=\frac{2|E(G)|}{|V(G)|}=\frac{4+2t-2l}{|V(G)|}$.
The  average degree  of $G$ is smallest when all vertices in $V_0(G)$ have degree 2, and so $2t+\text{deg}(u)\leq 4+2t-2l$. Hence $\text{deg}(u)\leq4-2l$. By Lemma~\ref{lemma: there is a vertex of degree 2 or 3}(i), $l=1$ and $\text{deg}(u)=2$. Then we may apply a 0-reduction at $u$ to obtain a $(2,1,3,1)$-gain tight graph $(G',\psi')$ on $m$ vertices. By induction, all $\mathcal{C}_s$-generic realisations of $\tilde{G}'$ are fully-symmetrically isostatic. Then the result holds by Lemma~\ref{lemma 0-extension maintain rank}. So, assume $|\overline{V(G)}|\geq2.$ 

By Lemma~\ref{there is an admissible reduction: reflection}, $(G,\psi)$ admits a reduction using one of the moves listed in the statement of the lemma. Let $(G',\psi')$ be a $(2,1,3,l)$-gain tight graph obtained by applying such a reduction to $(G,\psi)$. By induction, all $\mathcal{C}_s$-generic realisations of $\tilde{G'}$ are fully-symmetrically isostatic if $l=1$ and anti-symmetrically isostatic if $l=2$. Let $\tilde{q}'$ be a $\mathcal{C}_s$-generic configuration of $\tilde{G'}$ which also satisfies the conditions of Lemma~\ref{lemma: 1-extension maintain rank} (respectively, Lemma~\ref{lemma fix-1-extension maintain rank}) if $\tilde{G'}$ is obtained from $\tilde{G}$ by applying a 1-reduction (respectively, a fix-1-reduction). Such a configuration exists: if necessary, we may apply a small symmetry-preserving perturbation to  the points of a $\mathcal{C}_s$-generic framework, 
which will maintain $\mathcal{C}_s$-genericity. By Lemmas~\ref{lemma 0-extension maintain rank},~\ref{lemma fix-1-extension maintain rank},~\ref{lemma: loop-1-extensions don't change the rank} and~\ref{lemma: 1-extension maintain rank}, there is a realisation $(\tilde{G},\tilde{q})$ of $\tilde{G}$ which is fully-symmetrically isostatic if $l=1$ and anti-symmetrically isostatic if $l=2$. Since $\tilde{p}$ is $\mathcal{C}_s$-generic, the result follows.
\end{proof} 

The following  main result for $\mathcal{C}_s$ is now a consequence of Proposition~\ref{Necessary conditions for reflection} and Theorem~\ref{theorem: final result for reflection}.

\begin{theorem}
Let $(\tilde{G},\tilde{p})$ be a $\mathcal{C}_s$-generic framework with $\mathcal{C}_s$-gain framework $(G,\psi,p)$. $(\tilde{G},\tilde{p})$ is infinitesimally rigid if and only if the following hold:
\begin{itemize}
\item $(G,\psi)$ has a $(2,1,3,1)$-gain tight spanning subgraph.
\item $(G,\psi)$ has a $(2,1,3,2)$-gain tight spanning subgraph.
\end{itemize}
\end{theorem}

\subsection{Half-turn group}
Let $(\tilde{G},\tilde{p})$ be a $\mathcal{C}_2$-generic framework. Recall that $(G,\psi)$ is $(2,0,3,1)$-gain tight whenever $(\tilde{G},\tilde{p})$ is fully-symmetrically isostatic, and  $(G,\psi)$ is $(2,2,3,2)$-gain tight whenever $(\tilde{G},\tilde{p})$ is anti-symmetrically isostatic (see Proposition~\ref{necessary conditions for 2-fold rotation} in Section~\ref{sec:nec(2)}).
In this section, we show that the converse statements are also true. 

We do so by strong induction on $|\overline{V(G)}|$, using the vertex reduction moves shown in Section~\ref{sec:ext}. Hence, we first need to show that there is an admissible reduction of $(G,\psi)$. Let $v\in V(G)$ be a free vertex of degree 3. By Theorem~\ref{theorem on 1-reductions: (2,m,3,l)-tight}, there is always an admissible 1-reduction at $v$, unless $(G,\psi)$ is $(2,2,3,2)$-gain tight, $v$ has exactly one free neighbour and exactly one fixed neighbour. In the following Lemma, we take care of this remaining case.

\begin{lemma}
\label{lemma: there is an admissible reduction (rotation order 2)}
    Let $(G,\psi)$ be a $(2,2,3,2)$-gain tight graph with $|V_0(G)|\leq1$ and $|V(G)|\geq2$. Then $(G,\psi)$ admits a reduction.
\end{lemma}
\begin{proof}
The case where there is no fixed vertex is already known (see e.g., Theorem 6.8 in \cite{bt2015}). Hence, we may assume $V_0(G)=\{v_0\}.$ By Lemma~\ref{lemma: there is a vertex of degree 2 or 3}, there is a free vertex in $V(G)$ of degree 2 or 3. By the sparsity of $(G,\psi)$, no vertex of $G$ has a loop. We may assume that $G$ has no free vertex of degree 2. Otherwise, we may apply a 0-reduction to $(G,\psi)$ (clearly, any 0-reduction is admissible). Further, we may assume that all free vertices of degree 3 have exactly 2 distinct neighbours, one of which is $v_0$: otherwise, we may apply a 1-reduction to $(G,\psi)$, by Theorem~\ref{theorem on 1-reductions: (2,m,3,l)-tight}.

So let $v_1,\dots,v_t$ be the free vertices in $G$ of degree 3. For $1\leq i\leq t$ let $u_i$ be the free neighbour of $v_i$, and $e_i:=(u_i,v_0)$. By Lemma~\ref{lemma: there is a vertex of degree 2 or 3}(ii), $\textrm{deg}(v_0)\leq t$. So, if the edge $e_i$ is present for some $1\leq i\leq t,$ then $u_i$ must be a vertex of degree 3. Hence, we can apply a 2-vertex reduction at $u_i,v_i$. So, we may assume that $e_i\not\in E(G)$ for all $1\leq i\leq t$.

For $1\leq i\leq t,$ let $(G_i,\psi_i)$ be obtained from $(G,\psi)$ by removing $v_i$ and adding $e_i$ with gain $\textrm{id}$. We will show that, for some $1\leq i\leq t$, $(G_i,\psi_i)$ is an admissible 1-reduction. Assume, by contradiction, that for all $1\leq i\leq t$ there is a blocker $H_i$ for $(G_i,\psi_i)$. By Proposition~\ref{lemma: there is no tight subgraph containing all three neighbours}, each $H_i$ is a balanced blocker. 

Moreover, for each $1\leq i\neq j\leq t$, $v_j\not\in V(H_i)$. To see this, suppose, by contradiction, that $v_j\in V(H_i)$. Since $H_i$ is a balanced blocker, all of its vertices have degree at least $2$ (see the first paragraph of the proof of Lemma~\ref{lemma: there is a vertex of degree 2 or 3} for an argument). Hence, two of the edges incident to $v_j$ lie in $E(H_i)$. Moreover, since $H_i$ is balanced, it cannot contain parallel edges. Hence, $H_i$ contains exactly $2$ of the edges incident to $v_j$. Let $e$ be the edge incident to $v_j$ such that $e\not\in E(H_i)$. Then
\begin{equation*}
    |E(H_i+v_i+e)|=|E(H_i)|+4=2|V(H_i)|+1=2|V(H_i+v_i+e)|-1,
\end{equation*}
contradicting the sparsity of $(G,\psi)$. So $v_j\not\in V(H_i)$ for all $1\leq i\neq j\leq t$.

\medskip
\textit{\textbf{Claim: }$E(H_i\cap H_j)=\emptyset$ and $V(H_i\cap H_j)=\{v_0\}$ for all $1\leq i\neq j\leq t$}.

\medskip

\textit{Proof. }Choose some $1\leq i\neq j\leq t$. First, assume by contradiction that $E(H_i\cap H_j)\neq\emptyset$. In a similar way as we did in the proof of Lemma~\ref{lemma: what i need for 1-red.}(ii), we can see that $|E(H_i\cup H_j)|\geq2|V(H_i\cup H_j)|+3c-c_0-6$, where $c,c_0$ are, respectively, the number of connected components and isolated vertices of $H_i\cap H_j$. Notice that $c_0\leq c-1$ (since all isolated vertices of $H'$ are also connected components of $H'$, and $H'$ has at least one connected component with non-empty edge set), and so $|E(H_i\cup H_j)|\geq2|V(H_i\cup H_j)|+2c-5$. By the sparsity of $(G,\psi)$, $c=1$ and $|E(H_i\cup H_j)|=2|V(H_i\cup H_j)|-3$. But the graph $H$ obtained from $H_i\cup H_j$ by adding $v_i,v_j$ and its incident edges satisfies $|E(H)|=2|V(H)|-1$, contradicting the sparsity of $(G,\psi).$ Thus, $E(H_i\cap H_j)=\emptyset$ for all $1\leq i\neq j\leq t$. 

Now, if $V(H_i\cap H_j)\neq\{v_0\}$, then $|E(H_i\cup H_j)|=|E(H_i)|+|E(H_j)|=2|V(H_i\cup H_j)|+2|V(H_i\cap H_j)|-6\geq2|V(H_i\cup H_j)|-2.$ But then the graph $H$ obtained from $H_i\cup H_j$ by adding $v_i,v_j$ and its incident edges satisfies $|E(H)|=2|V(H)|,$ contradicting the sparsity of $(G,\psi)$. So $V(H_i\cap H_j)=\{v_0\}$. Since $i,j$ were arbitrary, the claim holds. $\square$

\medskip
Let $H:=\bigcup_{i=1}^tH_i$. By the Claim,
\begin{equation*}
        |E(H)|=\sum_{i=1}^t|E(H_i)|
        =2\sum_{i=1}^t|V(H_i)|-3t
        =2(|V(H)|+(t-1))-3t
        =2|V(H)|-t-2.
\end{equation*}
So for the graph $G'$ obtained from $H$ by adding the vertices $v_i$, $i=1,\ldots, t$, and their incident edges, we have $|E(G')|=2|V(G')|-2.$ This implies that there is no edge $e\in E(G)\setminus E(H)$ that joins two vertices in $V(H)$.

Next we show that there is no non-empty subgraph $H'$ of $G$ such that $V(G)$ is the disjoint union of $V(G')$ and $V(H')$. Assume, by contradiction, that such $H'$ exists. By assumption, all vertices of $H'$ have degree at least 4 in $G.$ Let $d(G',H')$ be the number of edges joining a vertex in $G'$ with one in $H'$.

We know $|E(H')|=2|V(H')|-\alpha$ for some $\alpha\geq2$. We have that $4|V(H')|\leq \sum_{v\in V(H')} deg_G(v)=2|E(H')|+d(G',H')=4|V(H')|-2\alpha+d(G',H'),$ and so $d(G',H')\geq2\alpha.$
Hence, 
\begin{equation*}
        |E(G)|=|E(G')|+|E(H')|+d(G',H')\geq2|V(G')|-2+2|V(H')|-\alpha +2\alpha=2|V(G)|-2+\alpha,
\end{equation*}
which contradicts the sparsity of $(G,\psi)$, since $\alpha\geq2$. So, $H'$ does not exist, and $G=G'$.

\medskip
Finally, fix some $1\leq i\leq t$ and let $n,m$ be the vertices of $H$ which have degree 2 and 3 in $H_i$. The average degree of $H_i$ is $\hat{\rho}=\frac{2|E(G)|}{|V(H_i)|}=\frac{4|V(H_i)|-6}{|V(H_i)|}$. The minimum average degree of $H_i$ is $\frac{4|V(H_i)|-2n-m}{|V(H_i)|}.$ Hence, $2n+m\geq6.$ In particular, there are at least 3 vertices of degree 2 or 3 in $V(H_i)$, and so there is a free vertex $v$ of degree 2 or 3 that is not $v_0$ or $u_i$. This means that $v$ has degree 2 or 3 in $G=G'$. But this is not possible, since we assumed there are no free vertices of degree 2 in $G$, and that all free vertices of degree 3 are $v_1,\dots,v_t$. The result follows.
\end{proof}

The following results will be proved in a very similar way to Theorem~\ref{theorem: final result for reflection}. However, we now work with the half-turn group. So $|V_0(G)|\leq 1$.

\begin{theorem}
\label{theorem: final theorem for 2-fold symmetry}
    Let $\Gamma$ be a cyclic group of order 2, let $(G,\psi)$ be a connected $\Gamma$-gain framework with $|V_0(G)|\leq 1,$ $\tau:\Gamma\rightarrow\mathcal{C}_2$ be a faithful representation, and $p:V(G)\rightarrow\mathbb{R}^2$ be $\mathcal{C}_2$-generic. The following hold:
    \begin{itemize}
    \item If $(G,\psi)$ is $(2,0,3,1)$-gain-tight, then the covering framework    $(\Tilde{G},\tilde{p})$  is fully-symmetrically isostatic.
    \item  If $(G,\psi)$ is $(2,2,3,2)$-gain tight, then  the covering framework $(\Tilde{G},\tilde{p})$ is anti-symmetrically isostatic.
    \end{itemize}
\end{theorem}

\begin{proof}
First, notice that if there is no free vertex, then $\tilde{G}$ is a single fixed vertex. In this case $\tilde G$ is not $(2,0,3,1)$-gain-tight. It is $(2,2,3,2)$-gain-tight and clearly also anti-symmetrically isostatic.

Hence, we may assume $|\overline{V(G)}|\geq1$. We prove the result by 
induction on $|\overline{V(G)}|$. Assume $|\overline{V(G)}|=1$. If $(G,\psi)$ is $(2,0,3,1)$-gain tight, $G$ is either composed of a free vertex and a loop, or a free vertex, a fixed vertex, and an edge connecting them. In either case, $O_0(G,\psi,p)$ is a non-zero row with one-dimensional kernel, and so $(\tilde{G},\tilde{p})$ is fully-symmetrically isostatic. If $(G,\psi)$ is $(2,2,3,2)$-gain tight, $G$ must be a single free vertex. Any anti-symmetric motion of any realisation $(\tilde{G},\tilde{p})$ of $\tilde{G}$ must be a translation of the whole framework, and so $(\tilde{G},\tilde{p})$ is anti-symmetrically isostatic. The base cases for the fully-symmetric and anti-symmetric case are given in Figure~\ref{Base graphs for 2-fold rotation}.

\begin{figure}[H]
    \centering
    \begin{tikzpicture}
  [scale=.8,auto=left]

  \draw(-1,1.8) -- (9,1.8);
  \node at (2.15,1.4) {Fully-symmetric};
  \node at (7,1.4) {Anti-symmetric};
  \draw(-1,1) -- (9,1);
  \draw(-1,-1) -- (-1,1.8);
  
  \draw(0,0) circle (0.15cm);
  \draw[->] (0.15,0) .. controls (0.55,0.5) and (-0.55,0.5) .. (-0.15,0);

  \draw (1,-1) -- (1,1);

  \draw(2,0) circle (0.15cm);
  \draw[fill = black](4,0) circle (0.15cm);
  \draw[->] (2.15,0) -- (3.85,0);

  \draw (5,-1) -- (5,1.8);

  \draw(6,0) circle (0.15cm);

  \draw (7,-1) -- (7,1);

  \draw[fill = black](8,0) circle (0.15cm);

  \draw (9,-1) -- (9,1.8);
  \draw(-1,-1) -- (9,-1);
\end{tikzpicture}
    \caption{Base graphs for 2-fold rotation.}
    \label{Base graphs for 2-fold rotation}
\end{figure}
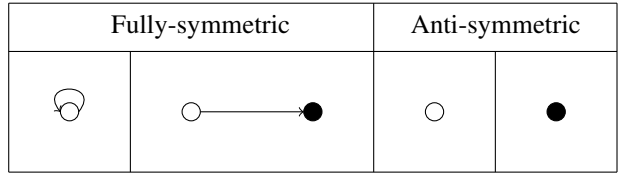

Assume the result holds whenever $|\overline{V(G)}|\leq m$ for some $m\in\mathbb{N}$ and consider the case where $|\overline{V(G)}|=m+1$. 

If $(G,\psi)$ is $(2,0,3,1)$-gain tight, $G$ has a free vertex $v$ of degree $2$ or $3$, by Lemma~\ref{lemma: there is a vertex of degree 2 or 3}. If $v$ has degree $2$, or if it has degree $3$ with a loop, then we may apply a 0-reduction or loop-1-reduction at $v$ to obtain a $(2,0,3,1)$-gain tight graph $(G',\psi')$, since 0-reductions and loop-1-reductions are always admissible. Moreover, if $v$ has degree 3 with a loop, then it is not incident to a fixed vertex, by the sparsity of $(G,\psi)$. By the inductive hypothesis, all $\mathcal{C}_2$-generic realisations of $\tilde{G'}$ are fully-symmetrically isostatic. Then, our result follows from Lemmas~\ref{lemma 0-extension maintain rank} and~\ref{lemma: loop-1-extensions don't change the rank}.

So, assume that $v$ has degree $3$ and no loops. By Lemma~\ref{theorem on 1-reductions: (2,m,3,l)-tight}, there is a $(2,0,3,1)$-gain tight graph $(G',\psi')$ obtained from $(G,\psi)$ by applying a $1$-reduction at $v$. By induction, all $\mathcal{C}_2$-generic realisations of $\tilde{G'}$ are fully-symmetrically isostatic, so take a $\mathcal{C}_2$-generic realisation $(\tilde{G'},\tilde{q}')$ of $\tilde{G'}$ such that the conditions in Lemma~\ref{lemma: 1-extension maintain rank} are satisfied. Then, our result holds by Lemma~\ref{lemma: 1-extension maintain rank}.

If $(G,\psi)$ is $(2,2,3,2)$-tight, then, by Lemma~\ref{lemma: there is an admissible reduction (rotation order 2)}, there is a $(2,2,3,2)$-gain tight graph $(G',\psi')$ on at most $m$ vertices (exactly $m$ if we apply a 0-reduction, loop-1-reduction, or 1-reduction, and exactly $m-1$ if we apply a 2-vertex reduction) obtained by applying a reduction to $(G,\psi)$. 

By the inductive hypothesis, all $\mathcal{C}_2$-generic realisations of $\tilde{G'}$ are anti-symmetrically isostatic. Let $\tilde{q}'$ be a $\mathcal{C}_2$-generic configuration of $\tilde{G'}$, which also satisfies the conditions of Lemma~\ref{lemma: 1-extension maintain rank} if $\tilde{G'}$ is obtained from $\tilde{G}$ by applying a 1-reduction. By Lemmas~\ref{lemma 0-extension maintain rank},~\ref{lemma: 1-extension maintain rank} and~\ref{2-vertex extension maintains rank}, our result holds.
\end{proof}

From Proposition~\ref{necessary conditions for 2-fold rotation} and Theorem~\ref{theorem: final theorem for 2-fold symmetry}, we obtain the following main result for $\mathcal{C}_2$.

\begin{theorem}
\label{theorem: final for 2-fold rotation}
Let $(\tilde{G},\tilde{p})$ be a  $\mathcal{C}_2$-generic framework with associated $\mathcal{C}_2$-gain framework $(G,\psi,p)$. $(\tilde{G},\tilde{p})$ is infinitesimally rigid if and only if:
\begin{itemize}
\item there is a spanning subgraph of $(G,\psi)$ which is $(2,0,3,1)$-gain tight; and
\item there is a spanning subgraph of $(G,\psi)$ which is $(2,2,3,2)$-gain tight.
\end{itemize}
\end{theorem}

\subsection{Rotation group of order 3}

Let $k\geq3$, and $(\tilde{G},\tilde{p})$ be a $\mathcal{C}_k$-generic framework. Recall that $(G,\psi)$ is $(2,0,3,1)$-gain tight whenever $(\tilde{G},\tilde{p})$ is fully-symmetrically isostatic, and $(G,\psi)$ is $(2,1,3,1)$-gain tight whenever $(\tilde{G},\tilde{p})$ is $\rho_1$-symmetrically isostatic and $\rho_{k-1}$-symmetrically isostatic. Here, we prove that the converse is also true, which will give us the desired characterisation for $\mathcal{C}_3$-generic frameworks.

\begin{theorem}
\label{theorem: final theorem for k-fold symmetry pt.1}
     Let $\Gamma$ be a cyclic group of order $k\geq3$, $(G,\psi)$ be a connected $\Gamma-$gain framework with $|V_0(G)|\leq 1$, $\tau:\Gamma\rightarrow\mathcal{C}_k$ be a faithful representation and $p:V(G)\rightarrow\mathbb{R}^2$ be $\mathcal{C}_k-$generic. The following hold:
    \begin{itemize}
        \item  If $(G,\psi)$ is $(2,0,3,1)-$gain tight, then the covering framework $(\Tilde{G},\tilde{p})$ is fully-symmetrically isostatic.
        \item  If $(G,\psi)$ is $(2,1,3,1)$-gain tight, then the covering framework $(\Tilde{G},\tilde{p})$ is $\rho_1$-symmetrically isostatic and $\rho_{k-1}-$symmetrically isostatic.
    \end{itemize} 
\end{theorem}

\begin{proof}
We prove the result by induction on $|\overline{V(G)}|,$ with the base cases given in Figure~\ref{Base graphs for 3-fold rotation}. It is easy to check that, in the first two examples of Figure~\ref{Base graphs for 3-fold rotation}, the $\rho_0-$orbit matrix has full rank and nullity 1, whereas, in the latter two cases, the $\rho_1-$orbit matrix and the $\rho_{k-1}-$orbit matrix have full rank and nullity 1. The base cases for the $\rho_0$-symmetric, $\rho_1$-symmetric and $\rho_{k-1}$-symmetric case are given in Figure~\ref{Base graphs for 3-fold rotation}.

For the inductive step, assume the result holds when $|\overline{V(G)}|=t$ for some $t\geq1$, and let $(G,\psi)$ be a $(2,m,3,1)-$gain tight graph with $|\overline{V(G)}|=t+1$, for some $0\leq t\leq1$. Suppose that $m=0$, and that $V(G)$ has an isolated fixed vertex. Then, we may remove it to obtain a $(2,0,3,1)-$gain tight graph $(G',\psi')$ on $t$ vertices. By the inductive hypothesis, all $\mathcal{C}_k$-generic realisations of $\tilde{G'}$ are fully-symmetrically isostatic. 
Let $\tilde{q}'$ be a $\mathcal{C}_k$-generic configuration of $\tilde{G'}$. For any extension $\tilde{q}:V(G)\rightarrow\mathbb{R}^2$ of $\tilde{q}'$, we have $O_0(G,\psi,q)=O_0(G',\psi',q')$. So, $(\tilde{G},\tilde{p})$ is also fully-symmetrically isostatic. By $\mathcal{C}_k$-genericity of $(\tilde{G},\tilde{p})$, the result follows. So, we may assume that each fixed vertex of $(G,\psi)$ has degree at least 1.

By Lemma~\ref{lemma: there is a vertex of degree 2 or 3}, $G$ has  a free vertex $v$ of degree 2 or 3 (both when $m=0$ and when $m=1$). If $v$ has degree $2$, or if it has degree $3$ with a loop, then we may apply a 0-reduction or loop-1-reduction at $v$ to obtain a $(2,m,3,1)-$tight graph $(G',\psi')$ on $t$ vertices. By the inductive hypothesis, all $\mathcal{C}_k$-generic realisations of $\tilde{G}$ are fully-symmetrically isostatic when $m=0$, and $\rho_1-$symmetrically isostatic, $\rho_{k-1}-$symmetrically isostatic when $m=1$. Moreover, when $m=0$, the vertex incident to $v$ is free, by the sparsity of $(G,\psi)$. Then, by Lemmas~\ref{lemma 0-extension maintain rank} and~\ref{lemma: loop-1-extensions don't change the rank}, the result holds. So, assume that $v$ has degree $3$ and no loop. By Theorem~\ref{theorem on 1-reductions: (2,m,3,l)-tight}, there is $(2,m,3,1)-$tight graph $(G',\psi')$ obtained from $(G,\psi)$ by applying a $1-$reduction at $v$. 

By the inductive hypothesis, all $\mathcal{C}_k$-generic realisations of $\tilde{G'}$ are fully-symmetrically isostatic when $m=0$, and anti-symmetrically isostatic when $m=1$. Let $(\tilde{G},\tilde{q}')$ be any $\mathcal{C}_k$-generic realisation of $\tilde{G'}$ which satisfies the conditions of Lemma~\ref{lemma: 1-extension maintain rank}.
Then, our result holds by 
Lemma~\ref{lemma: 1-extension maintain rank}. 
\end{proof}

    \begin{figure}[htp]
    \centering
    \begin{tikzpicture}
  [scale=.8,auto=left]

  \draw(-1,1.8) -- (9,1.8);
  \node at (2.15,1.4) {Fully-symmetric};
  \node at (7,1.4) {$\rho_0,\rho_{k-1}$-symmetric};
  \draw(-1,1) -- (9,1);
  \draw(-1,-1) -- (-1,1.8);
  
  \draw(0,0) circle (0.15cm);
  \draw[->] (0.15,0) .. controls (0.55,0.5) and (-0.55,0.5) .. (-0.15,0);

  \draw (1,-1) -- (1,1);

  \draw(2,0) circle (0.15cm);
  \draw[fill = black](4,0) circle (0.15cm);
  \draw[->] (2.15,0) -- (3.85,0);

  \draw (5,-1) -- (5,1.8);

  \draw(6,0) circle (0.15cm); 
  \draw[->] (6.15,0) .. controls (6.55,0.5) and (5.45,0.5) .. (5.85,0);

  \draw (7,-1) -- (7,1);

  \draw[fill = black](8,0) circle (0.15cm);

  \draw (9,-1) -- (9,1.8);
  \draw(-1,-1) -- (9,-1);
\end{tikzpicture}
    \caption{Base graphs for 3-fold rotation (and $k$-fold rotation for $\rho_0,\rho_1$ and $\rho_{k-1}$).}
    \label{Base graphs for 3-fold rotation}
\end{figure} 

We finally have our main combinatorial characterisation for $\mathcal{C}_3$, which is a direct result of Theorem~\ref{theorem: final theorem for k-fold symmetry pt.1}.

\begin{theorem}
\label{final theorem for rotation}
Let $(\tilde{G},\tilde{p})$ be a $\mathcal{C}_3-$generic framework with $\Gamma-$gain framework $(G,\psi,p)$ (here, $\Gamma$ is a cyclic group of order 3) and $|V_0(G)|\leq1$. $(\tilde{G},\tilde{p})$ is infinitesimally rigid if and only if:
\begin{itemize}
\item there is a spanning subgraph $(H,\psi|_{E(H)})$ of $(G,\psi)$ which is $(2,0,3,1)-$gain tight; and
\item there is a spanning subgraph $(H,\psi|_{E(H)})$ of $(G,\psi)$ which is $(2,1,3,1)-$gain tight.
\end{itemize}
\end{theorem}

Note that for $|\Gamma|>3$, there are irreducible representations of $\Gamma$ that lead to additional necessary sparsity counts for $\rho_j$-symmetric isostaticity. Hence we will discuss these groups in a separate paper \cite{lasch}.

\section{Further work}\label{sec:fut}

In this paper, we have given necessary conditions for reflection and rotationally symmetric frameworks to be infinitesimally rigid in the plane, regardless of whether the action of the group is free on the vertices or not. Moreover, for the groups of order 2 and 3, we have shown that these conditions are also \emph{sufficient} for symmetry-generic infinitesimal rigidity. 
In a second paper \cite{lasch}, we establish the analogous combinatorial characterisations of symmetry-generic infinitesimally rigid plane frameworks for the cyclic groups of odd order up to 1000 and of order 4 and 6. The proofs for these groups follow the same pattern as the ones given in this paper, but become  more complex due to an even more refined sparsity count.

A natural direction for future research is the completion of the study of symmetric plane frameworks by considering  cyclic groups of odd order at least 1000, even order cyclic groups of order at least 8, and dihedral groups. It is conjectured that the proofs of this paper extend to all odd order cyclic groups. The only issue here is to deal with the growing number of base graphs. In our second paper we will show that for cyclic groups of even order at least 8, the necessary sparsity conditions are no longer sufficient. Thus, it seems very challenging to try to characterise symmetry-generic infinitesimal rigidity for those groups. This leaves the case of the dihedral groups.


In \cite{jzt2016}, there is a combinatorial characterisation of forced-symmetric infinitesimal rigidity for frameworks that are generic with respect to a dihedral group $D_{2k}$, where $k$ is odd, and the group acts freely on the vertex set. However, no such characterisation has been found for the case when $k$ is even, despite significant efforts. (See Section 7.2.4 in \cite{jzt2016} and Section 4.4.3 in \cite{bt2015} for examples of graphs that satisfy the desired combinatorial counts described in \cite{jzt2016}, but are fully-symmetrically flexible.) So even in the free action case, for even $k$, a characterisation for $D_{2k}$-generic infinitesimal rigidity also does not exist.

For odd $k$, the key obstacle in obtaining a characterisation for $D_{2k}$-generic infinitesimal rigidity (in either the free or non-free action case) is that these groups are non-abelian and hence have irreducible representations of dimension greater than 1. For such representations, it is unclear how to define the corresponding gain graph or orbit rigidity matrix. However, we expect that the methods of this paper extend to characterising fully-symmetric  infinitesimal rigidity for $D_{2k}$, where $k$ is odd and the action is not free on the vertices. Similarly, we expect to be able to deal with  $\rho_i$-symmetric infinitesimal rigidity, where $\rho_i$ is $1$-dimensional, in either the free or non-free action case. This is all work in progress \cite{PhDThesis}.

It is a famous open problem to find a combinatorial characterisation of infinitesimally rigid generic bar-joint frameworks (without symmetry) in dimension at least 3. Hence, we can also not yet combinatorially characterise infinitesimal rigidity for symmetry-generic bar-joint frameworks in $\mathbb{R}^d$ for $d\geq3$. 
However, there are classes of graphs which are known to be generically rigid in $\mathbb{R}^3$. These include frameworks obtained  by a recursive constructions, starting with a simplex and applying a series of Henneberg 0- and  1-extensions, or triangulated simplicial polytopes in 3-space. Since Lemma~\ref{Lemma: rigidity matrix block diagonalises} holds in all dimensions, and our construction of orbit matrices generalises to higher dimensions, these classes of frameworks are amenable to  our approach, a least when the group is abelian. 

We note that initial higher-dimensional  results  have very recently been obtained (via a different approach) for the special class of $d$-pseuodmanifolds in $(d+1)$-space with a free $\mathbb{Z}_2$-action \cite{crjata}. There are also results for special classes of symmetric body-bar and body-hinge frameworks in $d$-space; see \cite{schtan_bodies}.

\bibliography{monster}{}
\bibliographystyle{plain}
\end{document}